\documentclass[10pt]{article}
\usepackage[all,2cell]{xy} \UseAllTwocells \SilentMatrices
\usepackage{latexsym,amsfonts,amssymb}
\usepackage{amsmath,amsthm,amscd}
\usepackage{hyperref,psfrag}
\usepackage{color,xypic}

\usepackage[dvips]{epsfig}
\usepackage{psfrag}

\usepackage{graphicx}

\usepackage{a4wide}

\normalfont\upshape

\usepackage{tikz}
\usetikzlibrary{decorations.markings}
\usetikzlibrary{decorations.pathreplacing}
\usepackage[outline]{contour}
\contourlength{1.2pt}

\tikzset{->-/.style={decoration={
  markings,
  mark=at position #1 with {\arrow{>}}},postaction={decorate}}}

\tikzset{middlearrow/.style={
        decoration={markings,
            mark= at position 0.5 with {\arrow{#1}} ,
        },
        postaction={decorate}
    }
}
\newcommand{\bbullet}{
\begin{tikzpicture}
  \draw[fill=black] circle (0.55ex);
\end{tikzpicture}
}

\newcommand{\hackcenter}[1]{
 \xy (0,0)*{#1}; \endxy}

\newcommand{\tte}{\mathtt{e}}

\newcommand{\ddxi}{\frac{\partial}{\partial x_i}}
\newcommand{\ddxj}{\frac{\partial}{\partial x_j}}
\newcommand{\ddxtwo}{\frac{\partial}{\partial x_2}}

\newcommand{\lb}{\lbrace}
\newcommand{\rb}{\rbrace}
\newcommand{\undvarepsilon}{\underline{\varepsilon}}
\newcommand{\undeta}{\underline{\eta}}
\newcommand{\Par}{\text{Par}}

\newcommand{\undx}{\underline{x}}
\newcommand{\undy}{\underline{y}}
\newcommand{\unda}{\underline{a}}
\newcommand{\undb}{\underline{b}}

\usepackage{fancyheadings}
\theoremstyle{definition}
\newtheorem{thm}{Theorem}[section]
\newtheorem{cor}[thm]{Corollary}

\newtheorem{lem}[thm]{Lemma}
\newtheorem{rem}[thm]{Remark}
\newtheorem{prop}[thm]{Proposition}
\newtheorem{defn}[thm]{Definition}
\newtheorem{example}[thm]{Example}

\numberwithin{equation}{section}

\usepackage{bbm}
\def\C{{\mathbbm C}}
\def\N{{\mathbbm N}}

\def\Z{{\mathbbm Z}}

\def\F{{\mathbbm F}}

\newcommand{\Hom}{{\rm Hom}}
\newcommand{\HOM}{{\rm HOM}}
\renewcommand{\to}{\rightarrow}

\newcommand{\op}{{\rm op}}

\newcommand{\Res}{{\rm Res}}

\newcommand{\END}{{\rm END}}

\newcommand{\oh}{\mathrm{OH}} 
\newcommand{\oht}{\widetilde{\mathrm{OH}}} 
\newcommand{\gr}{\mathrm{Gr}} 

\def\sltwo{\mathfrak{sl}_2}
\def\gloneone{\mathfrak{gl}_{1|1}}

\def\Res{{\mathrm{Res}}}

\def\lra{{\longrightarrow}}
\def\dmod{{\mathrm{-dmod}}}   

\def\grank{{\mathrm{grank}}}

\def\mc{\mathcal}
\def\mf{\mathfrak}
\def\shuffle{\,\raise 1pt\hbox{$\scriptscriptstyle\cup{\mskip
               -4mu}\cup$}\,}

\newcommand{\refequal}[1]{\xy {\ar@{=}^{#1}
(-1,0)*{};(1,0)*{}};
\endxy}

\def\Cc{\mathcal{C}}

\def\1bb{\mathbb{1}}

\def\1{\mathbf{1}}

\def\bigind{\mc{I}}
\def\smind{\mathrm{Ind}}
\def\bigres{\mc{R}}
\def\smres{\mathrm{Res}}

\usepackage[centertableaux]{ytableau}

\title{The differential graded odd nilHecke algebra}
\author{Alexander P. Ellis, You Qi}

\date{\today}

\newcommand{\sym}{\mathrm{\Lambda}}
\newcommand{\osym}{\mathrm{O\Lambda}}

\newcommand{\opol}{\mathrm{OPol}}

\newcommand{\onh}{\mathrm{ONH}}

\newcommand{\xt}{\widetilde{x}}
\newcommand{\et}{\widetilde{e}}
\newcommand{\htil}{\widetilde{h}}
\newcommand{\st}{\widetilde{s}}
\newcommand{\sh}{\widehat{s}}
\newcommand{\sth}{\widehat{\widetilde{s}}}

\newcommand{\pd}{\partial}

\newcommand{\osymt}{\widetilde{\osym}}

\newcommand{\zerobar}{\bar{0}}
\newcommand{\onebar}{\bar{1}}
\newcommand{\onez}{\mathbf{z}}
\newcommand{\onealpha}{\mathbf{1}_{\alpha}}

\newcommand{\ct}{\mathrm{ct}}
\newcommand{\ctbar}{\overline{\mathrm{ct}}}
\newcommand{\boxi}{{\tiny\ytableaushort{i}}}

\newcommand{\mH}{\mathrm{H}} 
\newcommand{\RHOM}{\mathbf{R}\mathrm{HOM}}

\newcommand{\oL}{\otimes^{\mathbf{L}}}

\begin{document}
%

\maketitle
\begin{abstract}
We equip the odd nilHecke algebra and its associated thick calculus category with digrammatically local differentials.  The resulting differential graded Grothendieck groups are isomorphic to two different forms of the positive part of quantum $\sltwo$ at a fourth root of unity.
\end{abstract}

\setcounter{tocdepth}{2} \tableofcontents

\section{Introduction}\label{sec-intro}

\subsection{Topological motivation}

Suitably normalized, the Alexander and Jones polynomials of a link $L$ coincide at a fourth root of unity:
\begin{equation}\label{eqn-equal-Jones-Alex-numbers}
V_L(q)\vert_{q=\sqrt{-1}}=\Delta_K(t)\vert_{t=\sqrt{-1}}.
\end{equation}
Both these invariants are values of special cases of the Witten-Reshetikhin-Turaev 3D TQFT \cite{Witten, RT}.  The input to the WRT construction is a modular tensor category; in the cases of the Alexander and Jones polynomials, these categories are quotients of the representation categories for the quantized enveloping algebras $u_t(\gloneone)$ and $u_q(\sltwo)$, respectively, with the quantum parameters specialized to roots of unity.

The equality \eqref{eqn-equal-Jones-Alex-numbers} admits a representation theoretic explanation.  If $q^4=1$, then the upper nilpotent parts of these two quantized enveloping algebras are isomorphic as algebras.  So it is not surprising that the Alexander polynomial can be constructed as a quantum link invariant for both $u_t(\gloneone)$ \cite{RozanskySaleur} and $u_{\sqrt{-1}}(\sltwo)$ \cite{Murakami1,Murakami2}.  A good survey of these two perspectives on the (multivariable) Alexander polynomial is \cite{Viro}.

The Alexander and Jones polynomials both admit categorifications.  The latter is categorified by even and odd Khovanov homology $\mH_{\mathrm{Kh}}$, $\mH_{\mathrm{OKh}}$ \cite{KhHom,ORS} and the former is by knot Floer homology $\mathrm{HFK}$ \cite{OSHFK}.  Both these invariants admit combinatorial descriptions but have very different origins.  (Even) Khovanov homology was first defined using a TQFT motivated by the representation theory of $U_q(\sltwo)$ and the original definition of knot Floer homology counts certain holomorphic disks in a symmetric product of a Heegaard surface.  Odd Khovanov homology is an integral lift of $\F_2$-valued Khovanov homology that is distinct from (even) Khovanov homology.  Both the even and odd Khovanov homology theories categorify the Jones polynomial.

One of our main goals in this paper and its sequels is to categorify the relation \eqref{eqn-equal-Jones-Alex-numbers}.  Our approach parallels the 2-representation theoretic construction of link homology theories following Khovanov-Lauda \cite{Lau1,KL1,KL2,KL3}, Rouquier \cite{Rou2}, and Webster \cite{Webster}.  The use of KLR (or quiver Hecke) and Webster algebras gives Khovanov homology a role in 2-representation theory analogous to that of the Jones polynomial in ordinary (quantum) representation theory. The odd nilHecke algebra was developed in \cite{EKL} to put odd Khovanov homology on a similar footing.  These algebras are a special case of the quiver Hecke superalgebras of \cite{KKT}; closely related spin Hecke algebras were introduced even earlier in \cite{Wang,KW1}.

The results described above categorify the representation theory of $U_q(\sltwo)$ for generic values of $q$.  A programme for categorically specializing $q$ to a root of unity was suggested by Khovanov in \cite{Hopforoots}.  At prime roots of unity $q^{2p}=1$, this is initiated in \cite{QYHopf, KQ, EliasQi, QiSussan, EliasQi2} by working over a field of characteristic $p>0$.  The present paper begins the lift of these results to characteristic zero in the case $q=\sqrt{-1}$.  In the generic-$q$ categorification theorems of Khovanov-Lauda, Rouquier, and Webster, a grading shift by $k$ categorifies multiplication by $q^k$.  Here, by equipping our algebras with a differential graded (or ``dg'') structure and passing to derived categories, we categorify $-1$ by a homological shift and $\sqrt{-1}$ by a grading shift.  The grading shift acts as a sort of categorical square root of the homological shift.  The use of superalgebras is essential here, forcing the use of the odd nilHecke algebras instead of even nilHecke algebras (the latter is the $\mathfrak{g}=\sltwo$ case of KLR algebras).  An odd analogue of Webster's tensor product algebras is expected to categorify odd Khovanov homology, so we pursue a categorical link between $\mathrm{H_{OKh}}$ and $\mathrm{HFK}$.

Though we make no use of the connection here, several authors have used the odd nilHecke and other related algebras in the categorification of quantized enveloping algebras for Kac-Moody superalgebras \cite{HillWang,ClarkWang,EL}.  See also the geometric approach of \cite{FanLi,CFLW}. The dg algebras used in \cite{KhoGL12,KhovanovSussan} to categorify quantum $\mathfrak{gl}_{1|m}$ are based on a particularly simple case of certain dg algebras appearing in knot Floer theory \cite{LOT}. The approach in \cite{LOT,KhoGL12,KhovanovSussan} is closely related to our current work as well.

\subsection{Outline}

Section \ref{sec-prelim} is a collection of background material: definitions of the quantum groups we will categorify, the homological algebra of dg algebras, diagrammatic algebras and the notion of diagrammatic locality, and two superalgebras that will play leading roles in our constructions: odd symmetric polynomials and the odd nilHecke algebras.  Nothing is essentially new here except for the simple observation that the derived category of ``half-graded'' chain complexes categorifies the Gaussian integers; see Subsection \ref{subsect-categorifying-Gaussian}. The reader is advised to skip all of Section \ref{sec-prelim} on a first reading and to return as needed later on.

In $\cite{EKL}$, the $n$-th \emph{odd nilHecke (super)algebra} $\onh_n$ categorifies the degree-$n\alpha$ piece of $U_q^+(\sltwo)$ for generic $q$.  This graded piece is one-dimensional and is spanned by the divided power $E^{(n)}$, and the isomorphism of the main theorem of that paper identifies $E^{(n)}$ with the $K_0$ class of the unique graded indecomposable projective module $P_n$ of $\onh_n$.  In fact, $\onh_n$ is isomorphic to a matrix algebra whose column module is $P_n$,
\begin{equation}\label{eqn-onh-osym}
\onh_n\cong\END_{\osym_n}(P_n).
\end{equation}
The superalgebra of \emph{odd symmetric polynomials} $\osym_n$ first appeared in $\cite{EK}$.  Section \ref{sec-local-d} adds a dg structure to the results of \cite{EKL}; this is a characteristic-zero lift of the results of \cite{KQ} for the rank-one case $\mathfrak{g}=\sltwo$ at the prime $p=2$.

Via the isomorphism \eqref{eqn-onh-osym}, a dg module structure on $P_n$ induces a dg algebra structure on $\onh_n$.  We prove that, up to an involution, there is a unique dg $\opol_n$-module structure on $P_n$ such that the induced differential on $\onh_n$ is diagrammatically local.  The resulting differential on $\onh_n$ takes a particularly nice diagrammatic form:
\begin{equation*}
d\left(
\hackcenter{\begin{tikzpicture}[scale=0.5]
    \draw[thick] (0,0) -- (0,2)
        node[pos=.5] () {\bbullet};
\end{tikzpicture}}
\right)=
\hackcenter{\begin{tikzpicture}[scale=0.5]
    \draw[thick] (0,0) -- (0,2)
        node[pos=.33] () {\bbullet}
        node[pos=.67] () {\bbullet};
\end{tikzpicture}},
\qquad
d\left(
\hackcenter{\begin{tikzpicture}[scale=.75]
    \draw[thick] (0,0) [out=90, in=-90] to (1,1);
    \draw[thick] (1,0) [out=90, in=-90] to (0,1);
\end{tikzpicture}}
\right)=
\hackcenter{\begin{tikzpicture}[scale=.75]
    \draw[thick] (0,0) [out=90, in=-90] to (0,1);
    \draw[thick] (1,0) [out=90, in=-90] to (1,1);
\end{tikzpicture}}.
\end{equation*}
When equipped with this differential and considered as an $(\opol_n,\osym_n)$-(super)bimodule, we refer to $P_n$ as the dg bimodule $Z_n$. We regard this bimodule as analogue of the usual (even) equivariant cohomology ring of the flag variety for $\mathrm{GL}(n,\C)$.
 Corollary \ref{cor-onh-end} is the resulting isomorphism
\begin{equation}
(\onh_n,d)\cong\END_{\osym_n^\mathrm{op}}(Z_n).
\end{equation}

Our first main result (Theorem \ref{thm-small-sl2}), a characteristic-zero dg categorification of the small quantum group $u^+:=u_{\sqrt{-1}}^+(\sltwo)$, now follows easily from a calculation showing that $\onh_n$ is acyclic if and only if $n\geq2$, so that the derived category $\mc{D}(\onh_n)\simeq 0$ whenever $n\geq 2$.  On the level of Grothendieck groups, we obtain an isomorphism
\begin{equation*}
K_0(\mc{D}(\onh)):=\bigoplus_{n\in \Z_{\geq0}}K_0(\mc{D}(\onh_n)) \cong u^+
\end{equation*}
of $\sqrt{-1}$-twisted bialgebras.  The more intrinsic categorical reason for the nilpotence of $E\in u^+$ is the existence of a chain complex $U_n$, acyclic if and only if $n\geq2$, such that
\begin{equation*}
Z_n\cong U_n\otimes \osym_n.
\end{equation*}
We describe the chain complex $U_n$ explicitly in Subsection \ref{subsec-gen-zoid-app} of the appendix.

The small quantum group $u^+$ is only two-dimensional because its generator $E$ is nilpotent, $E^2=0$.  The ``big'' quantum group $U^+:=U_{\sqrt{-1}}^+(\sltwo)$ is defined to have basis $1,E,E^{(2)},E^{(3)},\ldots$ over $\Z[\sqrt{-1}]$ and structure coefficients given by
\begin{equation*}
E^{(a)}E^{(b)}={a+b\brack a}=\frac{[a+b]!}{[a]![b]!}E^{(a+b)}.
\end{equation*}
This $q$-binomial coefficient is to be interpreted as specialized at $q=\sqrt{-1}$.  In other words, $u^+$ specializes the generators and relations definition of $U_q^+(\sltwo)$, while $U^+$ mimics a particular basis with structure coefficients.

In Section \ref{sec-thick}, we introduce a diagrammatic calculus for the category of $\osym_n$-modules.  This ``thick calculus'' is based on that of \cite{EKL}, which in turn is based on the even thick calculus of \cite{KLMS}.  We give two constructions of dg odd thick calculus: an extrinsic one via generators and relations and an intrinsic one via dg $(\osym_{a_1},\ldots,\osym_{a_r})$-bimodules generalizing the dg $\osym_n$-bimodule $Z_n$. These bimodules are odd analogues of the (even) equivariant cohomology rings of partial flag varieties. For the case $r=2$, we explicitly describe a chain complex $V_{a_1,a_2}$ such that $\osym_{a_1}\otimes\osym_{a_2}\cong V_{a_1,a_2}^*\otimes V_{a_1,a_2}\otimes\osym_{a_1+a_2}$ and study its differential in Subsection \ref{apx-cohomology-thick} of the appendix.

Our second main Theorem \ref{thm-big-sl2} uses derived induction and (twisted) restriction functors to categorify $U^+$ as a $\sqrt{-1}$-twisted bialgebra,
\begin{equation*}
K_0(\mc{D}(\osym)):=\bigoplus_{a\in \Z_{\geq0}}K_0(\mc{D}(\osym_a))\cong U^+.
\end{equation*}
The inclusion $u^+\hookrightarrow U^+$ is categorically lifted to an embedding of derived categories
\begin{equation*}
\mc{D}(\onh)\hookrightarrow \mc{D}(\osym)
\end{equation*}
realized by the derived tensor product with the bimodule $Z_n^\vee$ dual to $Z_n$ (see Corollary \ref{cor-embedding} and Corollary \ref{cor-correct-embedding}).

\subsection{Future work}

As explained above, this is the first in a series of papers in which we seek to explore the higher representation theory relating the Alexander and Jones polynomials by combining techniques of odd and dg categorification.

Lauda's 2-category $\mc{\dot{U}}$ categorifying the Beilinson-Lusztig-MacPherson idempotented form $\dot{\mathrm{U}}(\sltwo)$ has been adapted to both the odd \cite{EL} and $p$-dg contexts \cite{EliasQi,EliasQi2}.  There should be a dg structure on the $\pi=-1$ specialization of the super-2-category $\mc{\dot{U}_\pi}$ of \cite{EL} lifting the differential of \cite{EliasQi,EliasQi2} at $p=2$ to characteristic zero.  Such a dg enhancement of $\mc{\dot{U}_\pi}$ gives a notion of \emph{odd dg 2-representation of quantum $\sltwo$}.  Examples should include odd analogues of Webster algebras \cite{Webster} and Lauda's equivariant flag 2-category \cite{Lau2}.

By mimicking Webster's construction of $\mathrm{H_{Kh}}$ using tensor product algebras, dg odd categorification should give a construction of a differential on $\mathrm{H_{OKh}}$.  There is a folklore conjecture that there is a spectral sequence from $\mathrm{H_{Kh}}$ to $\mathrm{HFK}$ (and it is known that no such spectral sequence exists from $\mathrm{H_{OKh}}$).  So we expect this differential to give rise to a spectral sequence from $\mathrm{H_{OKh}}$ to an invariant closely related to $\mathrm{HFK}$.

There should also be extensions of the above to certain Kac-Moody algebras of higher rank, namely those which have associated quiver Hecke superalgebras \cite{KKT}.

\subsection{Acknowledgments}
A.P.E. received support from an AMS-Simons travel grant.  Both authors thank Mikhail Khovanov and Robert Lipshitz for helpful conversations.

\section{Preliminaries}\label{sec-prelim}

\paragraph{Conventions.}
We work over a commutative ground ring $\Bbbk$ with unit.  At times we will take $\Bbbk$ to be $\Z$ or a field.  Throughout, all algebras are $\Bbbk$-linear.  We will grade algebras by $\Z$, $\Z/2\Z$, and $\Z\times(\Z/2\Z)$ at various points.  Any $\Z/2\Z$-grading will be called \emph{parity} and the parity of a homogeneous element $x$ will be written as $p(x)$.  Any $\Z$-grading will be written as $|x|$.

Our $\Z/2\Z$-graded algebras will be considered as superalgebras (but often referred to simply as ``algebras'').  Writing $\Z/2\Z=\lb\zerobar,\onebar\rb$, we say a parity homogeneous element $x$ is \emph{even} (respectively, \emph{odd}) if $p(x)=\zerobar$ (respectively, $p(x)=\onebar$).  Any superalgebra $A$ has a \emph{parity involution} given on parity homogeneous elements by $\iota_A(a)=(-1)^{p(a)}a.$

For any integer $k$, let
\begin{equation*}
\lbrace k\rbrace=\begin{cases}0&k\text{ is even,}\\1&k\text{ is odd.}\end{cases}
\end{equation*}
This is the $q=-1$ specialization of the unbalanced $q$-integer $\lbrace k\rbrace_q=(1-q^k)/(1-q)$.

If $a=(a_1,\ldots,a_r)$ is a tuple of non-negative integers with $0\leq r\leq n$, let
\begin{equation*}
x^a=x_1^{a_1}\cdots x_r^{a_r}.
\end{equation*}

\subsection{Quantum groups at a root of unity}\label{subsec-qgs}

Denote by $u^+:=u_{\sqrt{-1}}^+(\sltwo)$ the \emph{positive half of small quantum $\mathfrak{sl}_2$ at a fourth root of unity}: as an algebra, it is simply
\[
u^+\cong \Z[\sqrt{-1}][E]/(E^2).
\]
It has a comultiplication given by
\[
r:u^+\lra u^+\otimes_{\Z[\sqrt{-1}]}u^+, \quad r(E)=E\otimes 1+ 1\otimes E.
\]

The \emph{positive half of big quantum $\mathfrak{sl}_2$ at a fourth root of unity}, denoted $U^+_{\sqrt{-1}}(\sltwo)$ or just $U^+$ for short, is the unital $\Z[\sqrt{-1}]$-algebra with basis $\{E^{(n)}\}_{n\in\Z_{\geq0}}$ ($E^{(0)}$ is understood to be $1\in \Z[\sqrt{-1}])$, with multiplication given by
\begin{equation}\label{eqn-EaEb}
E^{(a)}E^{(b)}={a+b \brack a}E^{(a+b)},
\end{equation}
for all $a,b\in \Z_{\geq0}$. Here, the quantum binomial coefficient
\[
{a+b\brack a}=\frac{[a+b]!}{[a]![b]!}
\]
is understood as evaluated at $q=\sqrt{-1}$, in $\Z[\sqrt{-1}]$. In particular, it is zero if $a+b$ is even.

The comultiplication $r: U^+\to U^+\otimes_{\Z[\sqrt{-1}]} U^+$ is defined by
\begin{equation}\label{eqn-comultiplication}
r(E^{(a)})=\sum_{c=0}^{a}(-\sqrt{-1})^{c(a-c)}E^{(c)}\otimes E^{(a-c)}.
\end{equation}

The algebras $u^+$ and $U^+$ equipped with the comultiplication maps $r$ defined above are twisted bialgebras in the sense of Lusztig \cite[Section 1.2.2]{Lus4}.  Fix $v\in\Bbbk^\times$ and let $B$ be a $\Z$-graded $\Bbbk$-algebra equipped with a graded coassociative comultiplication $r$.  Give $B\otimes_\Bbbk B$ the product structure
\begin{equation*}
(b_1\otimes b_2)(b_1'\otimes b_2')=v^{|b_2||b_1'|}(b_1b_1')\otimes(b_2b_2').
\end{equation*}
Then we say $B$ is a \emph{($v$-)twisted $\Bbbk$-bialgebra} if $r:B\to B\otimes_\Bbbk B$ is an algebra homomorphism.

If we set $|E|=1$ in $u^+$ and $|E^{(n)}|=n$ in $U^+$, then the algebras $u^+$ and $U^+$ are $\sqrt{-1}$-twisted $\Z[\sqrt{-1}]$-bialgebras and
\begin{equation*}
u^+\hookrightarrow U^+,\quad E\mapsto E^{(1)}
\end{equation*}
is an injective homomorphism of twisted bialgebras, since
$$E^{(1)}E^{(1)}=(\sqrt{-1}+(\sqrt{-1})^{-1})E^{(2)}=0.$$

\subsection{Differential graded algebra}

In this subsection, we gather some basic definitions from the theory of differential graded (dg) algebras and modules. The material apart from Subsubsection \ref{subsect-categorifying-Gaussian} is entirely standard, and the main purpose is to fix the notation we use. The reader is referred to \cite{BL,Ke1} for more details.  One minor adaptation from the standard terminology is that, for flexibility, we will grade our dg algebras by $\Z/2\Z$.  Usual $\Z$-graded dg algebras inherit a $\Z/2\Z$-grading simply by collapsing the grading modulo $2$.  Later, when we consider $\Z\times(\Z/2\Z)$-graded algebras, the dg structure will be with respect to the parity factor.

A dg algebra $(A,d_A)$ is a $\Z/2\Z$-graded $\Bbbk$-algebra $A=A^{\zerobar}\oplus A^{\onebar}$ and a $\Bbbk$-linear map $d_A:A\to A$ of degree $\onebar$ (called the
\emph{differential}) satisfying $d^2=0$ that is a derivation; that is, for any homogeneous $a,b\in A$,
\begin{equation}
d_A(ab)=d_A(a)b+(-1)^{p(a)}ad_A(b).
\end{equation}

A left dg $A$-module $(M,d_M)$ is a $\Z/2\Z$-graded $A$-module $M=M^{\zerobar}\oplus M^{\onebar}$ and a $\Bbbk$-linear map $d_M:M\to M$ of odd parity such that for any homogeneous elements $a\in A$, $m\in M$, we have
\begin{equation}\label{eqn-left-dg-module}
d_M(am)=d_A(a)m+(-1)^{p(a)}ad_M(m).
\end{equation}
For the notion of a right dg $A$-module, replace equation \eqref{eqn-left-dg-module} by
\begin{equation}
d_M(ma)=d_M(m)a+(-1)^{p(m)}md_A(a).
\end{equation}
We will usually drop mention of the differentials from the names of dg algebras and dg modules, decorating the names as necessary if ambiguity is possible.  If $A,B$ are dg algebras, then a dg $(A,B)$-bimodule is a $\Bbbk$-module equipped with a differential and commuting left dg $A$-module and right dg $B$-module structures.

\begin{rem}\label{rmk-dg-category}
A mild generalization of the notion of a dg algebra is that of a \emph{dg category}.  Informally, a dg category is like an algebra in which the unit has been replaced by a family of $d$-closed orthogonal idempotents.  A dg algebra is a dg category with one object.  Most of definitions and results below have easy generalizations to the setting of dg categories (and the proofs are nearly identical).  One added convenience is the fact that the category of dg modules over a dg category is itself a dg category.  We refer the reader to \cite{Ke1} for more details.
\end{rem}

Let $M$ be a dg module over a dg algebra $A$, and consider $\END_A(M)$ to be the space of all $A$-module endomorphisms on $M$. It is naturally $\Z/2\Z$-graded, with the homogeneous degree-$k$ part consisting of maps of parity $k$ :
\begin{equation}
\END_A^k(M) : = \{f:f(M^{l})\subseteq M^{l+k},~l\in\Z/2\Z\}.
\end{equation}

The algebra $\END_A(M)$ is a dg algebra via the differential $d_\END$ defined by
\begin{equation}\label{eqn-d-end}
d_\END(f)(m)=d(f(m))-(-1)^{p(f)}f(d(m)).
\end{equation}
Moreover, $M$ becomes a right dg module over $\END_A(M)$ if we define the right action as
\begin{equation}\label{eqn-d-end-element}
m\cdot f : = (-1)^{p(f)p(m)}f(m).
\end{equation}
This action commutes with the left $A$-module action on $M$, and both actions are compatible with the differentials. In this way $M$, has a natural dg bimodule structure over $A$ and $\END_A(M)$. We will usually drop the subscript and simply write $d$ for $d_\END$ if the context is clear. Similar statements hold when one replaces ``right'' notions with ``left'' ones.

We will repeatedly consider the case in which $M$ is a dg submodule of $A$ of the form $M\cong A\tte$ for some homogeneous idempotent $\tte\in A$. In this case $d(\tte)=a_0\tte$ for some $a_0\in A$. Recall that there is a canonical isomorphism of graded abelian groups
\begin{equation}\label{eqn-endo-idempotent}
\HOM_A(A\tte, A\tte)\cong \tte A \tte,
\end{equation}
where, given any homogeneous element $\tte a\tte \in \tte A\tte$, one has the corresponding left $A$-module map $f_{\tte a\tte}(b\tte):=(-1)^{p(a)p(b)}b\tte a \tte$. Then, under the isomorphism \eqref{eqn-endo-idempotent}, the algebra $\tte A \tte$ inherits a dg structure, which we compute as follows. For any $\tte a \tte \in \tte A \tte$ and $b\tte \in A\tte$, notice that $p(f_{\tte a \tte})=p(a)$.  It follows that
\begin{eqnarray*}
d(f_{\tte a \tte}) (b\tte) & = & d(f_{\tte a \tte}(b \tte))-(-1)^{p(a)}f_{\tte a \tte}(d(b \tte))\\
& = & (-1)^{p(a)p(b)}d(b\tte a \tte )- (-1)^{p(a)+p(a)(p(b)+1)}d(b \tte) \tte a \tte\\
& = & (-1)^{p(a)p(b)}d(b \tte) a \tte +(-1)^{p(a)p(b)+p(b)}b\tte d(a \tte) -(-1)^{p(a)p(b)}d(b \tte) a \tte\\
& = & (-1)^{(p(a)+1)p(b)}b\tte (d a \tte)\\
& = & f_{\tte d(a \tte)}(b\tte)~.
\end{eqnarray*}
The next result follows from this discussion.

\begin{lem}\label{lem-d-endo-idempotent} Let $(A, d_A)$ be a dg algebra and $\tte\in A$ be an idempotent with $d(\tte)\in A\tte$. Then $A\tte$ has a right dg module structure over the dg algebra $(\tte A \tte,d_{\tte})$ in which the differential $d_{\tte}$ is given by
\[
d_{\tte} (\tte a \tte)= \tte d_A(a \tte )
\]
for any $\tte a \tte \in \tte A \tte$. \hfill$\square$
\end{lem}

\subsubsection{The homotopy and derived categories}

Let $M$, $N$ be two dg modules over $A$. A map of graded $A$-modules $f \colon M \lra  N$ is called a \emph{morphism of dg modules over $A$}, or simply a \emph{morphism}, if it commutes with the differentials on $M$ and $N$ ($f$ necessarily has degree zero). A morphism $f \colon M \lra N$ is called \emph{null-homotopic} if there is a homogeneous $A$-module map $h$ of degree $-1$ such that
\[
f = d_N\circ h+h\circ d_M.
\]
It is easy to see that the category of all left dg modules over $A$, which we denote by $A\dmod$, is abelian and that the class of all null-homotopic morphisms forms an ideal in $A\dmod$. The categorical quotient of $A\dmod$ by this ideal is known as the \emph{homotopy category} of dg modules over $A$, denoted $\mc{H}(A)$.

The cohomology of a dg module $M$ over $A$ is, by definition, $\mH(M):= \mathrm{Ker}(d)/\mathrm{Im}(d)$. It is naturally $\Z/2\Z$-graded\footnote{ We sometimes write $\mH(M)=\mH^\bullet(M)$ to emphasize the grading or parity.} and equipped with a module structure over the algebra $\mH(A)$. We will regard $\mH(A)$ as a dg algebra with the zero differential and $\mH(M)$ as a dg module over this algebra.

A morphism of dg modules $f\colon M \lra N$ is called a \emph{quasi-isomorphism} if it induces an isomorphism on cohomology. The images of quasi-isomorphisms in $\mc{H}(A)$ constitute a localizing class $\Cc_{qis}$, and the category $\mc{D}(A):=\mc{H}(A)[\Cc_{qis}^{-1}]$ obtained by localization with respect to this class is called the \emph{derived category} of dg modules over $A$. By abuse of notation, we will refer to an object in $A\dmod$, $\mc{H}(A)$, or $\mc{D}(A)$ as a dg module.

The homotopy category $\mc{H}(A)$ and derived category $\mc{D}(A)$ are triangulated. The \emph{translation functor} $[1]$ acts on any dg module by a parity shift:
\[
(M[1])^k:= M^{k+\onebar}, \quad d_{M[1]}=-d_M.
\]
Let $f:M\lra N$ be a morphism of dg modules over $A$. The \emph{cone} of $f$ is by definition $C(f):=N\oplus M[1]$ equipped with the differential given in block matrix form by
\[
d_{C(f)}:=
\left(
\begin{matrix}
d_N & f\\
0 & -d_M
\end{matrix}
\right).
\]
The dg modules $N$, $M[1]$ can be naturally identified with, respectively, a dg sub or dg quotient module of the cone $C(f)$.
A \emph{standard distinguished triangle} in $\mc{H}(A)$ or $\mc{D}(A)$ is a diagram of the form
\[
\xymatrix{
M \ar[r]^f & N \ar[r]^-\iota & C(f) \ar[r]^-\pi & M[1],
}
\]
where $\iota$ and $\pi$ are the natural inclusion and projection maps. A \emph{distinguished triangle} in $\mc{H}(A)$ or $\mc{D}(A)$ is a diagram that is isomorphic to a standard distinguished triangle in the relevant category.

One useful feature of dg category theory is that the category $A\dmod$ behaves like a dg algebra with multiple objects. More precisely, $A\dmod$ has a dg enhancement given by defining the \emph{$\HOM$-complex} between any two dg modules $M$ and $N$ to be
\[
\HOM_A(M,N)=\Hom^{\zerobar}_A(M,N)\oplus \Hom_A^{\onebar}(M,N),
\]
where $\Hom^k_A(M,N)$ stands for $A$-module maps of parity $k$. The differential $d$ acts on a homogeneous map $f\in \HOM_A(M,N)$ as
\[
d(f):=d_N\circ f -(-1)^{p(f)}f\circ d_M.
\]
In this way, the endomorphism dg algebra $\END_A(M)$ we considered earlier is simply the endomorphism complex in the dg enhanced module category. Furthermore, one readily sees that the morphism space in the homotopy category $\mc{H}(A)$ is naturally isomorphic to the degree-zero part of the cohomology,
\begin{equation}\label{eqn-morphism-in-homotopy-category}
\Hom_{\mc{H}(A)}(M,N)=\mH^0(\HOM_A(M,N)).
\end{equation}
One can similarly compute the morphism space in the derived category using this formula, by first resolving $M$ with a ``nice enough'' replacement. Recall that a dg module $P$ over $A$ is called \emph{cofibrant} or \emph{K-projective} if, for any surjective quasi-isomorphism $f\colon M \lra N$ and any morphism $g\colon P\lra N$, there exists a morphism $\overline{g}$ making the diagram
\[
\xymatrix{
 & P \ar@{-->}[dl]_{\overline{g}}  \ar[d]^g\\
 M \ar[r]_f & N
}
\]
commute.  A morphism of dg modules $g\colon P\lra N$ is called a \emph{cofibrant replacement} or \emph{K-projective resolution} if $P$ is cofibrant and $g$ is a surjective quasi-isomorphism.

\begin{prop}\label{prop-cofibrant-replacement}Let $A$ be a dg algebra and $M$ a dg $A$-module. Then there exists a cofibrant replacement $ P_M\lra M $, unique up to quasi-isomorphism.
\end{prop}
\begin{proof}One can adapt the proofs in \cite[Section 10.12.2]{BL} or \cite[Section 3]{Ke1} to the $\Z/2\Z$-graded situation for the existence part. The uniqueness follows easily from the definition. Alternatively, one can apply the hopfological construction in \cite[Theorem 6.6]{QYHopf} to Sweedler's $4$-dimensional Hopf algebra to prove the result directly.
\end{proof}

If $M$, $N$ are dg modules over $A$, then there is a natural isomorphism
\begin{equation}\label{eqn-morphism-in-derived-category}
\Hom_{\mc{D}(A)}(M,N)\cong \Hom_{\mc{H}(A)}(P_M,N)=\mH^0(\HOM_A(P_M,N)).
\end{equation}
The uniqueness part of Proposition \ref{prop-cofibrant-replacement} implies that the right-hand side of \eqref{eqn-morphism-in-derived-category} is, up to isomorphism, independent of choice of cofibrant replacement. Motivated by the above, we define the $\RHOM$-complex between two dg modules to be
\begin{equation}\label{eqn-RHOM}
\RHOM_A(M,N):= \HOM_A(P_M,N),
\end{equation}
where $P_M$ is any cofibrant replacement for $M$. The uniqueness part of Proposition \ref{prop-cofibrant-replacement} again shows that up to isomorphism, the $\RHOM$-complex is independent of the choice of cofibrant replacement.

\subsubsection{The Grothendieck group of a dg algebra}

When talking about Grothendieck groups, one needs to restrict the ``size'' of dg modules to avoid trivial cancellations. A commonly adopted size restriction on dg modules is the compactness condition. Note that $\mc{D}(A)$ is an additive category with infinite direct sums. A dg module $M$ is called \emph{compact}\footnote{In \cite{Ke1}, a compact module is called a small module instead.} if the functor $\Hom_{\mc{D}(A)}(M,-)$ commutes with infinite direct sums. The strictly full subcategory of $\mc{D}(A)$ consisting of objects that are compact is called the \emph{compact} or \emph{perfect} derived category, which we will denote by $\mc{D}^c(A)$.

\begin{example}[Finite-cell modules]\label{eg-finite-cell}One easily sees that the left or right regular module $A$ over itself is compact. In fact, any dg summand\footnote{By a dg summand of $A$ we mean a projective $A$-module $A\tte$ defined by an idempotent $\tte\in A$ with $d(\tte)=0$.} of the regular module is also compact. More generally, a filtered dg module is also compact if it comes with a finite filtration whose subquotients are isomorphic as dg modules to dg summands of $A$. Such a filtered dg module will be called a \emph{finite-cell} module over $A$.
\end{example}

The \emph{Grothendieck group} of a dg algebra $A$, denoted by $K_0(A)$, is the quotient of the free abelian group on the isomorphism classes $[M]$ of compact dg modules $M$ by the relation
$[M]=[M_1]+[M_2]$ whenever
$$M_1\lra M \lra M_2\lra M_1[1]$$
is a distinguished triangle of compact objects in $\mc{D}(A)$.

Computing dg Grothendieck groups is, in general, a hard problem. But for a class of dg algebras that arise quite often in representation theory, this can be done relatively easily. Let $\Bbbk$ be a field. A $\Z$-graded dg algebra is called \emph{positive} if it satisfies the following conditions~\cite{SchPos}:
\begin{itemize}
\item the algebra $A=\oplus_{i\in \Z_{\geq0}} A^i$ is non-negatively graded,
\item the degree zero part $A^0$ is semisimple, and
\item the differential acts trivially on $A^0$.
\end{itemize}

\begin{thm}[\cite{SchPos}]\label{thm-compact-mod} Let $A$ be a positive dg algebra over a field. Then a dg module $M$ is compact if and only if it is isomorphic in the derived category to a finite-cell module. \hfill$\square$
\end{thm}

\begin{cor}\label{cor-K0-positive}Let $A$ be a positive dg algebra, and $A^0$ be its homogeneous degree zero part. Then
\[
K_0(A)\cong K_0(A^0).
\]
\end{cor}

\begin{proof}[Sketch of proof] By Theorem~\ref{thm-compact-mod}, any compact module is isomorphic in the derived category to a dg module with a filtration whose subquotients are dg summands of $A$. Such summands are, up to grading shift, given by indecomposable idempotents in $A^0$. The isomorphism classes of these indecomposable idempotents span $K_0(A)$.
\end{proof}

\subsubsection{Derived functors}

To a dg bimodule $_AM_B$ over dg algebras $A$ and $B$, we associate two derived functors. The first is the \emph{tensor functor},
\begin{equation}\label{eqn-tensor-functor}
M\oL_B (-): \mc{D}(B) \lra \mc{D}(A), \quad N \mapsto M\otimes_A P_N,
\end{equation}
where $N$ is an arbitrary dg $B$-module and $P_N$ is a cofibrant replacement of $N$. The other is the \emph{hom functor}, given by
\begin{equation}\label{eqn-hom-functor}
\RHOM_A(M,-): \mc{D}(A) \lra \mc{D}(B), \quad L \mapsto  \HOM_{A}(P_M, L).
\end{equation}
Here we choose a cofibrant replacement $P_M$ for $M$ as a dg bimodule over $A\otimes B^\op$ so that the complex $\HOM_{A}(P_M, L)$ carries a strict dg module structure over $B$. It is easy to see that these functors are well-defined and that they are adjoint to each other even on the level of $\HOM$-complexes:
\begin{equation}\label{eqn-tensor-hom-adjunction}
\RHOM_A(M\oL_B N, L)\cong \RHOM_B(N, \RHOM_A(M,L)).
\end{equation}
The usual tensor-hom adjunction on the level of derived categories follows from \eqref{eqn-tensor-hom-adjunction} by taking degree-zero cohomology.

Of particular interest to us is the case of the bimodule $_AA_B$ associated to a dg algebra homomorphism $\varphi:B\to A$ ($B$ acts on the right via $\varphi$). The tensor and hom functors in this case are also known as the induction and restriction functors, denoted $\varphi^*$ and $\varphi_*$ respectively:
\begin{equation}\label{eqn-ind-res}
\varphi^*(N):=A\oL_B N, \quad \varphi_*(L):=\RHOM_A(A,L).
\end{equation}

The following theorem is useful.

\begin{thm}\label{thm-qis-dga} Let $\varphi:B\lra A$ be a quasi-isomorphism of dg algebras. Then the induction and restriction functors $\varphi^*$, $\varphi_*$ are quasi-equivalences of triangulated categories, quasi-inverse to each other.
\end{thm}
\begin{proof}See \cite[Theorem 10.12.5.1]{BL} or \cite[Example 6.1]{Ke1}.
\end{proof}

\begin{cor}\label{cor-acyclic-dga} Let $A$ be a dg algebra.  The following are equivalent:
\begin{enumerate}
    \item The derived category $\mc{D}(A)$ is equivalent to $0$.
    \item The cohomology of $A$ vanishes, $\mH(A)=0$.
    \item There exists an $x\in A$ such that $d(x)=1$.
\end{enumerate}
\end{cor}
\begin{proof} Statement 1 follows from statement 2 by the previous theorem. Statement 1 implying 3 is easy. Thus it suffices to show that statement 3 implies 2. If $[y]\in \mH^*(A)$ is a cohomology class, then
\[
d(xy)=d(x)y-xd(y)=y,
\]
so that $[y]=0$. The result follows.
\end{proof}

\subsubsection{A categorification of the Gaussian integers}\label{subsect-categorifying-Gaussian}

Before moving on to specific dg algebras, we study the base category for our main theorems: ``half-graded'' chain complexes over $\Bbbk$ categorify the ring of Gaussian integers $\Z[\sqrt{-1}]$.

More precisely, we work with $\Z\times (\Z/2\Z)$-graded modules.  We write $\langle 1 \rangle$ for a shift of the $q$-degree by $1$ and $[1]$ for the parity shift. Next we define a differential between such modules to be a map of bidegree $(2,\overline{1})$ which squares to $0$ and a chain complex to be a $\Bbbk$-module equipped with such a differential. Applying the usual homotopy category construction to these chain complexes, we obtain the category of \emph{half-graded chain complexes up to homotopy}, which we denote by $\mc{C}(\Bbbk)$. The derived category of this homotopy category will be denoted $\mc{D}(\Bbbk)$. When $\Bbbk$ is a field, the quotient functor from $\mc{C}(\Bbbk)$ onto $\mc{D}(\Bbbk)$ is an equivalence of triangulated categories.

For simplicity, let us focus on either $\Bbbk=\Z$ or a field. We show that under this assumption, we have
\begin{equation}\label{eqn-categorifying-Gaussian}
K_0(\mc{D}(\Bbbk))\cong \Z[\sqrt{-1}].
\end{equation}

This can easily be seen as follows.  Any chain complex in $\mc{D}(\Bbbk)$ is isomorphic to a direct sum of indecomposable chain complexes of the following forms:
\begin{itemize}
\item a single copy of $\Bbbk$ sitting in any bidegree;
\item a copy of $S=(0\lra \Bbbk \stackrel{d}{\lra} \Bbbk \lra 0)$, where the left $\Bbbk$ term may sit in any bidegree $(a,b)$ and the one on the right has bidegree $(a+2,b+\overline{1})$.
\end{itemize}
Therefore the Grothendieck group is generated as a $\Z[q^{\pm 1}]$-module by the symbol $[\Bbbk]$, with $q[\Bbbk]=[\Bbbk\langle 1 \rangle]$. When $d$ is given by multiplication by a unit in $\Bbbk$, the last two-term complex becomes contractible and therefore isomorphic to $0$ in $\mc{D}(\Bbbk)$. Hence in the Grothendieck group, the relation
\[
(1+q^2)[\Bbbk]=0
\]
is imposed by summing along the $q$-degree. From the classification of objects, this is also the only relation, and it also forces the symbol of $S$ to be zero even when $d$ is multiplication by any element in $\Bbbk$. In conclusion, we obtain
\[
K_0(\mc{D}(\Bbbk))\cong \Z[q^{\pm 1}]/(1+q^2) \cong \Z[\sqrt{-1}].
\]

Had we taken $d$ to be of bidegree $(1,\overline{1})$, we would have obtained the usual notion of ($\Z$-graded) chain complexes of $\Bbbk$-modules.  The resulting Grothendieck group is isomorphic to $\Z$. For our choice of bidegree, the outcome is an isomorphism of functors $\langle 1\rangle\circ \langle 1 \rangle \cong [1]$ on $\mc{D}(\Bbbk)$, as if we have adjoined a square root of the homological grading shift functor. This explains the term ``half-graded'' above.

\begin{rem}[Homotopy versus derived category] One might wonder what happens if we use the dg homotopy category of half-graded complexes instead of the derived category as a base category.  If $\Bbbk$ is a field (or more generally a division ring), the homotopy and derived categories are equivalent, so there is no difference.  However, over $\Bbbk=\Z$, the complexes
\[
\xymatrix{0 \ar[r] & \Z \ar[r]^{n} & \Z\ar[r]& 0}, \quad \xymatrix{0\ar[r]& \Z/(n)\ar[r] & 0}
\]
are not homotopy equivalent if $n\neq\pm 1$.  So the homotopy and derived categories are not equivalent.  Our results throughout this paper hold over $\Z$ if one restricts to subcategories of free abelian groups.  This is because the homotopy category of half-graded free abelian groups is equivalent to the derived category of half-graded abelian groups.
\end{rem}

\subsection{Diagrammatic algebras}\label{subsec-diagrammatic}

The algebras which will appear in this paper ($\opol_n$, $\onh_n$, and others) have elements that are linear combinations of diagrams of ``strands.''  The diagrams for an algebra on $n$ strands have $n$ arcs connecting $n$ points on one horizontal line to $n$ points on another horizontal line.  Strands may cross and sometimes may carry one or more dots.  Dots and crossings occur at distinct heights, only two strands may cross at a given point, strands cannot have any critical points for the projection to the vertical axis, and we consider these diagrams up to isotopies through such diagrams that do not change relative heights of dots and crossings.

A product $xy$ is drawn as placing the diagram for $x$ above the diagram for $y$:
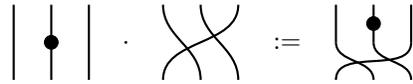
\begin{equation*}
\hackcenter{\begin{tikzpicture}
    \draw[thick] (0,0) -- (0,1);
    \draw[thick] (.5,0) -- (.5,1)
        node[pos=.5] {\bbullet};
    \draw[thick] (1,0) -- (1,1);
\end{tikzpicture}}
\quad\cdot\quad
\hackcenter{\begin{tikzpicture}
    \draw[thick] (0,0) [out=90, in=-90] to (1,1);
    \draw[thick] (.5,0) [out=90, in=-90] to (0,1);
    \draw[thick] (1,0) [out=90, in=-90] to (.5,1);
\end{tikzpicture}}
\quad:=\quad
\hackcenter{\begin{tikzpicture}[scale=0.5]
    \draw[thick] (0,0) [out=90, in=-90] to (2,1);
    \draw[thick] (1,0) [out=90, in=-90] to (0,1);
    \draw[thick] (2,0) [out=90, in=-90] to (1,1);
    \draw[thick] (0,1) -- (0,2);
    \draw[thick] (1,1) -- (1,2)
        node[pos=.5] {\bbullet};
    \draw[thick] (2,1) -- (2,2);
\end{tikzpicture}}.
\end{equation*}

We call the dots and crossings \emph{local generators},
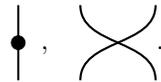
\begin{equation*}
\hackcenter{\begin{tikzpicture}
    \draw[thick] (0,0) -- (0,1)
        node[pos=.5] {\bbullet};
\end{tikzpicture}}
~,\quad
\hackcenter{\begin{tikzpicture}
    \draw[thick] (0,0) [out=90, in=-90] to (1,1);
    \draw[thick] (1,0) [out=90, in=-90] to (0,1);
\end{tikzpicture}}.
\end{equation*}
At this point we can impose \emph{local relations}: relations among generators on some number of strands along with all other relations generated by padding this on the left and right with vertical strands.  In other words, we think of a monoidal category: the points on the horizontal lines are objects, vertical strands are identity morphisms, and dots and crossings are other morphisms.  Imposing a relation among diagrams on $m$ strands induces relations on diagrams with $m$ or more strands.  For instance, the local relation
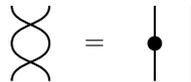
\begin{equation*}
\hackcenter{\begin{tikzpicture}[scale=0.5]
    \draw[thick] (0,0) [out=90, in =-90] to (1,1) [out=90, in=-90] to (0,2);
    \draw[thick] (1,0) [out=90, in =-90] to (0,1) [out=90, in=-90] to (1,2);
\end{tikzpicture}}
\quad=\quad
\hackcenter{\begin{tikzpicture}[scale=0.5]
    \draw[thick] (0,0) -- (0,2)
        node[pos=.5] {\bbullet};
    \draw[thick] (1,0) -- (1,2);
\end{tikzpicture}}
\end{equation*}
also includes the relations
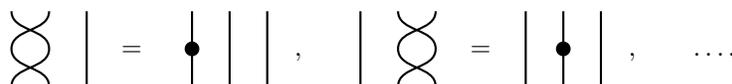
\begin{equation*}
\hackcenter{\begin{tikzpicture}[scale=0.5]
    \draw[thick] (0,0) [out=90, in =-90] to (1,1) [out=90, in=-90] to (0,2);
    \draw[thick] (1,0) [out=90, in =-90] to (0,1) [out=90, in=-90] to (1,2);
    \draw[thick] (2,0) -- (2,2);
\end{tikzpicture}}
\quad=\quad
\hackcenter{\begin{tikzpicture}[scale=0.5]
    \draw[thick] (0,0) -- (0,2)
        node[pos=.5] {\bbullet};
    \draw[thick] (1,0) -- (1,2);
    \draw[thick] (2,0) -- (2,2);
\end{tikzpicture}}\quad,\qquad
\hackcenter{\begin{tikzpicture}[scale=0.5]
    \draw[thick] (0,0) [out=90, in =-90] to (1,1) [out=90, in=-90] to (0,2);
    \draw[thick] (1,0) [out=90, in =-90] to (0,1) [out=90, in=-90] to (1,2);
    \draw[thick] (-1,0) -- (-1,2);
\end{tikzpicture}}
\quad=\quad
\hackcenter{\begin{tikzpicture}[scale=0.5]
    \draw[thick] (0,0) -- (0,2)
        node[pos=.5] {\bbullet};
    \draw[thick] (1,0) -- (1,2);
    \draw[thick] (-1,0) -- (-1,2);
\end{tikzpicture}}\quad,\qquad
\ldots.
\end{equation*}

Suppose for each $n\in\Z_{\geq0}$, $R_n$ is a superalgebra on $n$ strands, with all relations coming from local relations.  Define $\iota_{m,n}:R_m\otimes R_n\hookrightarrow R_{m+n}$ to be the $\Bbbk$-module map that horizontally juxtaposes diagrams with the diagram from $R_m$ on the left, with strands vertically extended so that generators in the diagram from the $R_m$ factor are all above those of the $R_n$ factor:
\begin{equation*}
\iota_{3,3}\left(\quad\hackcenter{\begin{tikzpicture}
    \draw[thick] (0,0) -- (0,1);
    \draw[thick] (.5,0) -- (.5,1)
        node[pos=.5] {\bbullet};
    \draw[thick] (1,0) -- (1,1);
\end{tikzpicture}}
\quad\otimes\quad
\hackcenter{\begin{tikzpicture}
    \draw[thick] (0,0) [out=90, in=-90] to (1,1);
    \draw[thick] (.5,0) [out=90, in=-90] to (0,1);
    \draw[thick] (1,0) [out=90, in=-90] to (.5,1);
\end{tikzpicture}}\quad\right)
\quad=\quad
\hackcenter{\begin{tikzpicture}
    \draw[thick] (1.5,0) [out=90, in=-90] to (2.5,.5);
    \draw[thick] (2,0) [out=90, in=-90] to (1.5,.5);
    \draw[thick] (2.5,0) [out=90, in=-90] to (2,.5);
    \draw[thick] (1.5,.5) -- (1.5,1);
    \draw[thick] (2,.5) -- (2,1);
    \draw[thick] (2.5,.5) -- (2.5,1);
    \draw[thick] (0,0) -- (0,1);
    \draw[thick] (.5,0) -- (.5,1)
        node[pos=.75] {\bbullet};
    \draw[thick] (1,0) -- (1,1);
\end{tikzpicture}}.
\end{equation*}
This is a homomorphism of superalgebras if and only if for all homogeneous elements $x\in R_m$ and $y\in R_n$, we have
\begin{equation*}
\iota_{m,n}(x\otimes 1)\iota_{m,n}(1\otimes y)=(-1)^{p(x)p(y)}\iota_{m,n}(1\otimes y)\iota_{m,n}(x\otimes 1).
\end{equation*}
Diagrammatically, the condition is \emph{distant supercommutativity}:
\begin{equation*}
\hackcenter{\begin{tikzpicture}[scale=0.5]
    \draw[thick] (0,0) -- (0,3);
    \draw[thick] (1,0) -- (1,3);
    \draw[thick] (2,0) -- (2,3);
    \draw[thick] (3,0) -- (3,3);
    \draw[thick] (4,0) -- (4,3);
    \draw[thick] (5,0) -- (5,3);
    \node[draw, thick, fill=white!20,rounded corners=4pt,inner sep=3pt,text width=35,align=center] at (1,2) {$x$};
    \node[draw, thick, fill=white!20,rounded corners=4pt,inner sep=3pt,text width=35,align=center] at (4,1) {$y$};
\end{tikzpicture}}
\quad=\quad(-1)^{p(x)p(y)}
\hackcenter{\begin{tikzpicture}[scale=0.5]
    \draw[thick] (0,0) -- (0,3);
    \draw[thick] (1,0) -- (1,3);
    \draw[thick] (2,0) -- (2,3);
    \draw[thick] (3,0) -- (3,3);
    \draw[thick] (4,0) -- (4,3);
    \draw[thick] (5,0) -- (5,3);
    \node[draw, thick, fill=white!20,rounded corners=4pt,inner sep=3pt,text width=35,align=center] at (1,1) {$x$};
    \node[draw, thick, fill=white!20,rounded corners=4pt,inner sep=3pt,text width=35,align=center] at (4,2) {$y$};
\end{tikzpicture}}\quad.
\end{equation*}
We will always impose the distant supercommutativity relation.

\begin{example} The group algebra $\Bbbk[S_n]$ of the symmetric group is the diagrammatic algebra on $n$ strands with crossings but no dots, giving crossings even parity.  The local relations are
\begin{equation*}
\hackcenter{\begin{tikzpicture}[scale=0.5]
    \draw[thick] (0,0) [out=90, in=-90] to (1,1) [out=90, in=-90] to (0,2);
    \draw[thick] (1,0) [out=90, in=-90] to (0,1) [out=90, in=-90] to (1,2);
\end{tikzpicture}}
\quad=\quad
\hackcenter{\begin{tikzpicture}[scale=0.5]
    \draw[thick] (0,0) -- (0,2);
    \draw[thick] (1,0) -- (1,2);
\end{tikzpicture}}
\quad,\qquad
\hackcenter{\begin{tikzpicture}[scale=0.5]
    \draw[thick] (0,0) [out=90, in=-90] to (2,2);
    \draw[thick] (2,0) [out=90, in=-90] to (0,2);
    \draw[thick] (1,0) [out=90, in=-90] to (0,1) [out=90, in=-90] to (1,2);
\end{tikzpicture}}
\quad=\quad
\hackcenter{\begin{tikzpicture}[scale=0.5]
    \draw[thick] (0,0) [out=90, in=-90] to (2,2);
    \draw[thick] (2,0) [out=90, in=-90] to (0,2);
    \draw[thick] (1,0) [out=90, in=-90] to (2,1) [out=90, in=-90] to (1,2);
\end{tikzpicture}}\quad.
\end{equation*}
\end{example}

\begin{example}\label{eg-odd-pol} The superalgebra of skew polynomials $\opol_n$ is the diagrammatic superalgebra with dots but no crossings, giving dots odd parity.  No other relations (other than distant supercommutativity) are imposed.  In equations,
\begin{equation*}
\opol_n=\Bbbk\langle x_1,\ldots,x_n\rangle/(x_ix_j+x_jx_i=0\text{ if }i\neq j),
\end{equation*}
where $x_i$ stands for a dot on the $i$-th strand.
\end{example}

\begin{example} KLR algebras (quiver Hecke algebras) are diagrammatic algebras, with the definition suitably generalized to allow for the coloring of strands by simple roots \cite{KL1}.  All dots and crossings have even parity.  For certain super Cartan data, the quiver Hecke superalgebras of Kang-Kashiwara-Tsuchioka \cite{KKT} can be expressed as diagrammatic algebras too.\end{example}

Suppose $\lb B_n\rb_{n\geq0}$ is a family of diagrammatic algebras defined by local relations as above and that each $B_n$ has a differential $d$.  We say that this forms a family of \emph{diagrammatically local dg algebras} if $d(1)=0$ and $d$ is determined by its values on diagrammatic generators and extended compatibly with the horizontal juxtaposition operation.

\begin{example} The superalgebra $\opol_n$ has a diagrammatically local differential given by
\begin{equation}\label{eqn-opol-d}
d\left(
\hackcenter{\begin{tikzpicture}[scale=0.5]
    \draw[thick] (0,0) -- (0,2)
        node[pos=.5] () {\bbullet};
\end{tikzpicture}}
\right)=
\hackcenter{\begin{tikzpicture}[scale=0.5]
    \draw[thick] (0,0) -- (0,2)
        node[pos=.33] () {\bbullet}
        node[pos=.67] () {\bbullet};
\end{tikzpicture}}\quad.
\end{equation}
In equations, $d(x_i)=x_i^2$.\end{example}

\subsection{Odd symmetric polynomials and odd nilHecke algebras}\label{subsec-osym}

In this subsection, we gather the necessary background material on odd symmetric polynomials and odd nilHecke algebras. See \cite{EK, EKL} for more details.

\subsubsection{Odd symmetric polynomials}

Let
\begin{equation*}
e_k:=\sum_{1\leq i_1<\ldots<i_k\leq n}x_{i_1}\cdots x_{i_k}
\end{equation*}
be the $k$-th \emph{untwisted odd elementary symmetric polynomial} in $\opol_n$.  By convention, $e_0=1$ and $e_k=0$ if $k<0$.  The subalgebra $\osym_n\subseteq\opol_n$ generated by the $e_k$'s is the algebra of \emph{untwisted odd symmetric polynomials} in $n$ variables.

The involution
\begin{equation}\label{eqn-theta-defn}
\begin{split}
\theta:\opol_n&\to\opol_n\\
x_i&\mapsto(-1)^{i-1}x_i=:\xt_i
\end{split}\end{equation}
is called \emph{twisting}.  The subalgebra $\osymt_n=\theta(\osym_n)$ is called the algebra of \emph{twisted odd symmetric polynomials} in $n$ variables, and $\et_k=\theta(e_k)$ is the $k$-th \emph{twisted odd elementary symmetric polynomial}.

\begin{rem} The odd elementary symmetric polynomials of \cite{EKL} are our twisted odd elementary symmetric polynomials.  The relations among the $e_k$'s are identical to the relations among the $\et_k$'s.  The map $e_k\mapsto\et_k$ is an isomorphism $\osym_n\to\osymt_n$ as well as between their limits (as considered in \cite{EK,EKL}).

\begin{rem} The subalgebra $\osymt_n\subseteq\opol_n$ is not closed under $d$.  For example, if $n=2$, then $\et_1=x_1-x_2$ but $d(x_1-x_2)=x_1^2-x_2^2$ is not in $\osymt_2$.  If the differential $d$ on $\osym_n$ is carried over to $\osymt_n$ via the isomorphism $\theta$, the resulting differential is not diagrammatically local (see Subsection \ref{subsec-local-d}).\end{rem}

In what follows, we will see that the twin constraints of $\osym_n$-linearity and diagrammatic locality force us to consider both of these variants of odd symmetric polynomials.\end{rem}

The symmetric group $S_n$ acts on $\opol_n$ by permuting indices,
\begin{equation*}
w(x_j)=x_{w(j)}.
\end{equation*}
Let $s_i=(i\quad i+1)$ be the $i$-th simple transposition in $S_n$ and let $w_0$ be the longest element of $S_n$ (with respect to the usual Coxeter length); that is, $w_0(j)=n+1-j$.  While $s_i$ does not preserve the subalgebra $\osym_n\subseteq\opol_n$, the longest element $w_0$ does.  It is an involution of superalgebras and
\begin{equation}
w_0(e_k)=(-1)^{\binom{k}{2}}e_k.
\end{equation}
The equation
\begin{equation}\label{eqn-theta-w0}
(\theta\circ w_0)(f)=(-1)^{(n+1)p(f)}(w_0\circ\theta)(f)
\end{equation}
is easily checked for homogeneous $f$.  It follows that
\begin{equation}
w_0(\et_k)=(-1)^{(n+1)k}(\theta\circ w_0)\theta(\et_k)=(-1)^{(n+1)k+\binom{k}{2}}\et_k.
\end{equation}

For $1\leq i\leq n-1$, define the $i$-th \emph{odd divided difference operator} $\pd_i:\opol_n\to\opol_n$ by
\begin{eqnarray}\label{eqn-divided-difference-action}
&\pd_i(x_j)=\begin{cases}1&j=i,i+1\\0&j\neq i,i+1,\end{cases}\\
&\pd_i(fg)=\pd_i(f)g+(-1)^{p(f)}s_i(f)\pd_i(g).
\end{eqnarray}
\begin{prop}[\cite{EKL}]\label{prop-onh-properties}\begin{enumerate}
\item The operator $\pd_i$ is well defined on $\opol_n$.
\item The twisted odd symmetric polynomials are the joint kernel and the joint image of the operators $\pd_i$,
\begin{equation}
\osymt_n=\bigcap_{i=1}^{n-1}\mathrm{Ker}(\pd_i)=\bigcap_{i=1}^{n-1}\mathrm{Im}(\pd_i).
\end{equation}
\item The operator $\pd_i$ is right $\osymt_n$-linear.
\item Considering $x_i$ as left multiplication, the following relations hold in $\END_{\osymt_n}(\opol_n)$ (right-linear endomorphisms):
\begin{eqnarray}
&x_ix_j+x_jx_i=0\text{ if }i\neq j,\\
&\pd_i\pd_j+\pd_j\pd_i=0\text{ if }|i-j|>1,\\
&\pd_i^2=0,\\
&\pd_i\pd_{i+1}\pd_i=\pd_{i+1}\pd_i\pd_{i+1},\\
&\pd_ix_j+x_j\pd_i=0\text{ if }j\neq i,i+1,\\
&x_i\pd_i+\pd_ix_{i+1}=1=\pd_ix_i+x_{i+1}\pd_i.
\end{eqnarray}
These relations suffice to determine all relations among these operators, and these operators generate all of $\END_{\osymt_n}(\opol_n)$.
\end{enumerate}\end{prop}
When we think of the algebra generated by the $x_i,\pd_i$ and the above relations abstractly (rather than acting on $\opol_n$), we will call it the \emph{odd nilHecke algebra} on $n$ strands, denoted by $\onh_n$.

The terminology of ``strands'' comes from the fact that $\onh_n$ can be expressed as a diagrammatic algebra in the sense of Subsection \ref{subsec-diagrammatic}.  A dot on the $i$-th strand from the left stands for $x_i$, so that $\opol_n\subseteq\onh_n$ is a subalgebra, and a crossing of the $i$-th and $(i+1)$-st strands from the left stands for $\pd_i$.  Diagrammatically, then, the defining relations of $\onh_n$ are as follows:
\begin{equation}
\hackcenter{\begin{tikzpicture}[scale=0.5]
    \draw[thick] (0,0) [out=90, in=-90] to (1,1) [out=90, in=-90] to (0,2);
    \draw[thick] (1,0) [out=90, in=-90] to (0,1) [out=90, in=-90] to (1,2);
\end{tikzpicture}}
\quad=\;\; 0,\qquad
\hackcenter{\begin{tikzpicture}[scale=0.5]
    \draw[thick] (0,0) [out=90, in=-90] to (2,2);
    \draw[thick] (1,0) [out=90, in=-90] to (0,1) [out=90, in=-90] to (1,2);
    \draw[thick] (2,0) [out=90, in=-90] to (0,2);
\end{tikzpicture}}
\quad=\quad
\hackcenter{\begin{tikzpicture}[scale=0.5]
    \draw[thick] (0,0) [out=90, in=-90] to (2,2);
    \draw[thick] (1,0) [out=90, in=-90] to (2,1) [out=90, in=-90] to (1,2);
    \draw[thick] (2,0) [out=90, in=-90] to (0,2);
\end{tikzpicture}}
\quad,
\end{equation}
\begin{equation}
\hackcenter{\begin{tikzpicture}[scale=0.75]
    \draw[thick] (0,0) .. controls (0,.5) and (1,.5) .. (1,1);
    \draw[thick] (1,0) .. controls (1,.5) and (0,.5) .. (0,1)
        node[pos=.75] () {\bbullet};
\end{tikzpicture}}
\quad+\quad
\hackcenter{\begin{tikzpicture}[scale=0.75]
    \draw[thick] (0,0) .. controls (0,.5) and (1,.5) .. (1,1);
    \draw[thick] (1,0) .. controls (1,.5) and (0,.5) .. (0,1)
        node[pos=.25] () {\bbullet};
\end{tikzpicture}}
\quad=\quad
\hackcenter{\begin{tikzpicture}[scale=0.75]
    \draw[thick] (0,0) -- (0,1);
    \draw[thick] (1,0) -- (1,1);
\end{tikzpicture}}
\quad=\quad
\hackcenter{\begin{tikzpicture}[scale=0.75]
    \draw[thick] (0,0) .. controls (0,.5) and (1,.5) .. (1,1)
        node[pos=.25] () {\bbullet};
    \draw[thick] (1,0) .. controls (1,.5) and (0,.5) .. (0,1);
\end{tikzpicture}}
\quad+\quad
\hackcenter{\begin{tikzpicture}[scale=0.75]
    \draw[thick] (0,0) .. controls (0,.5) and (1,.5) .. (1,1)
        node[pos=.75] () {\bbullet};
    \draw[thick] (1,0) .. controls (1,.5) and (0,.5) .. (0,1);
\end{tikzpicture}}
\quad;
\end{equation}
plus the distant supercommutativity relations:
\begin{eqnarray}
&\hackcenter{\begin{tikzpicture}[scale=0.5]
    \draw[thick] (0,0) -- (0,2)
        node[pos=.75] () {\bbullet};
    \node () at (1,1) {$\cdots$};
    \draw[thick] (2,0) -- (2,2)
        node[pos=.25] () {\bbullet};
\end{tikzpicture}}
\quad+\quad
\hackcenter{\begin{tikzpicture}[scale=0.5]
    \draw[thick] (0,0) -- (0,2)
        node[pos=.25] () {\bbullet};
    \node () at (1,1) {$\cdots$};
    \draw[thick] (2,0) -- (2,2)
        node[pos=.75] () {\bbullet};
\end{tikzpicture}}
\quad=\;\;0,\\
&\hackcenter{\begin{tikzpicture}[scale=0.5]
    \draw[thick] (0,0) -- (0,1) [out=90, in=-90] to (1,2);
    \draw[thick] (1,0) -- (1,1) [out=90, in=-90] to (0,2);
    \node () at (2,1) {$\cdots$};
    \draw[thick] (3,0) [out=90, in=-90] to (4,1) -- (4,2);
    \draw[thick] (4,0) [out=90, in=-90] to (3,1) -- (3,2);
\end{tikzpicture}}
\quad+\quad
\hackcenter{\begin{tikzpicture}[scale=0.5]
    \draw[thick] (3,0) -- (3,1) [out=90, in=-90] to (4,2);
    \draw[thick] (4,0) -- (4,1) [out=90, in=-90] to (3,2);
    \node () at (2,1) {$\cdots$};
    \draw[thick] (0,0) [out=90, in=-90] to (1,1) -- (1,2);
    \draw[thick] (1,0) [out=90, in=-90] to (0,1) -- (0,2);
\end{tikzpicture}}
\quad= \;\; 0,\\
&\hackcenter{\begin{tikzpicture}[scale=0.5]
    \draw[thick] (0,0) -- (0,1) [out=90, in=-90] to (1,2);
    \draw[thick] (1,0) -- (1,1) [out=90, in=-90] to (0,2);
    \node () at (2,1) {$\cdots$};
    \draw[thick] (3,0) -- (3,2)
        node[pos=.25] () {\bbullet};
\end{tikzpicture}}
\quad+\quad
\hackcenter{\begin{tikzpicture}[scale=0.5]
    \draw[thick] (0,0) [out=90, in=-90] to (1,1) -- (1,2);
    \draw[thick] (1,0) [out=90, in=-90] to (0,1) -- (0,2);
    \node () at (2,1) {$\cdots$};
    \draw[thick] (3,0) -- (3,2)
        node[pos=.75] () {\bbullet};
\end{tikzpicture}}
= \;\; 0,\\
&\hackcenter{\begin{tikzpicture}[scale=0.5]
    \draw[thick] (1,0) -- (1,2)
        node[pos=.75] () {\bbullet};
    \node () at (2,1) {$\cdots$};
    \draw[thick] (3,0) [out=90, in=-90] to (4,1) -- (4,2);
    \draw[thick] (4,0) [out=90, in=-90] to (3,1) -- (3,2);
\end{tikzpicture}}
\quad+\quad
\hackcenter{\begin{tikzpicture}[scale=0.5]
    \draw[thick] (1,0) -- (1,2)
        node[pos=.25] () {\bbullet};
    \node () at (2,1) {$\cdots$};
    \draw[thick] (3,0) -- (3,1) [out=90, in=-90] to (4,2);
    \draw[thick] (4,0) -- (4,1) [out=90, in=-90] to (3,2);
\end{tikzpicture}}
=\;\;0.
\end{eqnarray}
See \cite{EKL} for further details on $\onh_n$.  These algebras are the simplest cases of the quiver Hecke superalgebras of \cite{KKT}.

Once and for all, fix a reduced expression $w=s_{i_1}\cdots s_{i_r}$ for each $w\in S_n$ \emph{except for $w=w_0$} and set
\begin{equation}
\pd_w=\pd_{i_1}\cdots\pd_{i_r}.
\end{equation}
A different choice of reduced expression for $w$ can only change the operator $\pd_w$ by an overall sign.  In the case $w=w_0$, we insist on the choice
\begin{equation}\label{eqn-defn-pd-w0}
\pd_{w_0}=\pd_1(\pd_2\pd_1)\cdots(\pd_{n-1}\cdots\pd_1).
\end{equation}
Let $\delta=(n-1,n-2,\ldots,2,1)$.
\begin{lem}\label{lem-delta}\begin{enumerate}
\item If $f$ is parity homogeneous, then $fx^\delta=(-1)^{\binom{n-1}{2}p(f)}x^\delta\theta(f)$.
\item $\pd_{w_0}(x^\delta)=(-1)^{\binom{n}{3}}$.
\item For any skew polynomial $f$, the equation $\pd_{w_0}f\pd_{w_0}=\pd_{w_0}(f)\pd_{w_0}$ holds in $\onh_n$.
\item The element $\tte_n:=(-1)^{\binom{n}{3}}\pd_{w_0}x^\delta$ of $\onh_n$ is an idempotent.
\item The sets
\begin{equation}\label{eqn-onh-pbw}
\lbrace x^a\pd_w:a\in\Z_{\geq0}^n,w\in S_n\rbrace\text{ and }\lbrace \pd_wx^a:a\in\Z_{\geq0}^n,w\in S_n\rbrace
\end{equation}
are both $\Bbbk$-bases of $\onh_n$.
\end{enumerate}\end{lem}
\begin{proof} Statements 1 and 2 are straightforward to check.  Statement 3 follows because an element of $\onh_n$ is determined by how it acts on $\opol_n$ (Proposition \ref{prop-onh-properties}, Statement 4) and $\mathrm{Im}(\pd_{w_0})\subseteq\osymt_n$ (Proposition \ref{prop-onh-properties}, Statement 2):
\begin{equation*}
(\pd_{w_0}f\pd_{w_0})(g)=\pd_{w_0}(f\pd_{w_0}(g))=\pd_{w_0}(f)\pd_{w_0}(g)=\left(\pd_{w_0}(f)\pd_{w_0}\right)(g).
\end{equation*}
Statement 4 follows from Statements 2 and 3:
\begin{equation*}
\tte_n^2=\pd_{w_0}x^\delta\pd_{w_0}x^\delta=\pd_{w_0}(x^\delta)\pd_{w_0}x^\delta=(-1)^{\binom{n}{3}}\pd_{w_0}x^\delta=\tte_n.
\end{equation*}
Statement 5 is proved in \cite{EKL} (it follows from the relations and a computation using odd Schubert polynomials).
\end{proof}

\begin{lem}\label{lem-idempotent-commute} If $f$ is untwisted odd symmetric, then the equation
\begin{equation}\label{eqn-idempotent-commute}
\tte_nf=(\theta\circ w_0)(f)\tte_n
\end{equation}
holds in $\onh_n$.\end{lem}
\begin{proof} By equation (2.64) of \cite{EKL}, $\pd_{w_0}f=(-1)^{\binom{n}{2}}w_0(f)\pd_{w_0}$.  Along with Statement 1 of Lemma \ref{lem-delta} and \eqref{eqn-theta-w0}, the result follows.\end{proof}

As a summary of this subsection, the superalgebra involutions $w_0$, $\theta$, and $\iota=\iota_{\opol_n}$ of $\opol_n$ satisfy the relations
\begin{equation*}\begin{split}
&w_0^2=\theta^2=\iota^2=1,\\
&\iota\circ w_0=w_0\circ\iota,\quad \iota\circ\theta=\theta\circ\iota,\\
&w_0\circ\theta=\iota^{n+1}\circ\theta\circ w_0.
\end{split}\end{equation*}

\section{Differential graded thin calculus}\label{sec-local-d}

\subsection{Local differentials on skew polynomials}\label{subsec-local-d}

\subsubsection{The differential on $\opol_n$}

Let $\opol_n$ be the superalgebra of skew polynomials defined in Example \ref{eg-odd-pol}.  The local differential
\begin{equation}
d\left(
\hackcenter{\begin{tikzpicture}[scale=0.5]
    \draw[thick] (0,0) -- (0,2)
        node[pos=.5] () {\bbullet};
\end{tikzpicture}}
\right)=
\hackcenter{\begin{tikzpicture}[scale=0.5]
    \draw[thick] (0,0) -- (0,2)
        node[pos=.33] () {\bbullet}
        node[pos=.67] () {\bbullet};
\end{tikzpicture}}
\end{equation}
makes $\opol_n$ into a diagrammatically local dg algebra.  Unless otherwise specified, this is the differential we will use on $\opol_n$ \emph{when we consider it as a dg algebra}.  Note that
\begin{equation*}
d\left(
\hackcenter{\begin{tikzpicture}[scale=0.5]
    \draw[thick] (0,0) -- (0,2)
        node[pos=.5] () {\bbullet}
        node[pos=.5,left] () {\small$i$};
\end{tikzpicture}}
\right)=\lbrace i\rbrace
\hackcenter{\begin{tikzpicture}[scale=0.5]
    \draw[thick] (0,0) -- (0,2)
        node[pos=.5] () {\bbullet}
        node[pos=.5,left] () {\small$i+1$};
\end{tikzpicture}}
\quad.
\end{equation*}

Any dg module structure on the rank-one free left $\opol_n$-module is determined by the value of $d$ on the cyclic vector.  For any $\alpha\in\Bbbk^n$, let
\begin{equation}\label{eqn-onealpha}
d_\alpha(\onealpha)=\sum_{i=1}^n\alpha_ix_i\onealpha,
\end{equation}
where $\onealpha$ stands for the cyclic vector.  So if $f$ is any skew polynomial, \eqref{eqn-left-dg-module} and \eqref{eqn-onealpha} imply
\begin{equation}\label{eqn-d-onealpha}
d_\alpha(f\onealpha)=\left(d(f)+(-1)^{p(f)}f\sum_{i=1}^n\alpha_ix_i\right)\onealpha.
\end{equation}
This choice will give a dg module structure if and only if $d^2(\onealpha)=0$.
\begin{equation*}\begin{split}
d^2(\onealpha)&=\sum_{i=1}^n\alpha_id(x_i\onealpha)=\sum_{i=1}^n\alpha_i(d(x_i)\onealpha-x_id(\onealpha))\\
&=\sum_{i=1}^n(\alpha_i-\alpha_i^2)x_i^2\onealpha-\sum_{i<j}\alpha_i\alpha_j(x_ix_j+x_jx_i)\onealpha\\
&=\sum_{i=1}^n(\alpha_i-\alpha_i^2)x_i^2\onealpha.
\end{split}\end{equation*}
Writing $\opol_n(\alpha)$ for the pair $(\opol_n,d_\alpha)$, we have proved the following.

\begin{prop}\label{prop-d-0-1}
 $\opol_n(\alpha)$ is a left (respectively, right) dg module over $\opol_n$ if and only if $\alpha\in\lbrace0,1\rbrace^n$ (respectively, $\lb0,(-1)^{p(\onealpha)+1}\rb$). \hfill $\square$
\end{prop}

In what follows, we will usually drop the subscript $\alpha$ from $d_\alpha$ when there is no ambiguity.

\subsubsection{The differential on $\osym_n$}

The primary reason we prefer the untwisted odd symmetric polynomials to the twisted ones is the following result which tells us that they are preserved under the local differential on $\opol_n$.

\begin{lem} \label{lem-d-elementary-functions}
The differential $d$ on $\opol_n$ restricts to a differential on $\osym_n$ which acts on untwisted elementary polynomials as
$$d(e_k)=e_1e_k-\lbrace k+1\rbrace e_{k+1}$$
(with it understood that $e_{n+1}=0$).  Consequently, the natural inclusion $\osym_n\subset \opol_n$ is an inclusion of dg algebras.
\end{lem}
\begin{proof} Adjoin a supercentral variable $t$ of odd parity, set $d(t)=0$, and define the generating function
\begin{equation}\label{eqn-elem-generating-function}
E(t)=\prod_{i=1}^n(1+x_it)=\sum_{k=0}^n(-1)^{\binom{k}{2}}e_kt^k.
\end{equation}
We compute
\begin{equation*}\begin{split}
d(E(t))&=\sum_{i=1}\left[\prod_{j<i}(1+x_jt)\cdot x_i^2t\cdot\prod_{j>i}(1+x_jt)\right]\\
&=\sum_{i=1}^n\left[x_i\prod_{j=1}^n(1+x_jt)-\prod_{j<i}(1+x_jt)\cdot x_i\cdot\prod_{j>i}(1+x_jt)\right]\\
&=e_1E(t)-\sum_{k=0}^n(-1)^{\binom{k}{2}}\lbrace k+1\rbrace e_{k+1}t^k.
\end{split}\end{equation*}
Comparing with equation \eqref{eqn-elem-generating-function}, it follows that
\begin{equation}\label{eqn-d-elementary}
d(e_k)=e_1e_k-\lbrace k+1\rbrace e_{k+1}.
\end{equation}
In particular, $\osym_n$ is closed under $d$, finishing the proof of the lemma.
\end{proof}

\subsection{The differential on $\onh_n$ and the bimodule $Z_n$}\label{subsec-zoidberg}

It was proved in \cite{EKL} that the skew polynomial representation gives an isomorphism
\begin{equation*}
\onh_n\cong\END_{\osymt_n}(\opol_n),
\end{equation*}
where elements of $\onh_n$ are regarded as right $\osymt_n$-linear operators.  If $\opol_n$ is equipped with a dg module structure $\opol_n(\alpha)$ of \eqref{eqn-d-onealpha}, this isomorphism induces a differentials on $\onh_n$ via equations \eqref{eqn-d-end} and \eqref{eqn-d-end-element}.  The possible dg algebra structures which can arise on $\onh_n$ in this way are parameterized by the possible choices for $\alpha$ (see Proposition \ref{prop-d-0-1}).

Now we look for conditions on $\alpha$ such that the induced differential on $\onh_n$ is diagrammatically local. The discussion below is similar to \cite[Section 3.1]{KQ}, but some extra care is needed here since the dg algebras involved are noncommutative.

We start by regarding $\onh_n$ as a subalgebra inside $\END_{\Bbbk}(\opol_n(\alpha))$, ignoring the right $\osymt_n$-linearity. The induced differential on the endomorphism algebra acts on each of the polynomial generators $x_i\in \onh_n$ by $d(x_i)=x_i^2$. For degree reasons, the induced differential acts on the odd divided difference operators by
\[
d(\partial_i)=a_i + b_ix_i\partial_i+c_i x_{i+1}\partial_i \quad (i=1,\dots, n-1),
\]
which is diagrammatically depicted by
\begin{equation*}
d\left(
\hackcenter{\begin{tikzpicture}[scale=.75]
    \draw[thick] (0,0) [out=90, in=-90] to (1,1);
    \draw[thick] (1,1) -- (1,1.5);
    \node[below] at (0,0) () {\small$i$};
    \draw[thick] (1,0) [out=90, in=-90] to (0,1);
    \draw[thick] (0,1) -- (0,1.5);
    \node[below] at (1,0) () {\small$i+1$};
\end{tikzpicture}}
\right)=a_i
\hackcenter{\begin{tikzpicture}[scale=.75]
    \draw[thick] (0,0) [out=90, in=-90] to (0,1);
    \draw[thick] (0,1) -- (0,1.5);
    \node[below] at (0,0) () {\small$i$};
    \draw[thick] (1,0) [out=90, in=-90] to (1,1);
    \draw[thick] (1,1) -- (1,1.5);
    \node[below] at (1,0) () {\small$i+1$};
\end{tikzpicture}}
+b_i
\hackcenter{\begin{tikzpicture}[scale=.75]
    \draw[thick] (0,0) [out=90, in=-90] to (1,1);
    \draw[thick] (1,1) -- (1,1.5);
    \node[below] at (0,0) () {\small$i$};
    \draw[thick] (1,0) [out=90, in=-90] to (0,1);
    \draw[thick] (0,1) -- (0,1.5)
        node[pos=0] {\bbullet};
    \node[below] at (1,0) () {\small$i+1$};
\end{tikzpicture}}
+c_i
\hackcenter{\begin{tikzpicture}[scale=.75]
    \draw[thick] (0,0) [out=90, in=-90] to (1,1);
    \draw[thick] (1,1) -- (1,1.5)
        node[pos=0] {\bbullet};
    \node[below] at (0,0) () {\small$i$};
    \draw[thick] (1,0) [out=90, in=-90] to (0,1);
    \draw[thick] (0,1) -- (0,1.5);
    \node[below] at (1,0) () {\small$i+1$};
\end{tikzpicture}},
\end{equation*}
for certain constants $a_i,b_i,c_i\in \Z$. Diagrammatic locality reduces us to the case $n=2$, in which there is only one odd divided difference operator $\partial\in \onh_2$. We will suppose the dg module generator satisfies
\begin{equation}\label{eqn-d-one-rank-2}
d(\1)=\alpha x \1+\beta y \1,
\end{equation}
$\alpha,\beta\in \{0,1\}$, and the induced differential on $\onh_2$ is
\begin{equation}\label{eqn-d-partial-rank-2}
d(\partial)=a + bx\partial + c y\partial
\end{equation}
for constants $a,b,c\in\Bbbk$.  Here, $x,y$ are understood to be the odd polynomial generators. Since $\partial\cdot \1=0$, we must have
\[
0=d(\partial)\cdot \1-\partial \cdot d(\1).
\]
Plugging in equations \eqref{eqn-d-one-rank-2} and \eqref{eqn-d-partial-rank-2}, we obtain the constraint
\begin{equation}\label{eqn-constraint-1}
a=\alpha+\beta.
\end{equation}
Likewise, differentiating the equation
$
\partial(x\1)= \1
$
gives rise to the conditions
\begin{equation}\label{eqn-constraint-2}
a+b=1, \quad c=\alpha+\beta-1.
\end{equation}
Finally, differentiating
$
\partial(y\1)=\1
$
and using \eqref{eqn-d-one-rank-2} and \eqref{eqn-d-partial-rank-2}, we get
\begin{equation}\label{eqn-constraint-3}
a+c=1,\quad b=\alpha+\beta-1.
\end{equation}
Solving equations \eqref{eqn-constraint-1}, \eqref{eqn-constraint-2}, and \eqref{eqn-constraint-3} together results in
\begin{equation}
a=\alpha+\beta=1, \quad\textrm{and}\quad b=c=0.
\end{equation}
This discussion gives us half of the proof for the following result.

\begin{prop} \label{prop-local-diff-parameter-dependence}
If the differential on $\onh_n$ induced by its skew polynomial representation $\opol_n(\alpha)$ is diagrammatically local, then
\begin{equation*}
d\left(
\hackcenter{\begin{tikzpicture}[scale=.75]
    \draw[thick] (0,0) [out=90, in=-90] to (1,1);
    \draw[thick] (1,0) [out=90, in=-90] to (0,1);
\end{tikzpicture}}
\right)=
\hackcenter{\begin{tikzpicture}[scale=.75]
    \draw[thick] (0,0) [out=90, in=-90] to (0,1);
    \draw[thick] (1,0) [out=90, in=-90] to (1,1);
\end{tikzpicture}}~,
\end{equation*}
and either $\alpha=(0,1,0,1,\ldots)$ or $\alpha=(1,0,1,0,\ldots)$.
\end{prop}

\begin{proof} It suffices to consider the case $n=2$.  It remains to show that, under the assumption of the proposition, the equation $d(\partial)=1$ is satisfied on a $\Z$-basis of $\opol_2\cdot \1$.  We use the basis $\{x^ay^b\1:a,b\in \Z_{\geq0}\}$, where the differential acts on $\1$ by either
\[
d(\1)=x\1 \quad \textrm{or} \quad d(\1)=y\1.
\]
This will be done in three steps.

\emph{Step I.} We begin by checking the consistency of the differential on the part of the basis above which does not involve powers of $y$, namely $\{x^a\1:a\in \Z_{\geq0}\}$. This further depends on how $d$ acts on $\1$. We need to use
\begin{equation}\label{eqn-aciton-partial-monomials}
\begin{gathered}
\partial(x^a\1)=
\left\{
\begin{array}{ll}
\sum_{i=0}^{2k-1} (-1)^i x^{2k-1-i}y^i \1& a=2k,\\
&\\
\sum_{i=0}^{2k}x^{2k-i}y^i\1 & a=2k+1,
\end{array}
\right.
\end{gathered}
\end{equation}
where $k\in \Z_{\geq0}$, which follows from an easy computation. We show one example computation, the sub-case $d(\1)=x\1$ and $a=2k$. We check that
\[
d(\partial) (x^{2k}\1)=d(\partial(x^{2k}\1))+\partial ( d(x^{2k}\1))
\]
matches with the action $1\cdot x^{2k}\1=x^{2k}\1$, for any $k\in \Z_{\geq0}$. We compute using \eqref{eqn-aciton-partial-monomials}:
\begin{equation*}\begin{split}
d(\partial) (x^{2k}\1) & = d\left(\sum_{k=0}^{2k-1}(-1)^ix^{2k-1-i}y^i\1\right)+\partial(d(x^{2k}\1))\\
& =  \sum_{i=0}^{2k-1}\left((-1)^i\{2k-1-i\}x^{2k-i}y^i+(-1)^{2k-1}\{i\}x^{2k-1-i}y^{i+1}-x^{2k-i}y^i\right)\cdot \1 \\
& \qquad + \partial(x^{2k+1}\1)\\
& =  \sum_{i=0}^{2k-1}\{i+1\}x^{2k-i}y^i \1 - \sum_{i=1}^{2k}\{i-1\}x^{2k-i}y^i \1 -\sum_{i=0}^{2k-1}x^{2k-i}y^i \1 \\
& \qquad + \sum_{i=0}^{2k}x^{2k-i}y^i \1\\
& =  x^{2k}\1 +\sum_{i=1}^{2k-1}(1-1)\{i+1\}x^{2k-i}y^i \1 - \{2k-1\}y^{2k}\1+y^{2k}\1\\
& =  x^{2k}\1.
\end{split}\end{equation*}
In the third equality, we used the equality $(-1)^i\{i+1\}=\{i+1\}$ (valid for all $i\in \Z$). A similar computation establishes
\[
d(\partial) (x^{2k+1}\1)=d(\partial(x^{2k+1}\1))+\partial ( d(x^{2k+1}\1))
\]
with the same $d$-action on $\1$. For the case $d(\1)=y\1$, the proof is similar.

\vspace{0.2in}

\emph{Step II.} The equation $d(\partial)=1$ also holds on the basis elements in $\{y^a\1:a\in \Z_{\geq0}\}$: to see this, conjugate Step I by the involution of $\onh_n$ given by reflection of diagrams about a vertical axis.  Here one needs to use the fact that conjugation by this reflection interchanges the $d$-action on the two different generators and preserves the equality $d(\partial)=1$. Alternatively, one can check $d(\partial)=1$ directly using computations analogous to those from Step I.

\vspace{0.2in}

\emph{Step III.} Finally, we show that $d(\partial)=1$ holds on all of $\{x^ay^b\1:a,b\in \Z_{\geq0}\}$. We rewrite any $x^ay^b\1$ as $\pm e_2^{m}x^n\1$ or $\pm e_2^my^n\1$, where $m=\min(a,b)$ and $n=\max(a,b)-m$.
\begin{equation*}\begin{split}
d(\partial)(e_2^mx^n \1) & =   d(\partial(e_2^mx^n\1))+\partial(d(e_2^mx^n\1))\\
& =  d\left(\partial(e_2^m)x^n\1+(-1)^me_2^m\partial(x^n\1)\right)+\partial\left(d(e_2^m)x^n\1+e_2^md(x^n\1)\right)\\
& =  (-1)^md(e_2^m)\partial(x^n\1)+(-1)^me_2^md(\partial(x^n\1))+\partial(d(e_2^m))x^n\1\\
& \qquad -s(d(e_2^m))\partial(x^n\1) + s(e_2^m)\partial(d(x^n\1))\\
& = (-1)^m\{m\}e_1e_2^m\partial(x^n\1)+(-1)^me_2^md(\partial(x^n\1))+\partial(\{m\}e_1e_2^m)x^n\1\\
& \qquad -s(\{m\}e_1e_2^m)\partial(x^n\1)+(-1)^me_2^m\partial(d(x^n\1))\\
& =  2\{m\}e_2^mx^n\1+(-1)^me_2^md(\partial(x^n\1))+(-1)^me_2^m\partial(d(x^n\1))\\
& =  (2\{m\}+(-1)^m)e_2^mx^n\1\\
& =  e_2^mx^n\1,
\end{split}\end{equation*}
where $s$ stands for the permutation interchanging $x$ and $y$ (see \eqref{eqn-divided-difference-action}).  We have used Step I in the second to last equality. Likewise, one checks that $d(\pd)$ acts as the identity on any element of the form $e_2^my^n\1$ using Step II.
\end{proof}


By the above discussion, then, this gives rise to a family of dg algebra structures on $\onh_n$ (one for each $\alpha\in\lbrace0,1\rbrace^n$).  Only two of these are diagrammatically local, and we will always assume we are using one of these two.  As a left $\osym_n$- or $\osymt_n$-module, a convenient basis of $\opol_n$ is the ``staircase basis''
\begin{equation*}
\lbrace x^a\onealpha:a\in\Z_{\geq0}^n,0\leq a_i\leq n-i\rbrace.
\end{equation*}
We draw the element $x^a\onealpha$ as
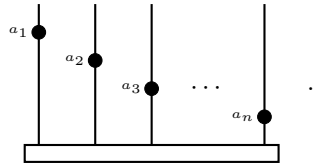
\begin{equation*}
\hackcenter{\begin{tikzpicture}[scale=0.75]
    \draw[thick] (1,0) -- (1,2.5)
        node[pos=.8] () {\bbullet}
        node[pos=.8, left] () {\tiny$a_1$};
    \draw[thick] (2,0) -- (2,2.5)
        node[pos=.6] () {\bbullet}
        node[pos=.6, left] () {\tiny$a_2$};
    \draw[thick] (3,0) -- (3,2.5)
        node[pos=.4] () {\bbullet}
        node[pos=.4, left] () {\tiny$a_3$};
    \node at (4,1) {$\cdots$};
    \draw[thick] (5,0) -- (5,2.5)
        node[pos=.2] () {\bbullet}
        node[pos=.2, left] () {\tiny$a_n$};
    \draw[fill=white, thick] (.75,-.3) rectangle (5.25,0);
\end{tikzpicture}}
\quad.
\end{equation*}
The box at the bottom stands for $\onealpha$; the purpose is to remind us that this is an element of a module rather than an algebra.

If $A$ is a dg algebra and $\tte\in A$ is an idempotent satisfying $d(\tte)=f\tte$ for some $f\in A$, then $A\tte$ with the same differential is a left dg $A$-module.  It follows from the results of \cite{EKL} that $\onh_n\tte_n\cong\opol_n$ as left $\onh_n$-modules (hence too as left $\opol_n$-modules).  We now verify that $d(\tte_n)\in\onh_n\tte_n$ for the differential $d(\pd_i)=1$.

\begin{lem}\label{lem-d-longest} The differential of the longest odd divided difference operator is given by
\begin{equation}
d(\pd_{w_0})=\sum_{i=1}^n\lbrace i-1\rbrace x_i\pd_{w_0}-(-1)^{\binom{n}{2}}\sum_{i=1}^n\lbrace n-i\rbrace\pd_{w_0}x_i.
\end{equation}\end{lem}
\begin{proof} When we want to indicate the number of strands $m$ in our notation for the longest odd divided difference operator, we will write $\pd_{w_0}^{(m)}$.  Our specific choice of reduced expression for $w_0$, expressed diagrammatically, reads
\begin{equation}\label{eqn-pd-w0-inductive}
\hackcenter{\begin{tikzpicture}[scale=0.5]
   \draw[thick] (1,0) -- (1,4);
   \draw[thick] (2,0) -- (2,4);
   \draw[thick] (3,0) -- (3,4);
   \draw[thick] (4,0) -- (4,4);
   \draw[thick] (5,0) -- (5,4);
   \node[draw, thick, fill=white!20,rounded corners=4pt,inner sep=3pt,text width=70] at (3,2) {\vphantom{$\pd_{w_0}^{(n)}$}};
   \node at (3,2) {$\pd_{w_0}^{(n)}$};
\end{tikzpicture}}
=
\hackcenter{\begin{tikzpicture}[scale=0.5]
   \draw[thick] (1,0) [out=90, in=-90] to (5,2) -- (5,4);
   \draw[thick] (2,0) [out=90, in=-90] to (1,2) -- (1,4);
   \draw[thick] (3,0) [out=90, in=-90] to (2,2) -- (2,4);
   \draw[thick] (4,0) [out=90, in=-90] to (3,2) -- (3,4);
   \draw[thick] (5,0) [out=90, in=-90] to (4,2) -- (4,4);
   \node[draw, thick, fill=white!20,rounded corners=4pt,inner sep=3pt,text width=56] at (2.5,2.7) {\vphantom{$\pd_{w_0}^{(n-1)}$}};
   \node at (2.5,2.7) {$\pd_{w_0}^{(n-1)}$};
\end{tikzpicture}}\quad.
\end{equation}
Bottom multiply this equation by $x_1$ and repeatedly apply the nilHecke relation $\pd_ix_i=1-x_{i+1}\pd_i$ to obtain the equation (for $n\geq1$)
\begin{equation}\label{eqn-d-pd-w0-lemma}
\hackcenter{\begin{tikzpicture}[scale=0.5]
   \draw[thick] (1,0) -- (1,4);
   \draw[thick] (2,0) [out=90, in=-90] to (5,2) -- (5,4);
   \draw[thick] (3,0) [out=90, in=-90] to (2,2) -- (2,4);
   \draw[thick] (4,0) [out=90, in=-90] to (3,2) -- (3,4);
   \draw[thick] (5,0) [out=90, in=-90] to (4,2) -- (4,4);
   \node[draw, thick, fill=white!20,rounded corners=4pt,inner sep=3pt,text width=56] at (2.5,2.7) {\vphantom{$\pd_{w_0}^{(n-1)}$}};
   \node at (2.5,2.7) {$\pd_{w_0}^{(n-1)}$};
\end{tikzpicture}}
=
\hackcenter{\begin{tikzpicture}[scale=0.5]
   \draw[thick] (1,0) -- (1,4)
        node[pos=.15] {\bbullet};
   \draw[thick] (2,0) -- (2,4);
   \draw[thick] (3,0) -- (3,4);
   \draw[thick] (4,0) -- (4,4);
   \draw[thick] (5,0) -- (5,4);
   \node[draw, thick, fill=white!20,rounded corners=4pt,inner sep=3pt,text width=70] at (3,2) {\vphantom{$\pd_{w_0}^{(n)}$}};
   \node at (3,2) {$\pd_{w_0}^{(n)}$};
\end{tikzpicture}}
-(-1)^{\binom{n}{2}}
\hackcenter{\begin{tikzpicture}[scale=0.5]
   \draw[thick] (1,0) -- (1,4);
   \draw[thick] (2,0) -- (2,4);
   \draw[thick] (3,0) -- (3,4);
   \draw[thick] (4,0) -- (4,4);
   \draw[thick] (5,0) -- (5,4)
        node[pos=.85] {\bbullet};
   \node[draw, thick, fill=white!20,rounded corners=4pt,inner sep=3pt,text width=70] at (3,2) {\vphantom{$\pd_{w_0}^{(n)}$}};
   \node at (3,2) {$\pd_{w_0}^{(n)}$};
\end{tikzpicture}}\quad.
\end{equation}
Now to prove the stated formula for $d(\pd_{w_0})$.  The formula is easily verified for $n=1$, so assume $n\geq2$ and proceed by induction on $n$.  By \eqref{eqn-pd-w0-inductive} and the Leibniz rule,
\begin{equation*}
d\left(\quad\hackcenter{\begin{tikzpicture}[scale=0.35]
   \draw[thick] (1,0) -- (1,4);
   \draw[thick] (2,0) -- (2,4);
   \draw[thick] (3,0) -- (3,4);
   \draw[thick] (4,0) -- (4,4);
   \draw[thick] (5,0) -- (5,4);
   \draw[thick] (6,0) -- (6,4);
   \node[draw, thick, fill=white!20,rounded corners=4pt,inner sep=3pt,text width=50] at (3.5,2) {\vphantom{$\pd_{w_0}^{(n+1)}$}};
   \node at (3,2) {$\pd_{w_0}^{(n+1)}$};
\end{tikzpicture}}\quad\right)
=
\substack{d\left(\quad
\hackcenter{\begin{tikzpicture}[scale=0.35]
   \draw[thick] (1,0) -- (1,4);
   \draw[thick] (2,0) -- (2,4);
   \draw[thick] (3,0) -- (3,4);
   \draw[thick] (4,0) -- (4,4);
   \draw[thick] (5,0) -- (5,4);
   \node[draw, thick, fill=white!20,rounded corners=4pt,inner sep=3pt,text width=40] at (3,2) {\vphantom{$\pd_{w_0}^{(n)}$}};
   \node at (3,2) {$\pd_{w_0}^{(n)}$};
   \draw[thick] (6,0) -- (6,4);
\end{tikzpicture}}\quad\right)\\
\hackcenter{\begin{tikzpicture}[scale=0.35]
   \draw[thick] (1,0) [out=90, in=-90] to (6,4);
   \draw[thick] (2,0) [out=90, in=-90] to (1,4);
   \draw[thick] (3,0) [out=90, in=-90] to (2,4);
   \draw[thick] (4,0) [out=90, in=-90] to (3,4);
   \draw[thick] (5,0) [out=90, in=-90] to (4,4);
   \draw[thick] (6,0) [out=90, in=-90] to (5,4);
\end{tikzpicture}}}
+(-1)^{\binom{n}{2}}
\substack{
\hackcenter{\begin{tikzpicture}[scale=0.35]
   \draw[thick] (1,0) -- (1,4);
   \draw[thick] (2,0) -- (2,4);
   \draw[thick] (3,0) -- (3,4);
   \draw[thick] (4,0) -- (4,4);
   \draw[thick] (5,0) -- (5,4);
   \node[draw, thick, fill=white!20,rounded corners=4pt,inner sep=3pt,text width=40] at (3,2) {\vphantom{$\pd_{w_0}^{(n)}$}};
   \node at (3,2) {$\pd_{w_0}^{(n)}$};
   \draw[thick] (6,0) -- (6,4);
\end{tikzpicture}}\\
d\left(\quad\hackcenter{\begin{tikzpicture}[scale=0.35]
   \draw[thick] (1,0) [out=90, in=-90] to (6,4);
   \draw[thick] (2,0) [out=90, in=-90] to (1,4);
   \draw[thick] (3,0) [out=90, in=-90] to (2,4);
   \draw[thick] (4,0) [out=90, in=-90] to (3,4);
   \draw[thick] (5,0) [out=90, in=-90] to (4,4);
   \draw[thick] (6,0) [out=90, in=-90] to (5,4);
\end{tikzpicture}}\quad\right)}
\end{equation*}
The second term on the right is an alternating sum over deleted crossings from the bottom factor.  Since any crossing among strands $1$ through $n$ adjacent to $\pd_{w_0}^{(n)}$ equals the zero diagram, only the term in which the leftmost crossing is deleted survives.  Using this and the inductive hypothesis, the above equals
\begin{equation*}
\ldots=\sum_{i=1}^n\left[\lbrace i-1\rbrace
\hackcenter{\begin{tikzpicture}[scale=0.3]
   \draw[thick] (1,0) [out=90, in=-90] to (6,3.5) -- (6,6);
   \draw[thick] (2,0) [out=90, in=-90] to (1,3) -- (1,6);
   \draw[thick] (3,0) [out=90, in=-90] to (2,3) -- (2,6);
   \draw[thick] (4,0) [out=90, in=-90] to (3,3) -- (3,6)
        node[pos=.85] {\bbullet}
        node[pos=1,above] {\small$i$};
    \draw (4,0) -- (4,0)
        node[pos=0,below] {\vphantom{\small$i$}};
   \draw[thick] (5,0) [out=90, in=-90] to (4,3) -- (4,6);
   \draw[thick] (6,0) [out=90, in=-90] to (5,3) -- (5,6);
   \node[draw, thick, fill=white!20,rounded corners=4pt,inner sep=3pt,text width=56] at (3,4) {\vphantom{$\pd_{w_0}^{(n)}$}};
   \node at (3,4) {$\pd_{w_0}^{(n)}$};
\end{tikzpicture}}
-(-1)^{\binom{n}{2}}\lbrace n-i\rbrace
\hackcenter{\begin{tikzpicture}[scale=0.3]
   \draw[thick] (1,0) [out=90, in=-90] to (6,3.5) -- (6,6);
   \draw[thick] (2,0) [out=90, in=-90] to (1,3) -- (1,6);
   \draw[thick] (3,0) [out=90, in=-90] to (2,3) -- (2,6);
   \draw[thick] (4,0) [out=90, in=-90] to (3,3);
   \node at (3.1,2.4) {\bbullet};
   \draw[thick] (3,3) -- (3,6)
        node[pos=1,above] {\vphantom{\small$i+1$}};
   \node[below] at (4,0) {\small$i+1$};
   \draw[thick] (5,0) [out=90, in=-90] to (4,3) -- (4,6);
   \draw[thick] (6,0) [out=90, in=-90] to (5,3) -- (5,6);
   \node[draw, thick, fill=white!20,rounded corners=4pt,inner sep=3pt,text width=56] at (3,4) {\vphantom{$\pd_{w_0}^{(n)}$}};
   \node at (3,4) {$\pd_{w_0}^{(n)}$};
\end{tikzpicture}}
\right]+(-1)^{\binom{n}{2}+n-1}
\hackcenter{\begin{tikzpicture}[scale=0.3]
   \draw[thick] (1,0) -- (1,6);
   \draw[thick] (2,0) [out=90, in=-90] to (6,3.5) -- (6,6);
   \draw[thick] (3,0) [out=90, in=-90] to (2,3) -- (2,6);
   \draw[thick] (4,0) [out=90, in=-90] to (3,3) -- (3,6);
   \draw[thick] (5,0) [out=90, in=-90] to (4,3) -- (4,6);
   \draw[thick] (6,0) [out=90, in=-90] to (5,3) -- (5,6);
   \node[draw, thick, fill=white!20,rounded corners=4pt,inner sep=3pt,text width=56] at (3,4) {\vphantom{$\pd_{w_0}^{(n)}$}};
   \node at (3,4) {$\pd_{w_0}^{(n)}$};
\end{tikzpicture}}
\ .
\end{equation*}
In the second summation, the $\pd_{w_0}^{(n)}$ allows the dot to slide to the bottom (with an accrued factor of $(-1)^{n-1}$), and the only extra term created is in the case $i=1$.  Accounting for boundary summation terms and using \eqref{eqn-d-pd-w0-lemma}, the result follows.
\end{proof}

\begin{lem}\label{lem-d-idempotent} The differential of the idempotent $\tte_n$ is given by
\begin{equation}\label{eqn-d-idempotent}
d(\tte_n)=\sum_{i=1}^n\lbrace i-1\rbrace x_i\tte_n.
\end{equation}\end{lem}
\begin{proof} It is easy to check that
\begin{equation}\label{eqn-d-x-delta}
d(x^\delta)=\sum_{i=1}^n\lbrace n-i\rbrace x_ix^\delta.
\end{equation}
Using this and Lemma \ref{lem-d-longest}, we compute
\begin{equation*}\begin{split}
(-1)^{\binom{n}{3}}d(\tte_n)&=d(\pd_{w_0})x^\delta+(-1)^{\binom{n}{2}}\pd_{w_0}d(x^\delta)\\
&=\left[\sum_{i=1}^n\lbrace i-1\rbrace x_i\pd_{w_0}-(-1)^{\binom{n}{2}}\sum_{i=1}^n\lbrace n-i\rbrace\pd_{w_0}x_i\right]x^\delta+(-1)^{\binom{n}{2}}\sum_{i=1}^n\lbrace n-i\rbrace\pd_{w_0}x_ix^\delta\\
&=(-1)^{\binom{n}{3}}\sum_{i=1}^n\lbrace i-1\rbrace x_i\tte_n.
\end{split}
\end{equation*}
The result follows.
\end{proof}

\begin{cor} As left dg $\opol_n$-modules, $\onh_n\tte_n\cong\opol_n(0,1,0,1,\ldots)$.\hfill$\square$\end{cor}

For $f\in\osym_n$, \eqref{eqn-idempotent-commute} implies $(\theta\circ w_0)(f)\in\osymt_n$.  It follows that the left action of $\onh_n$ on $\onh_n\tte_n$ commutes with the right action of $\osym_n$ on $\opol_n\tte_n$ (right multiplication).  Unlike $\osymt_n$, the algebra $\osym_n$ is a diagrammatically local dg subalgebra of $\opol_n$.  We have proven the following.

\begin{prop}\label{prop-zoidberg-as-opol} The subspace $\opol_n\tte_n=\lbrace f\tte_n:f\in\opol_n\rbrace$ is a dg $(\opol_n,\osym_n)$-bimodule.  As a left dg module it is isomorphic to $\opol_n(0,1,0,1,\ldots)$.  The right module structure is given by $g\tte_n\cdot f=g(\theta\circ w_0)(f)\tte_n$ for $f\in\osym_n$, $g\in\opol_n$.\hfill$\square$
\end{prop}
Although it occurs naturally as an idempotented part of $\onh_n$, it will be convenient to work with the above dg module via generators and relations.
\begin{defn}\label{def-zoidberg-bim} Let $Z_n$ be the dg $(\opol_n,\osym_n)$-bimodule
\begin{equation}\label{eqn-zoidberg}
Z_n:=\opol_n\onez,
\end{equation}
where $Z_n$ is a rank-one free left module with a cyclic vector $\onez$.  The right module structure is defined by
\begin{equation}\label{eqn-onez}
\onez f=(\theta\circ w_0)(f)\onez\text{ for }f\in\osym_n.
\end{equation}
The differential is given by
\begin{equation}\label{eqn-d-onez}
d(\onez)=\sum_{i=1}^n\lbrace i-1\rbrace x_i\onez.
\end{equation}
\end{defn}

We can summarize the results of this subsection as follows.
\begin{cor}\label{cor-onh-end} As dg algebras,
\begin{equation}
(\onh_n,d)\cong\END_{\osym_n^\mathrm{op}}(Z_n)
\end{equation}
The differential on $\onh_n$ induced by this isomorphism is
\begin{equation}
d(x_i)=x_i^2,\quad d(\pd_i)=1.
\end{equation}
Or, diagrammatically,
\begin{equation}
d\left(~
\hackcenter{\begin{tikzpicture}[scale=0.5]
    \draw[thick] (0,0) -- (0,2)
        node[pos=.5] () {\bbullet};
\end{tikzpicture}}
~\right)=
\hackcenter{\begin{tikzpicture}[scale=0.5]
    \draw[thick] (0,0) -- (0,2)
        node[pos=.33] () {\bbullet}
        node[pos=.67] () {\bbullet};
\end{tikzpicture}}
,\quad
d\left(~
\hackcenter{\begin{tikzpicture}[scale=0.5]
    \draw[thick] (0,0) [out=90, in=-90] to (2,2);
    \draw[thick] (2,0) [out=90, in=-90] to (0,2);
\end{tikzpicture}}
~\right)=
\hackcenter{\begin{tikzpicture}[scale=0.5]
    \draw[thick] (0,0) -- (0,2);
    \draw[thick] (1,0) -- (1,2);
\end{tikzpicture}}\quad.
\end{equation}\hfill$\square$
\end{cor}

\subsection{The differential of an odd Schur polynomial}\label{subsec-schur}

In this subsection, we use the structure of the bimodule $Z_n$ to compute the value of the differential $d$ on an odd Schur polynomial.  The reader interested only in thin calculus and applications to the small quantum group $u^+_{\sqrt{-1}}(\sltwo)$ can skip this subsection; its results will not be needed until Section \ref{sec-thick}.

To each partition $\lambda$ of length at most $n$, we associate an \emph{untwisted odd Schur polynomial}
\begin{equation}\label{eqn-untwisted-schur}
s_\lambda=w_0\theta\pd_{w_0}(\xt^\delta\xt^\lambda)
\end{equation}
and a \emph{twisted odd Schur polynomial}
\begin{equation}\label{eqn-twisted-schur}
\st_\lambda=(-1)^{\binom{n}{3}+\binom{n}{2}|\lambda|}w_0\pd_{w_0}(x^\lambda x^\delta).
\end{equation}

\begin{rem} Our twisted odd Schur polynomial $\st_\lambda$ is the odd Schur polynomial of \cite{EKL}.  The sign discrepancy of $(-1)^{\binom{n}{2}|\lambda|}$ between formulae for the two comes from the fact that $S_n$ acts by the the permutation representation here but by the tensor product of the permutation representation and the degree-th power of the sign representation in \cite{EKL}.\end{rem}

Using \eqref{eqn-theta-w0} and the computation
\begin{equation}\label{eqn-delta-lambda}\begin{split}
x^\lambda x^\delta&=(-1)^{\binom{n}{2}|\lambda|}(-1)^{n|\lambda|+\sum j\lambda_j}x^\delta x^\lambda\\
&=(-1)^{\binom{n-1}{2}|\lambda|+\binom{n}{3}}\xt^\delta\xt^\lambda,
\end{split}\end{equation}
it is easy to check that
\begin{equation}\label{eqn-untwisted-twisted-schur}
\st_\lambda=\theta(s_\lambda).
\end{equation}
Combining this with \eqref{eqn-theta-w0}, we have two forms of the following ``Schur-through-$\onez$,'' or ``SZ relation'':
\begin{equation}\label{eqn-sz}\begin{split}
\onez s_\lambda&=(-1)^{(n+1)|\lambda|}w_0(\st_\lambda)\onez,\\
\st_\lambda\onez&=\onez w_0(s_\lambda).
\end{split}\end{equation}

Some notation:
\begin{itemize}
\item For any partition $\lambda$ and any $1\leq i\leq\ell(\lambda)$, the partition $\frac{\lambda}{i}$ (respectively, $\frac{i}{\lambda}$) is $\lambda$ with rows $i$ and below (respectively, $i$ and above) removed.
\item If $B$ is a box in some Young diagram in row $i$ and column $j$, the \emph{content} of $B$ is defined to be $\ct(B)=j-i$.
\item For partitions $\lambda$ and $\mu$, the relation ``$\mu=\lambda+\boxi$'' means that $\mu$ is obtained from $\lambda$ by adding a box in the $i$-th row.  In this situation, $\ct(\boxi)$ means the content of the newly added box.
\end{itemize}
\begin{prop}\label{prop-d-schur} The differential of an untwisted odd Schur polynomial is given by
\begin{equation}\label{eqn-d-schur}
d(s_\lambda)=\sum_{\mu=\lambda+\boxi}(-1)^{\lvert\frac{\lambda}{i}\rvert+i-1}\lbrace\ct(\boxi)\rbrace s_\mu.
\end{equation}\end{prop}
\begin{proof} From the SZ relation \eqref{eqn-sz} and the definition of twisted odd Schur polynomials, we have
\begin{equation*}
\onez s_\lambda=(-1)^{\binom{n}{3}+|\lambda|+\sum j\lambda_j}\pd_{w_0}(x^\delta x^\lambda)\onez.
\end{equation*}
The differential of the left-hand side is
\begin{equation}\label{eqn-schur-lhs}\begin{split}
d(\onez s_\lambda)&=d(\onez)s_\lambda+\onez d(s_\lambda)\\
&=\sum_{i=1}^n\lbrace i-1\rbrace x_i\onez s_\lambda+\onez d(s_\lambda).
\end{split}\end{equation}
The differential of the right-hand side is $(-1)^{\binom{n}{3}+|\lambda|+\sum j\lambda_j}$ times
\begin{equation}\label{eqn-schur-rhs}\begin{split}
d(\pd_{w_0}(x^\delta x^\lambda)\onez)&=d(\pd_{w_0}x^\delta)x^\lambda\onez+\pd_{w_0}x^\delta d(x^\lambda)\onez+(-1)^{|\lambda|}\pd_{w_0}x^\delta x^\lambda d(\onez)\\
&=\sum_{i=1}^n\lbrace i-1\rbrace x_i\pd_{w_0}x^\delta x^\lambda\onez+\sum_{i=1}^{\ell(\lambda)}(-1)^{\lvert\frac{\lambda}{i}\rvert}\lbrace\lambda_i\rbrace\pd_{w_0}x^\delta x^{\lambda+\boxi}\onez\\
&\qquad\qquad+(-1)^{|\lambda|}\sum_{i=1}^n\lbrace i-1\rbrace\pd_{w_0}x^\delta x^\lambda x_i\onez.
\end{split}\end{equation}
The second equality used Lemma \ref{lem-d-idempotent} in the first term.  Again using SZ \eqref{eqn-sz}, the first terms in \eqref{eqn-schur-lhs} and \eqref{eqn-schur-rhs} are equal, yielding the equation
\begin{equation*}\begin{split}
\onez d(s_\lambda)&=(-1)^{\binom{n}{3}+\sum j\lambda_j}\left[(-1)^{|\lambda|}\sum_{i=1}^{\ell(\lambda)}(-1)^{\lvert\frac{\lambda}{i}\rvert}\lbrace\lambda_i\rbrace\pd_{w_0}x^\delta x^{\lambda+\boxi}\onez+\sum_{i=1}^n\lbrace i-1\rbrace(-1)^{\lvert\frac{i}{\lambda}\rvert}\pd_{w_0}x^\delta x^{\lambda+\boxi}\onez\right]\\
&=(-1)^{\binom{n}{3}+\sum j\lambda_j+|\lambda|}\sum_{\mu=\lambda+\boxi}(-1)^{\lvert\frac{\lambda}{i}\rvert}\left(\lbrace\lambda_i\rbrace+(-1)^{\lambda_i}\lbrace i-1\rbrace\right)\pd_{w_0}x^\delta x^\mu\onez\\
&=(-1)^{\binom{n}{3}+\sum j\lambda_j+|\lambda|}\sum_{\mu=\lambda+\boxi}(-1)^{\lvert\frac{\lambda}{i}\rvert}\lbrace\ct(\boxi)\rbrace\pd_{w_0}x^\delta x^\mu\onez\\
&=\sum_{\mu=\lambda+\boxi}(-1)^{\lvert\frac{\lambda}{i}\rvert+i-1}\lbrace\ct(\boxi)\rbrace\onez s_\mu.
\end{split}\end{equation*}
The second equality follows because $\pd_{w_0}(x^\delta x^\nu)=0$ unless $\nu$ is a partition.  The third follows from the easily verified identity $\lbrace\lambda_i\rbrace+(-1)^{\lambda_i}\lbrace i-1\rbrace=\lbrace\ct(\boxi)\rbrace$.  The fourth follows from the SZ relation \eqref{eqn-sz}.
\end{proof}

We emphasize that the above is a calculation of $d(s_\lambda)$ in the dg algebra $\osym_n$ (or the dg algebra $\opol_n$).  Although the computation made use of the $\onh_n$-action on the bimodule $Z_n$, the result is independent of the module.

\subsection{Categorification of $u^+_{\sqrt{-1}}(\sltwo)$}\label{subsec-cat-small}

Define operators $\ddxi:\opol_n\to\opol_n$ by
\begin{equation}\begin{split}
\ddxi(x_j)&=\delta_{ij},\\
\ddxi(fg)&=\ddxi(f)g+(-1)^{p(f)}f\ddxi(g)
\end{split}\end{equation}
for $f,g$ parity homogeneous and $1\leq i\leq n$ (odd analogues of partial derivatives).  Note that
\begin{equation*}
\ddxi(x_j^k)=\delta_{ij}\lb k\rb x_j^{k-1}.
\end{equation*}
By setting $\ddxi(\onealpha)=0$, these operators act on $\opol_n(\alpha)$ as well.
\begin{lem}\label{lem-odd-derivs}\begin{enumerate}
\item The operators $\lb\ddxi:1\leq i\leq n\rb$ generate an exterior algebra:
\begin{equation}
\left(\ddxi\right)^2=0,\qquad \ddxi\ddxj+\ddxj\ddxi=0.
\end{equation}
\item The differential on the dg algebra $\opol_n$ is given by
\begin{equation}\label{eqn-d-ddxi}
d=\sum_{i=1}^nx_i^2\ddxi.
\end{equation}
\end{enumerate}\end{lem}
\begin{proof} Both are easy calculations from the definitions.\end{proof}

\begin{rem} Equation \eqref{eqn-d-ddxi} holds on $\opol_n$ but only holds on $\opol_n(\alpha)$ if and only if $\alpha=(0,\ldots,0)$.\end{rem}

For any $\beta=(\beta_1,\ldots,\beta_n)\in\Bbbk^n$, let
\begin{equation}
h_\beta=\sum_{i=1}^n\beta_i\ddxi
\end{equation}
be the corresponding odd directional derivative operator.
\begin{lem} Let $\langle\cdot,\cdot\rangle$ denote the standard inner product on $\Bbbk^n$.  Then
\begin{equation}\label{eqn-nulhomotopy}
h_\beta d+dh_\beta=\langle\alpha,\beta\rangle
\end{equation}
as operators on $\opol_n(\alpha)$.\end{lem}
\begin{proof} Using the definition of $h_\beta$ and Lemma \ref{lem-odd-derivs}, we compute
\begin{equation*}\begin{split}
dh_\beta(f\onealpha)&=d\left(\sum_i\beta_i\ddxi(f)\onealpha\right)\\
&=\sum_{i,j}\beta_ix_j^2\ddxj\ddxi(f)\onealpha-(-1)^{p(f)}\sum_{i,j}\alpha_j\beta_i\ddxi(f)x_j\onealpha,\\
h_\beta d(f\onealpha)&=h_\beta\left(\sum_ix_i^2\ddxi(f)\onealpha+(-1)^{p(f)}\sum_i\alpha_ifx_i\onealpha\right)\\
&=\sum_{i,j}\beta_j\ddxj x_i^2\ddxi(f)\onealpha+(-1)^{p(f)}\sum_{i,j}\alpha_i\beta_j\ddxj(f)x_i\onealpha\\
&\qquad+\sum_i\alpha_i\beta_if\onealpha.
\end{split}\end{equation*}
The second summations of $dh_\beta$ and $h_\beta d$ cancel.  The calculation
\begin{equation*}
\ddxi x_j^2\ddxj=x_j^2\ddxi\ddxj=-x_j^2\ddxj\ddxi
\end{equation*}
shows that the first summations do as well.  The remaining term equals $\langle\alpha,\beta\rangle f\onealpha$.
\end{proof}

\begin{cor} If $\alpha\in\lb0,1\rb^n$ is not identically $0$, then the identity map on $\opol_n(\alpha)$ is null-homotopic.\end{cor}
\begin{proof} If $\alpha_i\neq0$, then for $\beta_j=\delta_{ij}$, $h_\beta$ is a null-homotopy of the identity map.\end{proof}

In particular, if $n\geq2$, then $\ddxtwo$ is a null-homotopy of the identity map of the dg bimodule $Z_n$ (considered as a left $\opol_n$-module).

For $n\geq2$, the dg algebra $\onh_n$ is itself acyclic.  One easy way to see this is to appeal to Corollary \ref{cor-acyclic-dga} and express the identity element as a coboundary: $1=d(\pd_1)$.  Diagrammatically,
\begin{equation*}
d\left(~
\hackcenter{\begin{tikzpicture}
    \draw[thick] (0,0) [out=90, in=-90] to (.5,1);
    \draw[thick] (.5,0) [out=90, in=-90] to (0,1);
    \draw[thick] (1,0) -- (1,1);
    \node at (1.5,.5) () {$\cdots$};
    \draw[thick] (2,0) -- (2,1);
\end{tikzpicture}}
~\right)\quad=\quad
\hackcenter{\begin{tikzpicture}
    \draw[thick] (0,0) -- (0,1);
    \draw[thick] (.5,0) -- (.5,1);
    \draw[thick] (1,0) -- (1,1);
    \node at (1.5,.5) () {$\cdots$};
    \draw[thick] (2,0) -- (2,1);
\end{tikzpicture}}
\quad.
\end{equation*}
Another way to see the acyclicity of $\onh_n$ is to give a more explicit formula for the differential coming from the isomorphism of Corollary \ref{cor-onh-end}.  The set
\begin{equation*}
B'_n=\lbrace x^a\onez:0\leq i\leq n-i\rbrace
\end{equation*}
is a basis for $Z_n$ as a right $\osym_n$-module.  The differential acts as
\begin{equation}\label{eqn-d-u}
\begin{split}
d(x^a\onez)&=\sum_i\left(\lb a_i\rb+(-1)^{a_i}\lb i-1\rb\right)x_ix^a\onez\\
&=\sum_i\lb a_i+i-1\rb x_ix^a\onez
\end{split}
\end{equation}
on elements of $B'_n$.  In particular $d(x_i^{n-i}\onez)=0$, so $d$ preserves
\begin{equation*}
U_n=\text{span}_\Bbbk(B'_n).
\end{equation*}
Thus
\begin{equation}\label{eqn-Zoidberg-cohomology}
Z_n\cong U_n\otimes_\Bbbk\osym_n
\end{equation}
and
\begin{equation*}
\onh_n\cong U_n^*\otimes_\Bbbk U_n\otimes_\Bbbk\osym_n
\end{equation*}
as right dg $\osym_n$-modules.  It is easy to check that $U_n$ is an acyclic chain complex for $n\geq2$; hence another proof that $\onh_n$ is an acyclic dg algebra.  We summarize the above as:
\begin{prop}\label{prop-Zoidberg-finite-cell}
\begin{enumerate}
\item For any $n\in \Z_{\geq0}$, $Z_n$ is a finite-cell right dg module over $\osym_n$.
\item If $n\geq2$, then $\onh_n$ an acyclic dg algebra. Consequently, the derived category $\mc{D}(\onh_n)$ is equivalent to the zero category whenever $n\geq 2$. \hfill$\square$
\end{enumerate}\end{prop}

In contrast to the above result, the module $Z_n$ is almost never cofibrant when considered as a left dg module over $\onh_n$.

\begin{prop} The following are equivalent:
\begin{enumerate}
    \item As a left dg module over $\onh_n$, $Z_n$ is cofibrant.
    \item As a chain complex, $Z_n$ is not acyclic.
    \item $n=0$ or $n=1$.
\end{enumerate}\end{prop}
\begin{proof} When $n=0$ or $1$, $Z_n\cong\onh_n$ as the left regular module, so the result is clear. When $n\geq 2$, it follows from equation \eqref{eqn-Zoidberg-cohomology} and Proposition \ref{prop-Zoidberg-finite-cell} that both $Z_n$ and $\onh_n$ are acyclic. It suffices to show that, for these $n$, $Z_n$ is not cofibrant as a left module over $\onh_n$. Consider the canonical multiplication map
\[
\onh_n\otimes Z_n\lra Z_n, \quad \xi\otimes f\onez\mapsto \xi(f)\onez,
\]
which is a surjective map of left dg modules over $\onh_n$. When $n\geq 2$, both sides of the equation are acyclic, so this map is a quasi-isomorphism. If $Z_n$ were cofibrant over $\onh_n$, then the multiplication map would admit a section, so that $Z_n$ would be a dg direct summand of $\onh_n\otimes Z_n$. Now, apply the functor $\HOM_{\onh_n}(-,Z_n)$ to the surjective quasi-isomorphism, we get an injection of chain-complexes, which would split if $Z_n$ were cofibrant:
\[
\osym_n\cong \HOM_{\onh_n}(Z_n,Z_n)\lra \HOM_{\onh_n}(\onh_n\otimes Z_n, Z_n)\cong \HOM_{\Bbbk}(Z_n,\Bbbk)\otimes Z_n.
\]
The right-hand side of the last map is acyclic when $n\geq 2$. Since $\osym_n$ is never acyclic (see Proposition \ref{prop-lima-cohomology}), it cannot be a dg direct summand of an acyclic complex.  Hence for $n\geq2$, $Z_n$ is not cofibrant over $\onh_n$.  The result follows.
\end{proof}
The reader interested in the combinatorics of the finite-cell filtration on $Z_n$ may consult the appendix.

Let
\[
\mc{D}(\onh):=\bigoplus_{n\in \Z_{\geq0}}\mc{D}(\onh_n)
\]
be the direct sum of the derived categories of dg modules over $\onh_n$ for all $n\geq0$. The inclusion map
\begin{equation}
\iota_{m,n}:\onh_m\otimes\onh_n\hookrightarrow\onh_{m+n}
\end{equation}
gives us derived induction and restriction functors on the dg derived categories:
\begin{equation}\label{eqn-induction-component}
\smind_{m,n}:=(\iota_{m,n})^* :\mc{D}(\onh_{m}\otimes\onh_n)\lra \mc{D}(\onh_{n+m}),
\end{equation}
\begin{equation}\label{eqn-restriction-component}
\smres_{m,n}:=(\iota_{m,n})_* :\mc{D}(\onh_{n+m})\lra \mc{D}(\onh_{m}\otimes\onh_n).
\end{equation}
Summing over all $n$ and $m$, we obtain functors
\begin{equation}\label{eqn-induction}
\smind:=\bigoplus_{m,n\in\Z_{\geq0}}\smind_{m,n}:\mc{D}(\onh\otimes \onh)\lra \mc{D}(\onh),
\end{equation}
\begin{equation}\label{eqn-restriction}
\smres:=\bigoplus_{m,n\in\Z_{\geq0}}\smres_{m,n}:\mc{D}(\onh)\lra \mc{D}(\onh\otimes \onh).
\end{equation}

By the discussion in Subsection \ref{subsect-categorifying-Gaussian}, the Grothendieck group
\begin{equation*}
K_0(\onh):=\bigoplus_{n\in \Z_{\geq0}}K_0(\mc{D}(\onh_n))
\end{equation*}
has a natural $\Z[\sqrt{-1}]$-module structure.  Furthermore, by Proposition \ref{prop-nice-property-Zoidberg}, the derived category of $\onh$ has the K\"{u}nneth-type property with respect to $K_0$:
\begin{equation}\label{eqn-small-Kunneth}
K_0(\onh_n\otimes \onh_m)\cong K_0(\onh_n)\otimes_{\Z[\sqrt{-1}]}K_0(\onh_m).
\end{equation}
Hence the symbols $[\smind]$ and $[\smres]$ turn $K_0(\onh)$ into a
$\sqrt{-1}$-bialgebra, as in \cite{EKL}.  Comparing the $\sqrt{-1}$-bialgebras $u^+$ and $K_0(\onh)$ using Subsection \ref{subsec-qgs} and Proposition \ref{prop-nice-property-Zoidberg}, we have the following.

\begin{thm} \label{thm-small-sl2}
The map
\begin{equation*}
u^+\to K_0(\onh), \quad E\mapsto [\onh_1],
\end{equation*}
is an isomorphism of $\sqrt{-1}$-bialgebras.\hfill$\square$
\end{thm}

Finally, we remark that the restriction functor $\smres$ is not quite the optimal one for categorifying $r$. Later, we will see that when we embed $u^+$ into a bigger twisted bialgebra, it is more natural to consider a different restriction functor. The situation has already appeared in the categorification of $u_q^+(\sltwo)$ at a prime root of unity in \cite[Section 3.3]{KQ}.

\section{Differential graded thick calculus}\label{sec-thick}

\subsection{Thick strands and functors}\label{subsec-thick-strands}

The diagrams we have been working with so far are useful for representing elements of $\onh_n$.  Recall that $\onh_n$ has, up to isomorphism and shift, a unique indecomposable projective module $P_n$ and that the regular representation decomposes as
\begin{equation*}
\onh_n\cong P_n^{\oplus[n]!}.
\end{equation*}
The module $P_n$ is just the skew polynomial representation $\opol_n\cong\onh_n\tte_n$ on which $x_i$ acts as multiplication and $\pd_i$ acts as an odd divided difference operator.  The left module structure on the bimodule $Z_n$ of the previous section is isomorphic to $P_n$.

In \cite{EKL}, diagrammatics were introduced for $\onh_n$-linear endomorphisms of $P_n$.  The identity map $P_n\to P_n$ is drawn as a \emph{thick strand},
\begin{equation*}
\hackcenter{\begin{tikzpicture}
    \draw[green!65!black,line width=4] (0,0) -- (0,1)
        node[pos=0, below, black] {\small$n$};
\end{tikzpicture}}=\tte_n.
\end{equation*}
For any $f\in\osymt_n$, we define
\begin{equation*}
\hackcenter{\begin{tikzpicture}
    \draw[green!65!black,line width=4] (0,0) -- (0,2)
        node[pos=0, below, black] {\small$n$};
   \node[draw, thick, fill=white!20,rounded corners=4pt,inner sep=3pt] at (0,1) {$f$};
\end{tikzpicture}}=\tte_nf\tte_n.
\end{equation*}

\begin{rem}
The technique of thick calculus was first introduced in \cite{KLMS} for the (even) nilHecke algebra.  Its odd analogue was studied in \cite{EKL}.  Our conventions differ from those of \cite{EKL} by a reflection about a horizontal axis.  This is because we use the idempotent $\tte_n=(-1)^{\binom{n}{3}}\pd_{w_0}x^\delta$ while the latter work adopts $(-1)^{\binom{n}{3}}x^\delta\pd_{w_0}$.
\end{rem}

Equation (2.64) of \cite{EKL} implies that for $f,g\in\osymt_n$ we have $\pd_{w_0}f=w_0(f)\pd_{w_0}$.  Hence $\tte_nf\tte_ng\tte_n=\tte_nfg\tte_n$.  Diagrammatically, this reads
\begin{equation*}
\hackcenter{\begin{tikzpicture}
    \draw[green!65!black,line width=4] (0,0) -- (0,3)
        node[pos=0, below, black] {\small$n$};
   \node[draw, thick, fill=white!20,rounded corners=4pt,inner sep=3pt] at (0,1) {$f$};
   \node[draw, thick, fill=white!20,rounded corners=4pt,inner sep=3pt] at (0,2) {$g$};
\end{tikzpicture}}\quad=\quad
\hackcenter{\begin{tikzpicture}
    \draw[green!65!black,line width=4] (0,0) -- (0,3)
        node[pos=0, below, black] {\small$n$};
   \node[draw, thick, fill=white!20,rounded corners=4pt,inner sep=3pt] at (0,1.5) {$gf$};
\end{tikzpicture}}.
\end{equation*}

Summarizing the above discussion, thick diagrams with labels are considered as elements of the graded endomorphism algebra $\END_{\onh_n}(P_n)$, which is isomorphic to $\osymt_n$ by Statement 2 of Proposition \ref{prop-onh-properties}, and $P_n$ becomes a $(\onh_n,\osymt_n)$-bimodule. We would like to regard $P_n$ as a right module over the untwisted odd symmetric polynomials $\osym_n$. To do this, we identify $P_n$ with the bimodule $Z_n=\opol_n\onez$ over $(\onh_n,\osym_n)$ (Definition \ref{def-zoidberg-bim}). In other words, the right $\osym_n$ action on $P_n$ comes from the ring isomorphism $\theta\circ w_0:\osym_n\lra \osymt_n$ (see equation \eqref{eqn-onez}).

A thick strand of thickness $1$ can also be drawn as a thin (ordinary) strand.  When convenient, several thin strands will be drawn as one thin strand with a label,
\begin{equation*}
\hackcenter{\begin{tikzpicture}
    \draw[thick] (0,0) -- (0,2);
    \node at (0,-.5) {\small$n$};
\end{tikzpicture}}\quad=\quad
\hackcenter{\begin{tikzpicture}
    \draw[thick] (0,0) -- (0,2);
    \draw[thick] (.5,0) -- (.5,2);
    \node at (1,1) {$\cdots$};
    \draw[thick] (1.5,0) -- (1.5,2);
    \node at (.75,-.5) {\small$n$\text{ strands}};
\end{tikzpicture}}\quad.
\end{equation*}
If there are generators present that make the label redundant, we drop the label.

In a diagram with $a+b$ strands, suppose we bunch the first $a$ together and the last $b$ together.  If we cross these bunches, we are drawing a divided difference operator $\pd_{w_{a,b}}$, where $w_{a,b}(i)=a+i$ for $1\leq i\leq a$ and $w_{a,b}(i)=i-a$ for $a+1\leq i\leq a+b$.  As a matter of convention, we will always use the reduced expression
\begin{equation*}
w_{a,b}=(s_bs_{b-1}\cdots s_1)(s_{b+1}s_a\cdots s_2)\cdots(s_{a+b-1}s_{a+b}\cdots s_a).
\end{equation*}
Diagrammatically,
\begin{equation}\label{eqn-pd-wab}
\hackcenter{\begin{tikzpicture}
    \draw[thick] (0,0) [out=90, in=-90] to (1,1);
    \node[below] at (1,0) {\small$b$};
    \draw[thick] (1,0) [out=90, in=-90] to (0,1);
    \node[below] at (0,0) {\small$a$};
\end{tikzpicture}}
\quad=\quad
\hackcenter{\begin{tikzpicture}[scale=0.7]
    \draw[thick] (0,0) -- (0,1.5) [out=90, in=-90] to (5,4.5);
    \draw[thick] (.5,0) -- (.5,1) [out=90, in=-90] to (5.5,4) -- (5.5,4.5);
    \draw[thick] (1,0) -- (1,.5) [out=90, in=-90] to (6,3.5) -- (6,4.5);
    \draw[thick] (3,0) [out=90, in=-90] to (2,4.5);
    \draw[thick] (3.5,0) [out=90, in=-90] to (2.5,4.5);
    \draw[thick] (4,0) [out=90, in=-90] to (3,4.5);
    \draw[thick] (4.5,0) [out=90, in=-90] to (3.5,4.5);
    \node[below] at (.5,0) {\small$a$};
    \node[below] at (3.75,0) {\small$b$};
\end{tikzpicture}}
\end{equation}
The rules for combining and associating crossings of bunched thin strands involve easily computed signs; see Subsection 3.2.4 of \cite{EKL}.

Let
\begin{equation*}
\osym_{a,b} := \osym_a\otimes_{\Bbbk}\osym_b.
\end{equation*}
For $a,b$ positive integers define \emph{splitters},
\begin{equation}
\hackcenter{\begin{tikzpicture}[scale=0.7]
    \draw[green!65!black,line width=4] (.5,0) -- (.5,1);
    \draw[green!65!black,line width=4] (.5,1) [out=90, in=-90] to (0,2);
    \draw[green!65!black,line width=4] (.5,1) [out=90, in=-90] to (1,2);
    \node at (.5,-.5) {\small$a+b$};
    \node at (0,2.5) {\small$a$};
    \node at (1,2.5) {\small$b$};
\end{tikzpicture}}\quad:=\quad
\hackcenter{\begin{tikzpicture}[scale=0.7]
    \draw[thick] (0,0) -- (0,1);
    \draw[thick] (-.5,1) -- (-.5,3);
    \draw[thick] (.5,1) -- (.5,3);
    \node at (0,-.5) {$a+b$};
    \node at (-.5,3.5) {$a$};
    \node at (.5,3.5) {$b$};
    \node[draw, thick, fill=white!20,rounded corners=4pt,inner sep=3pt] at (0,1) {$\tte_{a+b}$};
    \node[draw, thick, fill=white!20,rounded corners=4pt,inner sep=3pt] at (-.5,2.25) {$\tte_{a}$};
    \node[draw, thick, fill=white!20,rounded corners=4pt,inner sep=3pt] at (.5,2.25) {$\tte_{b}$};
\end{tikzpicture}}\quad,\qquad
\hackcenter{\begin{tikzpicture}[scale=0.7]
    \draw[green!65!black,line width=4] (.5,1) -- (.5,2);
    \draw[green!65!black,line width=4] (0,0) [out=90, in=-90] to (.5,1);
    \draw[green!65!black,line width=4] (1,0) [out=90, in=-90] to (.5,1);
    \node at (.5,2.5) {\small$a+b$};
    \node at (0,-.5) {\small$a$};
    \node at (1,-.5) {\small$b$};
\end{tikzpicture}}\quad:=\quad
\hackcenter{\begin{tikzpicture}[scale=0.7]
    \draw[thick] (0,0) -- (0,1);
    \draw[thick] (1,0) -- (1,1);
    \draw[thick] (0,1) [out=90, in=-90] to (1,2.5);
    \draw[thick] (1,1) [out=90, in=-90] to (0,2.5);
    \draw[thick] (.5,2.5) -- (.5,3.5);
    \node at (0,-.5) {$a$};
    \node at (1,-.5) {$b$};
    \node at (.5,4) {$a+b$};
    \node[draw, thick, fill=white!20,rounded corners=4pt,inner sep=3pt] at (0,.75) {$\tte_{a}$};
    \node[draw, thick, fill=white!20,rounded corners=4pt,inner sep=3pt] at (1,.75) {$\tte_{b}$};
    \node[draw, thick, fill=white!20,rounded corners=4pt,inner sep=3pt] at (.5,2.5) {$\tte_{a+b}$};
\end{tikzpicture}}\quad.
\end{equation}

Splitter diagrams can carry labels by elements of $\osymt_a\otimes_\Bbbk \osymt_b$ on the top legs and labels from $\osymt_{a+b}$ on the bottom. These diagrams are subject to the relation
\[\label{eqn-slider}
(-1)^{\binom{s}{2}}
\hackcenter{
\begin{tikzpicture}[scale=0.7]
    \draw[green!65!black,line width=4] (1,0) -- (1,1);
    \draw[green!65!black,line width=4] (1,1) [out=90, in=-90] to (0,2);
    \draw[green!65!black,line width=4] (1,1) [out=90, in=-90] to (2,2);
    \draw[green!65!black,line width=4] (0,2) -- (0,3);
    \draw[green!65!black,line width=4] (2,2) -- (2,3);
    \node at (1,-.5) {\small$a+b$};
    \node at (0,3.5) {\small$a$};
    \node at (2,3.5) {\small$b$};
    \node[draw, thick, fill=white!20,rounded corners=4pt,inner sep=3pt] at (1,.55) {$\et_{s}$};
\end{tikzpicture}}
=
\sum_{l=0}^s(-1)^{a\ell}\quad
\hackcenter{
\begin{tikzpicture}[scale=0.7]
    \draw[green!65!black,line width=4] (1,0) -- (1,1);
    \draw[green!65!black,line width=4] (1,1) [out=90, in=-90] to (0,2);
    \draw[green!65!black,line width=4] (1,1) [out=90, in=-90] to (2,2);
    \draw[green!65!black,line width=4] (0,2) -- (0,3);
    \draw[green!65!black,line width=4] (2,2) -- (2,3);
    \node at (1,-.5) {\small$a+b$};
    \node at (0,3.5) {\small$a$};
    \node at (2,3.5) {\small$b$};
    \node[draw, thick, fill=white!20,rounded corners=4pt,inner sep=3pt] at (2,2.45) {$\et_{l}$};
    \node[draw, thick, fill=white!20,rounded corners=4pt,inner sep=3pt] at (0,2) {$\et_{s-l}$};
\end{tikzpicture}},
\]
which can be checked using the convenient relation
\begin{equation*}
\tte_n x_1\cdots x_k\tte_n=\et_k\tte_n.
\end{equation*}

In other words, the collection of these diagrams span a $\Bbbk$-vector space isomorphic to $\osymt_a\boxtimes \osymt_b$, considered as a bimodule over the rings $\osym_a\otimes_\Bbbk \osym_b$ and $\osym_{a+b}$, with the untwisted odd symmetric polynomials acting on the module via the $\theta\circ w_0$ twisting as before. Tensor product with this bimodule gives rise to a functor
\begin{equation}
\left(\osymt_a\boxtimes \osymt_b\right)\otimes_{\osym_{a+b}}(-): \osym_{a+b}\dmod \lra \osym_a\otimes\osym_b\dmod.
\end{equation}
The thin diagrams correspond to the explicit embedding of this bimodule inside the space $\onh_{a+b}$, the latter considered as a bimodule over $\onh_a\otimes \onh_b$ and $\onh_{a+b}$. This embedding will help us determine the differential action on the thick generators in the next subsection.

The other splitter diagram also affords a similar functorial interpretation, and we leave it to the reader to spell out the details.

Two special cases of splitter diagrams called \emph{exploders} are defined by
\begin{equation}\label{eqn-defn-exploder}
\hackcenter{\begin{tikzpicture}[scale=0.7]
    \draw[green!65!black,line width=4] (0,2) -- (0,3);
    \draw[thick] (-2,0) [out=90, in=-150] to (0,2);
    \draw[thick] (-1.5,0) [out=90, in=-120] to (0,2);
    \draw[thick] (-1,0) [out=90, in=-90] to (0,2);
    \node at (0,.5) {\small$\cdots$};
    \draw[thick] (1,0) [out=90, in=-90] to (0,2);
    \draw[thick] (1.5,0) [out=90, in=-60] to (0,2);
    \draw[thick] (2,0) [out=90, in=-30] to (0,2);
    \node[above] at (0,3) {\small$n$};
    \node[below] at (0,0) {\vphantom{\small$n$}};
\end{tikzpicture}}
\quad=\quad
\hackcenter{\begin{tikzpicture}[scale=0.7]
    \draw[thick] (0,0) -- (0,3);
    \node[above] at (0,3) {\small$n$};
    \node[below] at (0,0) {\vphantom{\small$n$}};
    \node[draw, thick, fill=white!20,rounded corners=4pt,inner sep=3pt] at (0,1.5) {$\pd_{w_0}$};
\end{tikzpicture}}
\quad,\qquad
\hackcenter{\begin{tikzpicture}[scale=0.7]
    \draw[green!65!black,line width=4] (0,0) -- (0,1);
    \draw[thick] (0,1) [out=150, in=-90] to (-2,3);
    \draw[thick] (0,1) [out=120, in=-90] to (-1.5,3);
    \draw[thick] (0,1) [out=90, in=-90] to (-1,3);
    \node at (0,2.5) {\small$\cdots$};
    \draw[thick] (0,1) [out=90, in=-90] to (1,3);
    \draw[thick] (0,1) [out=60, in=-90] to (1.5,3);
    \draw[thick] (0,1) [out=30, in=-90] to (2,3);
    \node[below] at (0,0) {\small$n$};
    \node[above] at (0,3) {\vphantom{\small$n$}};
\end{tikzpicture}}
\quad=\quad
\hackcenter{\begin{tikzpicture}[scale=0.7]
    \draw[thick] (0,0) -- (0,3);
    \node[below] at (0,0) {\small$n$};
    \node[above] at (0,3) {\vphantom{\small$n$}};
    \node[draw, thick, fill=white!20,rounded corners=4pt,inner sep=3pt] at (0,1.5) {$\tte_n$};
\end{tikzpicture}}
\quad.
\end{equation}
For relations involving exploded subsets of strands, relative heights of exploders, and the associativity of exploders, see Subsection 4.2.1 of \cite{EKL}.  The latter allows us to consider the dg bimodule $Z_n$ as a functor
\begin{equation}
Z_n\otimes_{\osym_n}(-): \osym_n\dmod \lra \opol_n\dmod.
\end{equation}

We have the following relation
\begin{equation*}
\hackcenter{\begin{tikzpicture}[scale=0.7]
    \draw[green!65!black,line width=4] (0,5) -- (0,6);
    \draw[thick] (-2,3) [out=90, in=-150] to (0,5);
    \draw[thick] (-1.5,3) [out=90, in=-120] to (0,5);
    \draw[thick] (-1,3) [out=90, in=-90] to (0,5);
    \node at (0,3.75) {\small$\cdots$};
    \draw[thick] (1,3) [out=90, in=-90] to (0,5);
    \draw[thick] (1.5,3) [out=90, in=-60] to (0,5);
    \draw[thick] (2,3) [out=90, in=-30] to (0,5);
    \node[above] at (0,6) {\vphantom{\small$n$}};
    \node[below] at (0,0) {\small$n$};
    \draw[green!65!black,line width=4] (0,0) -- (0,1);
    \draw[thick] (0,1) [out=150, in=-90] to (-2,3);
    \draw[thick] (0,1) [out=120, in=-90] to (-1.5,3);
    \draw[thick] (0,1) [out=90, in=-90] to (-1,3);
    \node at (0,2.25) {\small$\cdots$};
    \draw[thick] (0,1) [out=90, in=-90] to (1,3);
    \draw[thick] (0,1) [out=60, in=-90] to (1.5,3);
    \draw[thick] (0,1) [out=30, in=-90] to (2,3);
    \node[draw, thick, fill=white!20,rounded corners=4pt,inner sep=3pt] at (0,3) {\hphantom{abcdefgh}$f$\hphantom{abcdefgh}};
\end{tikzpicture}}
\quad=\quad
\hackcenter{\begin{tikzpicture}[scale=0.7]
    \draw[green!65!black,line width=4] (0,0) -- (0,6);
    \node[draw, thick, fill=white!20,rounded corners=4pt,inner sep=3pt] at (0,3) {$\pd_{w_0}(f)$};
    \node[above] at (0,6) {\vphantom{\small$n$}};
    \node[below] at (0,0) {\small$n$};
\end{tikzpicture}}
\quad\text{ for }f\in\opol_n.
\end{equation*}
In particular, if $\lambda$ is a partition with at most $n$ parts, then
\begin{equation*}
\hackcenter{\begin{tikzpicture}[scale=0.7]
    \draw[green!65!black,line width=4] (0,5) -- (0,6);
    \draw[thick] (-2,3) [out=90, in=-150] to (0,5);
    \draw[thick] (-1.5,3) [out=90, in=-120] to (0,5);
    \draw[thick] (-1,3) [out=90, in=-90] to (0,5);
    \node at (0,3.75) {\small$\cdots$};
    \draw[thick] (1,3) [out=90, in=-90] to (0,5);
    \draw[thick] (1.5,3) [out=90, in=-60] to (0,5);
    \draw[thick] (2,3) [out=90, in=-30] to (0,5);
    \node[above] at (0,6) {\vphantom{\small$n$}};
    \node[below] at (0,0) {\small$n$};
    \draw[green!65!black,line width=4] (0,0) -- (0,1);
    \draw[thick] (0,1) [out=150, in=-90] to (-2,3);
    \draw[thick] (0,1) [out=120, in=-90] to (-1.5,3);
    \draw[thick] (0,1) [out=90, in=-90] to (-1,3);
    \node at (0,2.25) {\small$\cdots$};
    \draw[thick] (0,1) [out=90, in=-90] to (1,3);
    \draw[thick] (0,1) [out=60, in=-90] to (1.5,3);
    \draw[thick] (0,1) [out=30, in=-90] to (2,3);
    \node[draw, thick, fill=white!20,rounded corners=4pt,inner sep=3pt] at (0,3) {\hphantom{abcdef}$x^\lambda x^\delta$\hphantom{abcdef}};
\end{tikzpicture}}
\quad=\quad(-1)^{\binom{n}{3}+\binom{n}{2}|\lambda|}
\hackcenter{\begin{tikzpicture}[scale=0.7]
    \draw[green!65!black,line width=4] (0,0) -- (0,6);
    \node[draw, thick, fill=white!20,rounded corners=4pt,inner sep=3pt] at (0,3) {$w_0(\st_\lambda)$};
    \node[above] at (0,6) {\vphantom{\small$n$}};
    \node[below] at (0,0) {\small$n$};
\end{tikzpicture}}\quad.
\end{equation*}
We will use these relations in the next subsection.

\subsection{Odd dg thick calculus}\label{subsec-dg-thick}

\subsubsection{Differential on splitters}

As we have seen in the previous subsection, splitters (and exploders) sit inside idempotented pieces of odd nilHecke algebras. Therefore we can use Lemma \ref{lem-d-endo-idempotent} to compute the action of the differential $d_{\onh}$ on them. In the next subsection, we will give a more intrinsic explanation of these differentials from a dg module-theoretic point of view.

To start, we simply observe that the thick strand with no label on it is identified with the identity endomorphism inside the ring
$$\END_{\onh_n}(P_n)\cong \END_{\onh_n}(\onh_n\tte_n)\cong \tte_n\onh_n\tte_n.$$
By Lemma \ref{lem-d-idempotent},
\begin{equation*}
\tte_nd(\tte_n)=\sum_{i=1}^n\lbrace i-1\rbrace\tte_nx_i\tte_n=\sum_{i=1}^n\lbrace i-1\rbrace\pd_{w_0}\left(x^\delta x_i\right)\pd_{w_0}x^\delta.
\end{equation*}
Unless $i=1$, $\pd_{w_0}\left(x^\delta x_i\right)=0$, and it follows that $\tte_nd(\tte_n)=0$. In general, by Lemma \ref{lem-d-endo-idempotent}, the differential acts on a decorated thick strand as
\begin{equation*}
d(\tte_n\varphi\tte_n)=\tte_nd(\varphi)\tte_n+(-1)^{p(\varphi)}\tte_n\varphi d(\tte_n),
\end{equation*}
where $\varphi\in \osymt_n$.

\begin{prop} \label{prop-diagrammatic-thick-diff} The differentials of the splitters are given by
\begin{eqnarray}
&d\left(\hackcenter{\begin{tikzpicture}[scale=0.7]
    \draw[green!65!black,line width=4] (.5,0) -- (.5,1);
    \draw[green!65!black,line width=4] (.5,1) [out=90, in=-90] to (0,2);
    \draw[green!65!black,line width=4] (.5,1) [out=90, in=-90] to (1,2);
    \node at (.5,-.5) {\small$a+b$};
    \node at (0,2.5) {\small$a$};
    \node at (1,2.5) {\small$b$};
\end{tikzpicture}}\right)\quad=\quad\lbrace a\rbrace
\hackcenter{\begin{tikzpicture}[scale=0.7]
    \draw[green!65!black,line width=4] (.5,0) -- (.5,1);
    \draw[green!65!black,line width=4] (.5,1) [out=90, in=-90] to (0,2);
    \draw[green!65!black,line width=4] (.5,1) [out=90, in=-90] to (1,2);
    \draw[green!65!black,line width=4] (0,2) -- (0,3);
    \draw[green!65!black,line width=4] (1,2) -- (1,3);
    \node at (.5,-.5) {\small$a+b$};
    \node at (0,3.5) {\small$a$};
    \node at (1,3.5) {\small$b$};
    \node[draw, thick, fill=white!20,rounded corners=4pt,inner sep=3pt] at (1,2.25) {$\et_1$};
\end{tikzpicture}}\quad,\\
&d\left(\hackcenter{\begin{tikzpicture}[scale=0.7]
    \draw[green!65!black,line width=4] (.5,1) -- (.5,2);
    \draw[green!65!black,line width=4] (0,0) [out=90, in=-90] to (.5,1);
    \draw[green!65!black,line width=4] (1,0) [out=90, in=-90] to (.5,1);
    \node at (.5,2.5) {\small$a+b$};
    \node at (0,-.5) {\small$a$};
    \node at (1,-.5) {\small$b$};
\end{tikzpicture}}\right)\quad=\quad(-1)^{ab-1}\lbrace b\rbrace
\hackcenter{\begin{tikzpicture}[scale=0.7]
    \draw[green!65!black,line width=4] (.5,2) -- (.5,3);
    \draw[green!65!black,line width=4] (0,1) [out=90, in=-90] to (.5,2);
    \draw[green!65!black,line width=4] (1,1) [out=90, in=-90] to (.5,2);
    \draw[green!65!black,line width=4] (0,0) -- (0,1);
    \draw[green!65!black,line width=4] (1,0) -- (1,1);
    \node at (.5,3.5) {\small$a+b$};
    \node at (0,-.5) {\small$a$};
    \node at (1,-.5) {\small$b$};
    \node[draw, thick, fill=white!20,rounded corners=4pt,inner sep=3pt] at (0,.75) {$\et_1$};
\end{tikzpicture}}\quad.
\end{eqnarray}\end{prop}
Here and elsewhere, putting a coupon labeled by the elementary symmetric polynomial $e_k$ on a subset of the strands means to use the polynomial $e_k$ in those variables.  For example: if $\et_1$ is placed on strands $a+1$ through $b+1$ as above, this stands for the element
\begin{equation*}
x_{a+1}-x_{a+2}+x_{a+3}-\ldots+(-1)^{b-1}x_{a+b}\in\osymt(x_{a+1},\ldots,x_{a+b})\subseteq\opol_{a+b}.
\end{equation*}
\begin{proof} First we introduce some notation.  We work on $a+b$ strands, so $\pd_w$ is the divided difference operator for a permutation $w\in S_{a+b}$.  Let
\begin{itemize}
\item $\pd_w'$ denote the divided difference operator $\pd_w$ on strands $1,\ldots,a$ for $w\in S_a$;
\item $\pd_w''$ denote the divided difference operator $\pd_w$ on strands $a+1,\ldots,a+b$ for $w\in S_b$;
\item $\pd_{a,b}$ be the divided difference operator of \eqref{eqn-pd-wab};
\item $\theta_{a,b}\in\lbrace\pm1\rbrace$ be defined by $\pd_{w_0}=\theta_{a,b}\pd_{a,b}\pd_{w_0}'\pd_{w_0}''$;
\item $x^{\delta'}=x_1^{a-1}x_2^{a-2}\cdots x_{a-1}$, $x^{\delta''}=x_{a+1}^{b-1}x_{a+2}^{b-2}\cdots x_{a+b-1}$.
\end{itemize}
The upwards-opening splitter is given by $\tte_a\tte_b\tte_{a+b}$.  We compute:
\begin{equation*}\begin{split}
d\left(\hackcenter{\begin{tikzpicture}[scale=0.7]
    \draw[green!65!black,line width=4] (.5,0) -- (.5,1);
    \draw[green!65!black,line width=4] (.5,1) [out=90, in=-90] to (0,2);
    \draw[green!65!black,line width=4] (.5,1) [out=90, in=-90] to (1,2);
\end{tikzpicture}}\right)
&\quad=\quad\hackcenter{\begin{tikzpicture}[scale=0.7]
    \draw[thick] (0,0) -- (0,1);
    \draw[thick] (-.5,1) -- (-.5,3);
    \draw[thick] (.5,1) -- (.5,3);
    \node[draw, thick, fill=white!20,rounded corners=4pt,inner sep=3pt] at (0,1) {$d(\tte_{a+b})$};
    \node[draw, thick, fill=white!20,rounded corners=4pt,inner sep=3pt] at (-.5,2.25) {$\tte_{a}$};
    \node[draw, thick, fill=white!20,rounded corners=4pt,inner sep=3pt] at (.5,2.25) {$\tte_{b}$};
\end{tikzpicture}}
\quad\refequal{\eqref{eqn-d-idempotent}}\quad\sum_{i=1}^a\lbrace i-1\rbrace
\hackcenter{\begin{tikzpicture}[scale=0.7]
    \draw[thick] (0,0) -- (0,1);
    \draw[thick] (-.5,1) -- (-.5,4);
    \draw[thick] (.5,1) -- (.5,4);
    \node[draw, thick, fill=white!20,rounded corners=4pt,inner sep=3pt] at (0,1) {$\tte_{a+b}$};
    \node[draw, thick, fill=white!20,rounded corners=4pt,inner sep=3pt] at (-.5,3.25) {$\tte_{a}$};
    \node[draw, thick, fill=white!20,rounded corners=4pt,inner sep=3pt] at (.5,3.25) {$\tte_{b}$};
    \node[draw, thick, fill=white!20,rounded corners=4pt,inner sep=3pt] at (-.5,2.25) {$x_i$};
\end{tikzpicture}}
\quad+\quad\sum_{i=a+1}^{a+b}\lbrace i-1\rbrace
\hackcenter{\begin{tikzpicture}[scale=0.7]
    \draw[thick] (0,0) -- (0,1);
    \draw[thick] (-.5,1) -- (-.5,4);
    \draw[thick] (.5,1) -- (.5,4);
    \node[draw, thick, fill=white!20,rounded corners=4pt,inner sep=3pt] at (0,1) {$\tte_{a+b}$};
    \node[draw, thick, fill=white!20,rounded corners=4pt,inner sep=3pt] at (-.5,3.25) {$\tte_{a}$};
    \node[draw, thick, fill=white!20,rounded corners=4pt,inner sep=3pt] at (.5,3.25) {$\tte_{b}$};
    \node[draw, thick, fill=white!20,rounded corners=4pt,inner sep=3pt] at (.5,2.25) {$x_i$};
\end{tikzpicture}}\\
&=\quad\lbrace a\rbrace
\hackcenter{\begin{tikzpicture}[scale=0.7]
    \draw[thick] (0,0) -- (0,1);
    \draw[thick] (-.5,1) -- (-.5,4);
    \draw[thick] (.5,1) -- (.5,4);
    \node[draw, thick, fill=white!20,rounded corners=4pt,inner sep=3pt] at (0,1) {$\tte_{a+b}$};
    \node[draw, thick, fill=white!20,rounded corners=4pt,inner sep=3pt] at (-.5,3.25) {$\tte_{a}$};
    \node[draw, thick, fill=white!20,rounded corners=4pt,inner sep=3pt] at (.5,3.25) {$\tte_{b}$};
    \node[draw, thick, fill=white!20,rounded corners=4pt,inner sep=3pt] at (.5,2.25) {$x_{a+1}$};
\end{tikzpicture}}
\quad=\quad(-1)^{\binom{b}{3}}\lbrace a\rbrace
\hackcenter{\begin{tikzpicture}[scale=0.7]
    \draw[thick] (0,0) -- (0,1);
    \draw[thick] (-.75,1) -- (-.75,5);
    \draw[thick] (.75,1) -- (.75,5);
    \node[draw, thick, fill=white!20,rounded corners=4pt,inner sep=3pt] at (0,1) {\hphantom{ab}$\tte_{a+b}$\hphantom{ab}};
    \node[draw, thick, fill=white!20,rounded corners=4pt,inner sep=3pt] at (-.75,4.25) {$\tte_{a}$};
    \node[draw, thick, fill=white!20,rounded corners=4pt,inner sep=3pt] at (.75,3.75) {$\pd_{w_0}''$};
    \node[draw, thick, fill=white!20,rounded corners=4pt,inner sep=3pt] at (.75,2.25) {$x^{\delta''}x_{a+1}$};
\end{tikzpicture}}\\
&\refequal{\eqref{eqn-twisted-schur}}\quad\lbrace a\rbrace\quad
\hackcenter{\begin{tikzpicture}[scale=0.5]
    \draw[green!65!black,line width=4] (.5,0) -- (.5,1);
    \draw[green!65!black,line width=4] (.5,1) [out=90, in=-90] to (0,2);
    \draw[green!65!black,line width=4] (.5,1) [out=90, in=-90] to (1,2);
    \draw[green!65!black,line width=4] (0,2) -- (0,3);
    \draw[green!65!black,line width=4] (1,2) -- (1,3);
    \node[draw, thick, fill=white!20,rounded corners=4pt,inner sep=3pt] at (1,2.25) {\small$\et_1$};
\end{tikzpicture}}\quad.
\end{split}\end{equation*}
The third equality deserves comment.  Since $\tte_a$ and $\tte_b$ commute, we have either
\begin{equation*}
\lbrace i-1\rbrace\pd_{w_0}'x^{\delta'}x_i\pd_{w_0}\qquad\text{or}\qquad\lbrace i-1\rbrace\pd_{w_0}''x^{\delta''}x_i\pd_{w_0}
\end{equation*}
as a factor in each term; the former for $1\leq i\leq a$, the latter for $a+1\leq i\leq a+b$.  In either expression, if there are adjacent equal powers $x_j^kx_{j+1}^k$, then the expression is equal to $0$.  So the only potentially nonzero term is the latter, in the case $i=a+1$.

The other differential is computed similarly:
\[
\begin{array}{l}
d\left(\hackcenter{\begin{tikzpicture}[scale=0.7]
    \draw[green!65!black,line width=4] (.5,1) -- (.5,2);
    \draw[green!65!black,line width=4] (0,0) [out=90, in=-90] to (.5,1);
    \draw[green!65!black,line width=4] (1,0) [out=90, in=-90] to (.5,1);
\end{tikzpicture}}\right)
~  = ~
d\left(\hackcenter{\begin{tikzpicture}[scale=0.7]
    \draw[thick] (0,0) -- (0,1);
    \draw[thick] (1,0) -- (1,1);
    \draw[thick] (0,1) [out=90, in=-90] to (1,2.5);
    \draw[thick] (1,1) [out=90, in=-90] to (0,2.5);
    \draw[thick] (.5,2.5) -- (.5,3.5);
    \node[draw, thick, fill=white!20,rounded corners=4pt,inner sep=3pt] at (0,.75) {$\tte_{a}$};
    \node[draw, thick, fill=white!20,rounded corners=4pt,inner sep=3pt] at (1,.75) {$\tte_{b}$};
    \node[draw, thick, fill=white!20,rounded corners=4pt,inner sep=3pt] at (.5,2.5) {$\tte_{a+b}$};
\end{tikzpicture}}\right)
~=~(-1)^{\binom{a}{3}+\binom{b}{3}+\binom{a}{2}\binom{b}{2}}\theta_{a,b}
d\left(\hackcenter{\begin{tikzpicture}[scale=0.7]
    \draw[thick] (-.25,-.5) -- (-.25,1);
    \draw[thick] (-.25,1) [out=90, in=-90] to (0,2.5);
    \draw[thick] (1.25,-.5) -- (1.25,1);
    \draw[thick] (1.25,1) [out=90, in=-90] to (1,2.5);
    \draw[thick] (.5,2.5) -- (.5,4.5);
    \node[draw, thick, fill=white!20,rounded corners=4pt,inner sep=3pt] at (-.25,1.25) {$x^{\delta'}$};
    \node[draw, thick, fill=white!20,rounded corners=4pt,inner sep=3pt] at (1.25,.5) {$x^{\delta''}$};
    \node[draw, thick, fill=white!20,rounded corners=4pt,inner sep=3pt] at (.5,3.5) {$\tte_{a+b}$};
    \node[draw, thick, fill=white!20,rounded corners=4pt,inner sep=3pt] at (.5,2.5) {$\pd_{w_0}$};
\end{tikzpicture}}\right)\\
=~(-1)^{\binom{a}{3}+\binom{b}{3}+\binom{a}{2}\binom{b}{2}}\theta_{a,b}
\left[\hackcenter{\begin{tikzpicture}[scale=0.7]
    \draw[thick] (-.25,-.5) -- (-.25,1);
    \draw[thick] (-.25,1) [out=90, in=-90] to (0,2.5);
    \draw[thick] (1.25,-.5) -- (1.25,1);
    \draw[thick] (1.25,1) [out=90, in=-90] to (1,2.5);
    \draw[thick] (.5,2.5) -- (.5,4.5);
    \node[draw, thick, fill=white!20,rounded corners=4pt,inner sep=3pt] at (-.25,1.25) {$x^{\delta'}$};
    \node[draw, thick, fill=white!20,rounded corners=4pt,inner sep=3pt] at (1.25,.5) {$x^{\delta''}$};
    \node[draw, thick, fill=white!20,rounded corners=4pt,inner sep=3pt] at (.5,3.5) {$\tte_{a+b}$};
    \node[draw, thick, fill=white!20,rounded corners=4pt,inner sep=3pt] at (.5,2.5) {$d(\pd_{w_0})$};
\end{tikzpicture}}
+(-1)^{\binom{a+b}{2}}
\hackcenter{\begin{tikzpicture}[scale=0.7]
    \draw[thick] (-.25,-.5) -- (-.25,1);
    \draw[thick] (-.25,1) [out=90, in=-90] to (0,2.5);
    \draw[thick] (1.25,-.5) -- (1.25,1);
    \draw[thick] (1.25,1) [out=90, in=-90] to (1,2.5);
    \draw[thick] (.5,2.5) -- (.5,4.5);
    \node[draw, thick, fill=white!20,rounded corners=4pt,inner sep=3pt] at (-.25,1.25) {$d(x^{\delta'})$};
    \node[draw, thick, fill=white!20,rounded corners=4pt,inner sep=3pt] at (1.25,.5) {$x^{\delta''}$};
    \node[draw, thick, fill=white!20,rounded corners=4pt,inner sep=3pt] at (.5,3.5) {$\tte_{a+b}$};
    \node[draw, thick, fill=white!20,rounded corners=4pt,inner sep=3pt] at (.5,2.5) {$\pd_{w_0}$};
\end{tikzpicture}}
+(-1)^{\binom{a+b}{2}+\binom{a}{2}}
\hackcenter{\begin{tikzpicture}[scale=0.7]
    \draw[thick] (-.25,-.5) -- (-.25,1);
    \draw[thick] (-.25,1) [out=90, in=-90] to (0,2.5);
    \draw[thick] (1.25,-.5) -- (1.25,1);
    \draw[thick] (1.25,1) [out=90, in=-90] to (1,2.5);
    \draw[thick] (.5,2.5) -- (.5,4.5);
    \node[draw, thick, fill=white!20,rounded corners=4pt,inner sep=3pt] at (-.25,1.25) {$x^{\delta'}$};
    \node[draw, thick, fill=white!20,rounded corners=4pt,inner sep=3pt] at (1.25,.5) {$d(x^{\delta''})$};
    \node[draw, thick, fill=white!20,rounded corners=4pt,inner sep=3pt] at (.5,3.5) {$\tte_{a+b}$};
    \node[draw, thick, fill=white!20,rounded corners=4pt,inner sep=3pt] at (.5,2.5) {$\pd_{w_0}$};
\end{tikzpicture}}\right]~.
\end{array}
\]
Inside the square brackets, there are three differentials.  We compute them using equations \eqref{eqn-d-idempotent} and \eqref{eqn-d-x-delta}:
\begin{equation*}\begin{split}
&d(\pd_{w_0})x^{\delta'}x^{\delta''}=\sum_{i=1}^{a+b}\lbrace i-1\rbrace x_i\pd_{w_0}x^{\delta'}x^{\delta''}-(-1)^{\binom{a+b}{2}}\sum_{i=1}^{a+b}\lbrace a+b-i\rbrace\pd_{w_0}x_ix^{\delta'}x^{\delta''},\\
&(-1)^{\binom{a+b}{2}}\pd_{w_0}d(x^{\delta'})x^{\delta''}=(-1)^{\binom{a+b}{2}}\sum_{i=1}^a\lbrace a-i\rbrace\pd_{w_0}x_ix^{\delta'}x^{\delta''},\\
&(-1)^{\binom{a+b}{2}+\binom{a}{2}}\pd_{w_0}x^{\delta'}d(x^{\delta''})=(-1)^{\binom{a+b}{2}+\binom{a}{2}}\sum_{i=a+1}^{a+b}\lbrace a+b-i\rbrace\pd_{w_0}x^{\delta'}x_ix^{\delta''}.
\end{split}\end{equation*}
The first summation of $d(\pd_{w_0})x^{\delta'}x^{\delta''}$ becomes $0$ when left-multiplied by $\tte_{a+b}$ (as with the third equality of the previous computation).  The $i=a+1,\ldots,a+b$ terms of the second summation of the same differential cancel with $\pd_{w_0}x^{\delta'}d(x^{\delta''})$.  Hence
\begin{equation*}\begin{split}
d\left(\hackcenter{\begin{tikzpicture}[scale=0.7]
    \draw[green!65!black,line width=4] (.5,1) -- (.5,2);
    \draw[green!65!black,line width=4] (0,0) [out=90, in=-90] to (.5,1);
    \draw[green!65!black,line width=4] (1,0) [out=90, in=-90] to (.5,1);
\end{tikzpicture}}\right)
&=\quad(-1)^{\binom{a}{3}+\binom{b}{3}+\binom{a}{2}\binom{b}{2}}\theta_{a,b}\tte_{a+b}\sum_{i=1}^a(-1)^{\binom{a+b}{2}}\left(\lbrace a-i\rbrace-\lbrace a+b-i\rbrace\right)\pd_{w_0}x_ix^{\delta'}x^{\delta''}\\
&=\quad(-1)^{a-1}\lbrace b\rbrace
\hackcenter{\begin{tikzpicture}[scale=0.7]
    \draw[green!65!black,line width=4] (.5,2) -- (.5,3);
    \draw[green!65!black,line width=4] (0,1) [out=90, in=-90] to (.5,2);
    \draw[green!65!black,line width=4] (1,1) [out=90, in=-90] to (.5,2);
    \draw[green!65!black,line width=4] (0,0) -- (0,1);
    \draw[green!65!black,line width=4] (1,0) -- (1,1);
    \node[draw, thick, fill=white!20,rounded corners=4pt,inner sep=3pt] at (0,.75) {$\et_1$};
\end{tikzpicture}}.
\end{split}
\end{equation*}
We finish the proof of the formulas by observing that $(-1)^{a-1}\lbrace b \rbrace=(-1)^{ab-1} \{b\}$ for any integers $a,~b$.
\end{proof}

\begin{cor}\label{cor-d-exploders}
The action of the differential on exploders are given as follows:
\begin{equation}\label{eqn-defn-exploder}\begin{split}
d\left(\hackcenter{\begin{tikzpicture}[scale=0.7]
    \draw[green!65!black,line width=4] (0,0) -- (0,1);
    \draw[thick] (0,1) [out=150, in=-90] to (-2,3);
    \draw[thick] (0,1) [out=120, in=-90] to (-1.5,3);
    \draw[thick] (0,1) [out=90, in=-90] to (-1,3);
    \node at (0,2.5) {\small$\cdots$};
    \draw[thick] (0,1) [out=90, in=-90] to (1,3);
    \draw[thick] (0,1) [out=60, in=-90] to (1.5,3);
    \draw[thick] (0,1) [out=30, in=-90] to (2,3);
    \node[below] at (0,0) {\small$n$};
    \node[above] at (0,3) {\vphantom{\small$n$}};
\end{tikzpicture}}\right)
\quad&=\quad
\sum_{i=1}^n\lbrace i-1\rbrace
\hackcenter{\begin{tikzpicture}[scale=0.7]
    \draw[green!65!black,line width=4] (0,0) -- (0,1);
    \draw[thick] (0,1) [out=150, in=-90] to (-2,3);
    \draw[thick] (0,1) [out=120, in=-90] to (-1.5,3);
    \draw[thick] (0,1) [out=90, in=-90] to (-1,3);
    \node at (0,2.5) {\small$\cdots$};
    \node at (.85,2.5) {\bbullet};
    \node[above] at (1,3) {\small$i$};
    \draw[thick] (0,1) [out=90, in=-90] to (1,3);
    \draw[thick] (0,1) [out=60, in=-90] to (1.5,3);
    \draw[thick] (0,1) [out=30, in=-90] to (2,3);
    \node[below] at (0,0) {\small$n$};
    \node[above] at (0,3) {\vphantom{\small$n$}};
\end{tikzpicture}}
\quad,\\
d\left(
\hackcenter{\begin{tikzpicture}[scale=0.7]
    \draw[green!65!black,line width=4] (0,2) -- (0,3);
    \draw[thick] (-2,0) [out=90, in=-150] to (0,2);
    \draw[thick] (-1.5,0) [out=90, in=-120] to (0,2);
    \draw[thick] (-1,0) [out=90, in=-90] to (0,2);
    \node at (0,.5) {\small$\cdots$};
    \draw[thick] (1,0) [out=90, in=-90] to (0,2);
    \draw[thick] (1.5,0) [out=90, in=-60] to (0,2);
    \draw[thick] (2,0) [out=90, in=-30] to (0,2);
    \node[above] at (0,3) {\small$n$};
    \node[below] at (0,0) {\vphantom{\small$n$}};
\end{tikzpicture}}\right)
\quad&=\quad
-(-1)^{\binom{n}{2}}\sum_{i=1}^n\lbrace n-i\rbrace
\hackcenter{\begin{tikzpicture}[scale=0.7]
    \draw[green!65!black,line width=4] (0,2) -- (0,3);
    \draw[thick] (-2,0) [out=90, in=-150] to (0,2);
    \draw[thick] (-1.5,0) [out=90, in=-120] to (0,2);
    \draw[thick] (-1,0) [out=90, in=-90] to (0,2);
    \node at (0,.5) {\small$\cdots$};
    \draw[thick] (1,0) [out=90, in=-90] to (0,2);
    \node at (.85,.5) {\bbullet};
    \node[below] at (1,0) {\small$i$};
    \draw[thick] (1.5,0) [out=90, in=-60] to (0,2);
    \draw[thick] (2,0) [out=90, in=-30] to (0,2);
    \node[above] at (0,3) {\small$n$};
    \node[below] at (0,0) {\vphantom{\small$n$}};
\end{tikzpicture}}
\quad.
\end{split}
\end{equation}
\end{cor}
\begin{proof}
 Follows directly from Lemmas \ref{lem-d-longest} and \ref{lem-d-idempotent}
\end{proof}

This differential action on exploders should be compared with the dg structure on the bimodule $Z_n$ given by \eqref{eqn-d-onez}.

\subsection{Thick calculus and bimodules}\label{subsec-thick-intrinsic}

In this subsection, we give a more intrinsic explanation of the origin of the dg thick calculus defined above.

\subsubsection{An odd analogue of the cohomology of Grassmannians}
It is well known that usual cohomology ring of the Grassmannian $\gr(a,b):=\{V\subset \C^{a+b}:\textrm{dim}(V)=a\}$ has a presentation as
\[
\mH^*(\gr(a,b),\Z)\cong  \sym_a/(h_m:m>b).
\]
The \emph{Borel presentation} of this ring, which is more symmetric in $a$ and $b$, is a realization of $\mH(\gr(a,b),\Z)$ as a specialization of the $GL$-equivariant cohomology ring
\[
\mH^*_{GL(a+b,\C)}(\gr(a,b),\Z)\cong \sym_a\otimes \sym_b
\]
via the base change map
\[
\mH^*_{GL(a+b,\C)}(\mathrm{pt},\Z)\cong \sym_{a+b}\lra \Z, \quad\ h_m\mapsto 0~(m>0).
\]
Therefore, denoting the augmentation ideal of the last map by $\sym_{a+b}^+$, we have
\[
\mH^*(\gr(a,b),\Z)\cong (\sym_a\otimes \sym_b)/\sym_{a+b}^+.
\]

An odd analogue of $\mH^*(\gr(a,b),\Z)$ is the superalgebra
\begin{equation}
\oht_{a,b}=\osymt_a/(\htil_m:m>b)
\end{equation}
first introduced in \cite[Section 5]{EKL}.  Now we define the analogue of the Borel presentation of this ring.  We will work with untwisted odd symmetric polynomials; twisted counterparts can be defined analogously.

Consider the rank-one $(\osym_a\otimes \osym_b)$-submodule $\osym_a\boxtimes \osym_b$ inside $\opol_{a+b}$. If we write $e_k(\undx)=e_k(x_1,\ldots,x_a)$ for odd elementary polynomials in $\osym_a$ and $e_k(\undy)=e_k(x_{a+1},\ldots,x_{a+b})$ for odd elementary polynomials in $\osym_b$, then it is clear that $\osym_a\boxtimes \osym_b$ is stable under right multiplication by elements of $\osym_{a+b}$, the latter generated by the elementary polynomials in $x_1,\ldots,x_{a+b}$,
\begin{equation*}
e_k(\undx,\undy)=e_k(x_1,\ldots,x_{a+b})\in\osym_{a+b}.
\end{equation*}
Let
\begin{equation*}
M=(\osym_a\boxtimes\osym_b)\cdot \osym_{a+b}^+\subset\osym_a\boxtimes\osym_b
\end{equation*}
be the left submodule generated by odd elementary polynomials in both $\undx$ and $\undy$ of positive degree.  The submodule $M$ is also closed under right multiplication by elements of $\osym_{a+b}$, so we can view $M$ as a $(\osym_a\otimes\osym_b,\osym_{a+b})$-bimodule.

\begin{lem} \label{lem-iso-as-bim} The map
\begin{equation}\label{eqn-oh-iso}
\begin{split}
\gamma:\oh_{a,b}&\to(\osym_a\boxtimes\osym_b)/M\\
e_k&\mapsto(-1)^{\binom{k}{2}}e_k(\undx)
\end{split}
\end{equation}
is an isomorphism of left $\osym_a$-modules.
\end{lem}
\begin{proof} In $\oh_{a,b}$, there is a relation
\begin{equation}\label{eqn-oh-1}
\sum_{i=0}^k(-1)^{i+\binom{i}{2}}e_ih_{k-i}=0\quad(k\geq0)
\end{equation}
\cite[Lemma 2.8]{EKL}.  In $\osym_a\otimes\osym_b$, there is the obvious relation
\begin{equation}\label{eqn-oh-2}
e_k(\undx,\undy)=\sum_{i=0}^ke_i(\undx)e_{k-i}(\undy).
\end{equation}
Combining these equations, it follows that the map $\gamma$ is well defined and that $\gamma(h_k)=(-1)^ke_k(\undy)$.  Since elementary polynomials in the separate sets of variables $\undx$ and $\undy$ generate $\osym_a\boxtimes\osym_b$, this map is surjective.  By comparing graded ranks, it follows that $\gamma$ is an isomorphism.
\end{proof}

The map $\gamma$ is not a homomorphism of $(\osym_a\otimes\osym_b)$-modules.  By the odd Grassmannian relation, the images under $\gamma$ of the complete symmetric polynomials $h_k$ in $\oh_{a,b}$ are multiples of odd elementary polynomials in $1\boxtimes\osym_b$.  Hence these images supercommute with odd elementary polynomials in $\osym_a\boxtimes1$.  But odd complete and odd elementary polynomials do not supercommute in $\oh_{a,b}$.

In what follows, we will mainly consider the odd analogue $\oh(\gr(a,b))\cong \osym_a\otimes \osym_b$ of the equivariant cohomology ring. Note that $\oh(\gr(a,b))$ is naturally a $(\osym_a\otimes\osym_b, \osym_{a+b})$-bimodule.

\subsubsection{A generalization: the bimodules $Z_{\unda}$}\label{subsubsec-gen-zoid}

To match our earlier conventions, let $\osymt_{a}\boxtimes \osymt_b$ be the rank-one $\osym_a\otimes\osym_b$-module with one generator $\onez$:
\[
\osymt_{a}\boxtimes \osymt_b:=(\osym_a\otimes\osym_b)\cdot\onez,
\]
where the ``$\cdot$'' action is via the ring isomorphism $\theta:\osym_{a}\otimes\osym_b \lra \osymt_a\otimes\osymt_b$ of equation \eqref{eqn-theta-defn}:
\[
(f\otimes g)\cdot \onez :=(\theta(f)\otimes \theta(g))\onez
\]
The isomorphism $\theta$ also equips with $\osymt_a\boxtimes \osymt_b$ a right module structure over $\osym_{a+b}$: for any $f\in \osym_{a+b}$,
\[
\onez \cdot f := (\theta\circ w_0)(f)\onez.
\]
This twisted structure is needed to match with Lemma \ref{lem-idempotent-commute}. One could instead proceed with untwisted odd symmetric polynomials in what follows and apply the twisting maps at the end.

Recall the ``hat'' variant of the odd Schur polynomials \cite[Definition 4.10]{EKL}.  Define
\begin{eqnarray}
\eta^n_\lambda&=&\sum_j\lambda_j\binom{n-j+1}{2},\\
\sth_\lambda&=&(-1)^{\binom{n}{3}+\eta^n_\lambda}(w_0\circ\pd_{w_0})\left(x_1^{\lambda_n}\cdots x_n^{n-1+\lambda_1}\right)\in\osymt_n,\\
\sh_\lambda&=&(-1)^{\binom{n}{3}+\eta^n_\lambda}(\theta\circ w_0\circ\pd_{w_0})\left(x_1^{\lambda_n}\cdots x_n^{\lambda_1+n-1}\right)\in\osym_n.
\end{eqnarray}
Note the differences with \emph{loco citato}: our $\st_\lambda$, $\sth_\lambda$ are respectively that paper's $s_\lambda$, $\sh_\lambda$ and our $s_\lambda,\sh_\lambda$ do not appear.  The relationship between ordinary and hat Schur polynomials is rather simple.
\begin{lem}\label{lem-shat} The four variants of odd Schur polynomials are related as follows:
\begin{equation}\label{eqn-schur-square}
\xymatrix{s_\lambda   \ar@{<->}[rrrr]^-{(-1)^{\binom{n}{3}+\binom{n-1}{2}|\lambda|+\sum_{i<j}\lambda_i\lambda_j}w_0}   \ar@{<->}[d]_-\theta   &&&&   \sh_\lambda   \ar@{<->}[d]^-\theta\\
\st_\lambda   \ar@{<->}[rrrr]^-{(-1)^{\binom{n}{3}+\binom{n}{2}|\lambda|+\sum_{i<j}\lambda_i\lambda_j}w_0}   &&&&   \sth_\lambda}
\end{equation}
\end{lem}
\begin{proof} The proof is a dance with signs and is left to the reader.  A useful fact is that
\begin{equation*}
\pd_1(\pd_2\pd_1)\cdots(\pd_{n-1}\cdots\pd_1)=\pd_{n-1}(\pd_{n-2}\pd_{n-1})\cdots(\pd_1\cdots\pd_{n-1}).
\end{equation*}
(The left-hand side was our definition of $\pd_{w_0}$.)
\end{proof}
With a little computation, Lemma \ref{lem-shat} implies
\begin{equation}
\sth_\lambda=(-1)^{\sum_{i<j}\lambda_i\lambda_j}\pd_{w_0}\left(x^\lambda x^\delta\right).
\end{equation}
We have hat versions of the SZ relation \eqref{eqn-sz}, the Pieri rule, and the computation of the differential of an untwisted odd Schur polynomial:
\begin{eqnarray}
\sth_\lambda\onez&=&\onez w_0(\sh_\lambda),\\
\et_1\sth_\lambda&=&\sum_{\mu=\lambda+\boxi}(-1)^{\binom{n-1}{2}+\lvert\frac{\lambda}{i}\rvert}\sth_\mu,\\
d(\sh_\lambda)&=&(-1)^{\binom{n-1}{2}}\sum_{\mu=\lambda+\boxi}(-1)^{\lvert\frac{i}{\lambda}\rvert+i-1}\lb\ct(\boxi)\rb\sh_\mu.
\end{eqnarray}
We omit proofs of these; a useful fact for the last is that $d\circ w_0=w_0\circ d$.

The odd hat Schur polynomials form the basis dual to the basis of Schur polynomials in the bimodules $Z_{a,b}$ we will define below.  For a partition $\alpha\in\Par(a,b)$, let $\widehat{\alpha}\in\Par(b,a)$ be its conjugate.  Consider the space of diagrams considered in \cite[Theorem 4.16]{EKL},

\[
\mathrm{E}_{a,b}:=
\text{span}\left\{
\hackcenter{
\begin{tikzpicture}[scale=0.7]
    \draw[green!65!black,line width=4] (1,0) -- (1,1);
    \draw[green!65!black,line width=4] (1,1) [out=90, in=-90] to (0,2);
    \draw[green!65!black,line width=4] (1,1) [out=90, in=-90] to (2,2);
    \draw[green!65!black,line width=4] (0,2) -- (0,3);
    \draw[green!65!black,line width=4] (2,2) -- (2,3);
    \draw[green!65!black,line width=4] (0,-1) [out=90, in=-90] to (1,0);
    \draw[green!65!black,line width=4] (2,-1) [out=90, in=-90] to (1,0);
    \draw[green!65!black,line width=4] (0,-2) -- (0,-1);
    \draw[green!65!black,line width=4] (2,-2) -- (2,-1);
    \node at (0,3.5) {\small$a$};
    \node at (2,3.5) {\small$b$};
    \node at (0,-2.5) {\small$a$};
    \node at (2,-2.5) {\small$b$};
    \node[draw, thick, fill=white!20,rounded corners=4pt,inner sep=3pt] at (0,-1.25) {$\sth_{\hat{\alpha}}$};
    \node[draw, thick, fill=white!20,rounded corners=4pt,inner sep=3pt] at (2,2.25) {$\st_{{\beta}}$};
    \node[draw, thick, fill=white!20,rounded corners=4pt,inner sep=3pt] at (1,0.5) {$f$};
\end{tikzpicture}}
\bigg| \alpha, \beta \in P(b,a);~f\in \osymt_{a+b}
\right\}.
\]
Our goal in this section is to show that this space, equipped with the differential given above, naturally arises as the dg endomorphism algebra of a certain dg bimodule $Z_{a,b}$.

Regarding $\osymt_a\boxtimes\osymt_b$ as the rank-one free module over $\osymt_a\otimes\osymt_b$, there is a right $\osymt_{a+b}$-linear ``trace'' map
\[
\pd_{a,b}:=\pd_{w_{a,b}}:\osymt_a\boxtimes\osymt_b\lra \osymt_{a+b}, \quad f\mapsto \pd_{w_{a,b}}(f),
\]
where $\pd_{w_{a,b}}$ is given in equation \eqref{eqn-pd-wab}. We will abbreviate $\pd_{w_{a,b}}$ by $\pd_{a,b}$ in what follows. This is not quite a trace map in the usual sense because $\osymt_{a+b}$ is not supercommutative. We will draw this trace map as
\[
\pd_{{a,b}}=
\hackcenter{\begin{tikzpicture}[scale=0.7]
    \draw[green!65!black,line width=4] (.5,1) -- (.5,2);
    \draw[green!65!black,line width=4] (0,0) [out=90, in=-90] to (.5,1);
    \draw[green!65!black,line width=4] (1,0) [out=90, in=-90] to (.5,1);
    \node at (.5,2.5) {\small$a+b$};
    \node at (0,-.5) {\small$a$};
    \node at (1,-.5) {\small$b$};
\end{tikzpicture}} \ .
\]

As a right module over $\osymt_{a+b}$, it is shown in \cite[Proposition 4.11]{EKL} that $\osymt_a\boxtimes\osymt_b$ has bases
\begin{equation}\label{eqn-dual-bases}
\lb \sth_\lambda(\undx)=\sth_\lambda\boxtimes1\rb_{\lambda\in\Par(a,b)},\quad\lb \st_\mu(\undy)=1\boxtimes \st_\mu\rb_{\lambda\in\Par(b,a)}
\end{equation}
satisfying
\begin{equation}\label{eqn-dual-basis-pairing}
\pd_{{a,b}}(\sth_\lambda(\undx)\st_\mu(\undy))=\pm\delta_{\lambda,\hat{\mu}}.
\end{equation}
The explicit sign is given by \cite[(4.35)]{EKL}.

\begin{defn}\label{def-zoidberg-generalized}
Let $d$ be the differential on the $(\osym_{a,b},\osym_{a+b})$-bimodule $\osymt_a\boxtimes\osymt_b$ determined by its action on the module generator $\onez$ by
\[
d(\onez)=\{a\}\et_1(\undy)\onez
\]
and extended to the whole space using to the Leibniz rule. This is a left dg module over $\osym_a\otimes \osym_b$. Furthermore, it is easy to check that the right $\osym_{a+b}$-action
\[
\onez \cdot h:=(\theta\circ w_0)(h)\onez.
\]
is compatible with $d$. We denote the resulting dg bimodule by $Z_{a,b}$.
\end{defn}

\begin{lem}\label{lem-d-thick-basis}
The differential on $Z_{a,b}$ acts on the basis element $\st_\lambda(\undy)\onez$ by
\[
d(\st_\lambda(\undy)\onez)=\sum_{\mu=\lambda+\boxi}(-1)^{\lvert\frac{\lambda}{i}\rvert+i-1}\lb a+\ct(\boxi)\rb \st_\mu(\undy)\onez,
\]
where the sum is over all partitions $\mu\in\Par(b,a)$ obtained from $\lambda$ by adding a box.
\end{lem}
\begin{proof} For the duration of this proof, we abbreviate $f(\undy)\in Z_{a,b}$ by simply $f$.  The odd Pieri rule \cite[Proposition 3.6]{EOddLR} tells us that
\begin{equation}\label{eqn-odd-pieri-eoddlr}
\st_\lambda\et_1=\sum_{\mu=\lambda+\boxi}(-1)^{\lvert\frac{i}{\lambda}\rvert}\st_\mu.
\end{equation}
It follows, using Proposition \ref{prop-d-schur}, that
\begin{eqnarray*}
d(\st_{\lambda}\onez) & = & d(\st_{\lambda})\onez+(-1)^{|\lambda|}\st_{\lambda}d(\onez)\\
 & = & \sum_{\mu=\lambda+\boxi}(-1)^{\lvert\frac{\lambda}{i}\rvert+i-1}\{\ct(\boxi)\}\st_\mu\onez + \{a\}\sum_{\mu=\lambda+\boxi}(-1)^{|\lambda|-\lvert\frac{i}{\lambda}\rvert}\st_{\mu}\onez\\
 & = & \sum_{\mu=\lambda+\boxi}(-1)^{\lvert\frac{\lambda}{i}\rvert+i-1}\{\ct(\boxi)\}\st_\mu\onez +
 \{a\}\sum_{\mu=\lambda+\boxi}(-1)^{\lvert\frac{\lambda}{i}\rvert+\lambda_i}\st_{\mu}\onez\\
 & = & \sum_{\mu=\lambda+\boxi}(-1)^{\lvert\frac{\lambda}{i}\rvert+i-1}\left(\{\ct(\boxi)\}
 +(-1)^{\ct(\boxi)}\{a\}
 \right)\st_\mu\onez\\
 & = & \sum_{\mu=\lambda+\boxi}(-1)^{\lvert\frac{\lambda}{i}\rvert+i-1}\{\ct(\boxi)+a\}\st_{\mu}\onez,
\end{eqnarray*}
where we have used that $\ct(\boxi)=\lambda_i-i+1$. The lemma follows.
\end{proof}

\begin{cor}\label{cor-cofibrance-generalized-Zoidberg}
The basis $\lb\st_\mu(\undy)\onez|\mu \in \Par(b,a)\rb$ of $Z_{a,b}$ over $\osym_{a+b}$ is $d$-stable. In particular, $Z_{a,b}$ is cofibrant as a right dg module over $\osym_{a+b}$.
\end{cor}
\begin{proof} This follows from the previous lemma since, if a box in the first row and $(a+1)$-st column is added to a Young diagram in $\Par(b,a)$ resulting in a new Young diagram $\mu$, the corresponding coefficient in front of $\mu$ equals
\[
(-1)^{\lvert\frac{\lambda}{i}\rvert+i-1}\{a+a\}\st_{\mu}=0.
\]
The first claim follows. To show the second, introduce an increasing filtration on this $d$-stable basis, starting with the top degree term and iteratively adding terms of equal or lower degrees.  This induces a dg filtration on $Z_{a,b}$ whose subquotients are all isomorphic to $\osym_{a+b}$. Therefore, $Z_{a,b}$ is a finite-cell module in the sense of Example \ref{eg-finite-cell}, and thus is cofibrant.
\end{proof}

The composition of the $\opol_n$-involution $\theta:\osymt_{a+b}\lra \osym_{a+b}$ with the map $\pd_{a,b}$ gives a $\osym_{a+b}$-linear trace
\begin{equation}
\onez^\vee := \theta\circ\pd_{a,b},
\end{equation}
which is an element of the space
\begin{equation}\label{eqn-dual-Zoidberg}
Z_{a,b}^\vee:=\HOM_{\osym_{a+b}}(Z_{a,b},\osym_{a+b}).
\end{equation}
Since $Z_{a,b}$ is free as a right module over $\osym_{a+b}$ of graded rank\footnote{We use quantum numbers with a subscript $q$ to indicate that it is a Laurent polynomial in $\Z[q^{\pm 1}]$, instead of evaluated at $\sqrt{-1}$.} ${a+b \brack a}_q$ and free of rank one on the left over $\osym_{a,b}$, it follows that $Z_{a,b}^\vee$ is a $(\osym_{a+b}, \osym_{a,b})$-bimodule with $\onez^\vee$ as a lowest degree bimodule generator.

\begin{defn}\label{defn-Zoidberg-dual-generalized}
In what follows, we will call $Z_{a,b}^\vee$ the \emph{dual bimodule to $Z_{a,b}$}. As a right module over $\osym_{a,b}$,
\begin{equation}\label{eqn-dual-differential-action}
Z_{a,b}^\vee \cong \onez^\vee \cdot (\osym_{a}\otimes \osym_b), \quad d(\onez^\vee)=(-1)^{ab-1}\{b\}\onez^\vee\et_1(\undx).
\end{equation}
\end{defn}

\begin{cor}\label{cor-d-invariance-trace}
The trace map
$
\onez^\vee:Z_{a,b}\lra \osym_{a+b}
$,
and its differential are compatible with the dg bimodule structure on $Z_{a,b}$ in the sense that
\[
d(\onez^\vee(f\onez))=d(\onez^\vee)(f\onez)+(-1)^{ab}\onez^\vee(d(f\onez)).
\]
\end{cor}
\begin{proof} Right multiplying both sides by an element of $\osym_{a+b}$ does not change the validity of the equation. Therefore, it suffices to check the equation on the basis elements $\{\st_{\mu}(\undy)\onez|\mu\in \Par(b,a)\}$ of $Z_{a,b}$. There are two non-trivial cases: $\mu=(a^b)$ and $\mu=(a^{b-1},a-1)$.  The three quantities to be compared are
\begin{enumerate}
	\item $d((\theta\circ\pd_{a,b})(\st_\mu(\undy)\onez))$,
	\item $(-1)^{ab-1}\{b\}(\theta\circ\pd_{a,b})(\et_1(\undx)\st_\mu(\undy)\onez)$, and
	\item $(-1)^{ab}(\theta\circ\pd_{a,b})(d(\st_\mu(\undy)\onez))$.
\end{enumerate}
One computes these quantities using dual bases of \eqref{eqn-dual-bases} (along with the sign of \cite[(4.35)]{EKL}), the Pieri rule \eqref{eqn-odd-pieri-eoddlr}, and Lemma \ref{lem-d-thick-basis}.  Note that $\onez^\vee$ takes values in $\osym_{a+b}$ \emph{considered as a dg algebra}, so that $d(\onez^\vee(f))=0$ if $\onez^\vee(f)$ is a scalar.  It follows that for $\mu=(a^b)$, all three terms are zero.  For $\mu=(a^{b-1},a-1)$, the first term is zero and the second and third cancel by the computations of \cite[Section 4]{EKL}.
\end{proof}

\begin{cor}\label{cor-cofibrance-dual-generalized-Zoidberg}As a left dg module over $\osym_{a+b}$, $Z_{a,b}^\vee$ is cofibrant of graded rank ${a+b \brack a}_q$.
\end{cor}
\begin{proof}It follows from Corollary \ref{cor-d-invariance-trace} and the orthogonality equation \ref{eqn-dual-basis-pairing} that
\[
\{\onez^\vee\sth_\lambda(\undx)|\lambda\in \Par(a,b)\}
\]
is a $d$-stable basis of $Z_{a,b}^\vee$. Therefore, as in the proof Corollary \ref{cor-cofibrance-generalized-Zoidberg}, $Z_{a,b}^\vee$ is a right finite-cell dg module over $\osym_{a+b}$ and is therefore cofibrant. The graded rank count also follows.
\end{proof}

Using the trace map, we can construct various endomorphisms of the bimodule $Z_{a,b}$ as follows. For any $g\onez\in Z_{a,b}$, and $\onez^\vee f \in Z_{a,b}^\vee,$ we define
\[
Z_{a,b}\otimes_{\osym_{a+b}}Z^\vee_{a,b} \lra \END_{\osym_{a+b}}(Z_{a,b}), \quad (g\onez, \onez^\vee f)\mapsto g w_0(\pd_{a,b}(f\cdot -))\onez.
\]
Diagrammatically, given $h(\undx,\undy)\in \osymt_{a+b}$, $\st_{\beta}(\undy)\in \osymt_b$ and $\sth_{\hat{\alpha}}(\undx)\in \osymt_b$, consider the diagram
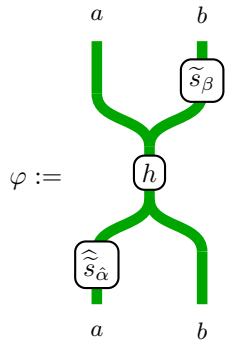
\[
\varphi:=
\begin{gathered}
\begin{tikzpicture}[scale=0.7]
    \draw[green!65!black,line width=4] (1,0) -- (1,1);
    \draw[green!65!black,line width=4] (1,1) [out=90, in=-90] to (0,2);
    \draw[green!65!black,line width=4] (1,1) [out=90, in=-90] to (2,2);
    \draw[green!65!black,line width=4] (0,2) -- (0,3);
    \draw[green!65!black,line width=4] (2,2) -- (2,3);
    \draw[green!65!black,line width=4] (0,-1) [out=90, in=-90] to (1,0);
    \draw[green!65!black,line width=4] (2,-1) [out=90, in=-90] to (1,0);
    \draw[green!65!black,line width=4] (0,-2) -- (0,-1);
    \draw[green!65!black,line width=4] (2,-2) -- (2,-1);
    \node at (0,3.5) {\small$a$};
    \node at (2,3.5) {\small$b$};
    \node at (0,-2.5) {\small$a$};
    \node at (2,-2.5) {\small$b$};
    \node[draw, thick, fill=white!20,rounded corners=4pt,inner sep=3pt] at (0,-1.25) {$\sth_{\hat{\alpha}}$};
    \node[draw, thick, fill=white!20,rounded corners=4pt,inner sep=3pt] at (2,2.25) {$\st_{{\beta}}$};
    \node[draw, thick, fill=white!20,rounded corners=4pt,inner sep=3pt] at (1,0.5) {$h$};
\end{tikzpicture}
\end{gathered}
\]
from the beginning of this section.  This map $\varphi$ acts on $Z_{a,b}$ by first multiplying any element of $Z_{a,b}$ on the left with $\sth_{\hat{\alpha}}$, then taking the trace into $\osymt_{a+b}$ while left multiplying by $h$, and finally multiplying the resulting element with $\st_{\beta}$ and sending it back into $Z_{a,b}$. It is clear that such a map, by construction, is right $\osym_{a+b}$-linear. Denote by $E_{a,b}$ the space of all such endomorphisms of $Z_{a,b}$.

\begin{prop}\label{prop-nice-property-Zoidberg}For any $a,b\in \N$, the following properties of the dg module $Z_{a,b}$ hold.
\begin{enumerate}
\item The graded right $\osym_{a+b}$-linear endomorphism algebra $\END_{\osym_{a+b}^{\mathrm{op}}}(Z_{a,b})$ computes the $\mathbf{R}\END$-complex of $Z_{a,b}$ over $\osym_{a,b}$:
\[
\mathbf{R}\END_{\osym_{a+b}^{\mathrm{op}}}(Z_{a,b})\cong \END_{\osym_{a+b}^{\mathrm{op}}}({Z_{a,b}}).
\]
Likewise, the graded left $\osym_{a+b}$-linear endomorphism algebra $\END_{\osym_{a+b}}(Z_{a,b}^\vee)$ also computes its own $\mathbf{R}\END$-complex in the derived category.
\item The diagrammatic dg algebra $\mathrm{E}_{a,b}$ is isomorphic to the graded right dg endomorphism algebra of $Z_{a,b}$ as well as the left dg endomorphism algebra of the dual bimodule:
\[
\mathrm{E}_{a,b}\cong \END_{\osym_{a+b}^{\mathrm{op}}}(Z_{a,b})\cong \END_{\osym_{a+b}}(Z^\vee_{a,b}).
\]
\end{enumerate}
\end{prop}

\begin{proof}The first statement follows from Corollary \ref{cor-cofibrance-generalized-Zoidberg}. The second one follows by a graded rank count.
\end{proof}

More generally, let $\unda=(a_1,\ldots,a_r)$ be a composition of $n$.  There is a $(\osym_{a_1}\otimes\cdots\otimes\osym_{a_r},\osym_n)$-bimodule structure on
\begin{equation*}
Z_{\underline{a}}:=(\osymt_{a_1}\boxtimes\cdots\boxtimes\osymt_{a_r}) \onez_{\unda}.
\end{equation*}
Using \eqref{eqn-slider} and splitter associativity, a right $\osymt_n$-spanning set for this bimodule can be drawn as
\begin{equation*}
(f_1\boxtimes\cdots\boxtimes f_r)\onez_{\unda}=
\hackcenter{\begin{tikzpicture}[scale=0.7]
    \node at (3,2.5) {$\cdots$};
    \draw[green!65!black,line width=4] (2.5,0) -- (2.5,1);
    \draw[green!65!black,line width=4] (2.5,1) [out=90, in=-90] to (1,2);
    \draw[green!65!black,line width=4] (2.5,1) [out=90, in=-90] to (2,2);
    \draw[green!65!black,line width=4] (2.5,1) [out=90, in=-90] to (4,2);
    \draw[green!65!black,line width=4] (1,2) -- (1,3.5);
    \draw[green!65!black,line width=4] (2,2) -- (2,3.5);
    \draw[green!65!black,line width=4] (4,2) -- (4,3.5);
    \node[draw, thick, fill=white!20,rounded corners=4pt,inner sep=3pt] at (1,2.95) {$f_1$};
    \node[draw, thick, fill=white!20,rounded corners=4pt,inner sep=3pt] at (2,2.7) {$f_2$};
    \node[draw, thick, fill=white!20,rounded corners=4pt,inner sep=3pt] at (4,2.45) {$f_r$};
\end{tikzpicture}},
\end{equation*}
with $f_i\in\osymt_{a_i}$ for each $i$.  Give this bimodule the differential structure determined on the generator $\onez$ by
\[
d\onez:=\sum_{i=2}^{n}\{a_1+\cdots+a_{i-1}\} \et_1(x_{a_1+\cdots+a_{i-1}+1},\dots, x_{a_1+\dots+a_i}) \onez.
\]
We call this dg bimodule $Z_{\unda}$.  It possesses a trace map
\begin{equation}
\pd_{{\unda}}: Z_{\unda} \lra \osym_{n}, \quad \pd_{{\unda}}:= \pd_{{a_1,a_2}}\pd_{{a_1+a_2,a_3}}\cdots \pd_{{a_1+\cdots a_{r-1}, a_r}}.
\end{equation}
Let $\undb$ be another composition of $n$.  Elements of the dual $Z_{\undb}^\vee$ of $Z_{\undb}$ this with respect to the trace pairing $\pd_{\unda}$ are drawn as
\begin{equation*}
\onez_{\undb}^\vee(g_1\boxtimes\cdots\boxtimes g_s)=
\hackcenter{\begin{tikzpicture}[scale=0.7]
    \node at (3,1) {$\cdots$};
    \draw[green!65!black,line width=4] (2.5,2.5) -- (2.5,3.5);
    \draw[green!65!black,line width=4] (1,1.5) [out=90, in=-90] to (2.5,2.5);
    \draw[green!65!black,line width=4] (2,1.5) [out=90, in=-90] to (2.5,2.5);
    \draw[green!65!black,line width=4] (4,1.5) [out=90, in=-90] to (2.5,2.5);
    \draw[green!65!black,line width=4] (1,0) -- (1,1.5);
    \draw[green!65!black,line width=4] (2,0) -- (2,1.5);
    \draw[green!65!black,line width=4] (4,0) -- (4,1.5);
    \node[draw, thick, fill=white!20,rounded corners=4pt,inner sep=3pt] at (1,1.05) {$g_1$};
    \node[draw, thick, fill=white!20,rounded corners=4pt,inner sep=3pt] at (2,.8) {$g_2$};
    \node[draw, thick, fill=white!20,rounded corners=4pt,inner sep=3pt] at (4,.55) {$g_s$};
\end{tikzpicture}}.
\end{equation*}


Note that in the case $r=2$, diagrams for $Z_{(a_1,a_2)}$ (intrinsic) look identical to the splitter diagrams in the odd nilHecke algebra as described in the previous subsection (extrinsic).  By the end of this subsection, we will have proven the intrinsic and extrinsic descriptions are of isomorphic dg bimodules.  The generator $\onez$ of the extrinsic description corresponds to the ``$n$-to-$1$'' intrinsic generator $\onez_{(1,1,\ldots,1)}$.  When confusion is unlikely, we will write $\onez$ for any $\onez_{\unda}$.

It follows from Lemma \ref{lem-d-thick-basis} that, if $\lambda_0=(a_1^{a_2})$ is the full $(a_1\times a_2)$-rectangle, then
\[
d(\st_{\lambda_0}(\undy)\onez)=0,
\]
since the only possible box that can be added has content $-a_1$. It also follows that the free abelian group spanned by the basis elements
$\lb \st_\lambda(\undy)\onez:\lambda\in\Par(b,a)\rb$, denoted by
$$V_{a_1,a_2}:=\bigoplus_{\lambda\in \Par(b,a)}\Z\cdot\st_{\lambda}(\undy)\onez,$$
is $d$-stable. The cohomology of the chain complex $V_{a,b}$ is computed in Subsection \ref{apx-cohomology-thick} of the appendix.

It is easy to generalize the above using induction on the length of the composition $\unda$.  The result is the next proposition. To state it,  we will interprete the graded rank of the dg-bimodule $Z_{\unda}$ in terms of quantum multinomial coefficients for a \emph{generic $q$} value: $[n]_q:=(q^n-q^{-n})/(q-q^{-1})$, $[n]_q!:= [n]_q[n-1]_q\cdots [1]_q$, and, for any decomposition $\underline{a}=(a_1,\dots, a_k)$ of $n$, we set
$$
\left[n\atop\unda\right]_q:=\dfrac{[n]_q!}{[a_1]_q!\cdots [a_k]_q!}.
$$
\begin{prop}\begin{enumerate}
\item The set
\begin{equation*}
\{\st_{\lambda_1}(x_1,\ldots,x_{a_1})\st_{\lambda_2}(x_1,\ldots,x_{a_1+a_2})\cdots\st_{\lambda_{k-1}}(x_1,\ldots,x_{a_1+\ldots+a_{k-1}})\},
\end{equation*}
where each $\lambda_i$ ranges over all partitions which fit into a $(a_1+\ldots+a_{i-1})\times a_i$ box, is a basis for $Z_{\unda}$ as a free $\osymt_n$-module.
\item The $\Z$-span of this basis is $d$-stable.
\item The space of $\osym_n$-module homomorphisms $Z_{\unda}\to Z_{\undb}$ is a free $\osym_n$-module of graded rank
\begin{equation*}
\grank_{\osym_n}\HOM(Z_{\unda},Z_{\undb})=\left[n\atop\unda \right]_q\left[n\atop\undb\right]_q.
\end{equation*}
\end{enumerate}
\end{prop}

\subsection{Categorification of {$U^+_{\sqrt{-1}}(\mathfrak{sl}_2)$}} \label{subsec-cat-big}

For simplicity, we fix our ground ring $\Bbbk$ to be a field\footnote{Since $\osym_n$ is formal by the cohomology computation in Proposition \ref{prop-poly-alg}), everything here works over $\Z$ as well, using Theorem \ref{thm-qis-dga}.} in this section.  Our goal in this subsection is to categorify $U^+$ as a $\sqrt{-1}$-twisted bialgebra together with the embedding of twisted bialgebras $u^+\subset U^+$. To do so, we will consider the following dg functors.

\begin{defn}\label{def-multiplication-functor}
The dual bimodule $Z_{a,b}^\vee:=\HOM_{\osym_{a+b}}(Z_{a,b},\osym_{a+b})$, whose generator $\onez^\vee$ is given by the trace map of Corollary \ref{cor-d-invariance-trace}, gives rise to the derived functor
\begin{equation*}\begin{split}
Z_{a,b}^\vee \otimes_{\osym_{a,b}}^{\mathbf{L}}(-): \mc{D}(\osym_a)\times \mc{D}(\osym_b)&\lra \mc{D}(\osym_{a+b}),\\
(M,N)&\mapsto Z_{a,b}^\vee \otimes_{\osym_{a,b}}^{\mathbf{L}} (M\boxtimes N).
\end{split}\end{equation*}
Summing over all $a,b\in \Z_{\geq0}$, and denoting $\mc{D}(\osym):=\oplus_{a\in \Z_{\geq0}}\mc{D}(\osym_a)$, we obtain a functor
\[
\bigind:=\bigoplus_{a,b\in \Z_{\geq0}}Z_{a,b}^\vee \otimes_{\osym_{a,b}}^{\mathbf{L}}(-): \mc{D}(\osym)\times \mc{D}(\osym)\lra \mc{D}(\osym),
\]
which we will call the \emph{multiplication functor}.
\end{defn}

Proposition \ref{prop-nice-property-Zoidberg} gives an explicit way to compute the space of natural transformations of this functor $\bigind$, component-wise and diagrammatically, in terms of dg bimodule homomorphisms.

\begin{defn}\label{def-comultiplication-functor}
\begin{enumerate}
\item The bimodule $Z_{a,b}^\natural$ is the dg bimodule over $(\osym_{a,b},\osym_{a+b})$ given by
\[
(\osymt_{a}\otimes \osymt_b)\cdot \mathbf{z}^\natural,\quad d(\onez^\natural)=0.
\]
The degree of the generator $\onez^\natural$ is defined to be zero.
\item The bimodule  $Z_{a,b}^\natural$ gives rise to a derived functor
\begin{equation*}\begin{split}
Z_{a,b}^\natural \otimes_{\osym_{a+b}}^{\mathbf{L}}(-): \mc{D}(\osym_{a+b})&\lra \mc{D}(\osym_{a,b}),\\
M&\mapsto Z_{a,b}^\natural \otimes_{\osym_{a,b}}^{\mathbf{L}} M.
\end{split}
\end{equation*}
Taking sum over all $a,b\in \Z_{\geq0}$, we obtain a functor $\bigres$, which will be called the \emph{comultiplication functor},
\[
\bigres:=\bigoplus_{a,b\in \Z_{\geq0}}Z_{a,b}^\natural \otimes_{\osym_{a+b}}^{\mathbf{L}}(-):\mc{D}(\osym)\lra \mc{D}(\osym \otimes \osym).
\]
Here  $\mc{D}(\osym\otimes \osym)$ is understood to be the direct sum of categories
\[
\mc{D}(\osym\otimes \osym):= \bigoplus_{a,b\in \Z_{\geq0}}\mc{D}(\osym_a\otimes \osym_b).
\]
\end{enumerate}
\end{defn}

\begin{lem}\label{lemma-Kunneth}
On the level of Grothendieck groups, there is an isomorphism of free $\Z[\sqrt{-1}]$-modules
\[
K_0(\mc{D}(\osym\otimes \osym))\cong K_0(\osym)\otimes_{\Z[\sqrt{-1}]}K_0(\osym).
\]
\end{lem}
\begin{proof}It suffices to check that, component-wise,
\[
K_0(\mc{D}(\osym_a\otimes \osym_b))\cong K_0(\osym_a)\otimes_{\Z[\sqrt{-1}]}K_0(\osym_b).
\]
 Since $\osym_{a}$, $\osym_b$ and $\osym_{a,b}$ are all positive dg algebras, the result follows from Corollary \ref{cor-K0-positive}.  Both sides are isomorphic to $\Z[\sqrt{-1}]$.
\end{proof}

Recall the $\sqrt{-1}$-bialgebra $U^+$ of Subsection \ref{subsec-qgs}.
\begin{thm}\label{thm-big-sl2}
The derived category $\mc{D}(\osym)$, together with the functors
\[
\bigind:  \mc{D}(\osym)\times \mc{D}(\osym)\lra \mc{\osym},\quad
\bigres:  \mc{D}(\osym)\lra \mc{D}(\osym\otimes \osym),
\]
categorifies $U^+$ as a $\sqrt{-1}$-twisted bialgebra. That is, $K_0(\mc{D}(\osym))\cong U^+$, with the symbols of the rank-one free dg module $[\osym_a\langle-\binom{a}{2}\rangle]$ identified with the generator $E^{(a)}$. The  symbols of the functors $\bigind$ and $\bigres$ give rise to the multiplication and comultiplication structures on $U^+$.
\end{thm}

\begin{proof}It is clear that
\[
K_0(\mc{D}(\osym))=\bigoplus_{a\in \Z_{\geq0}}K_0(\mc{D}(\osym_a))\cong \bigoplus_{a\in \Z_{\geq0}}\Z[\sqrt{-1}][E^{(a)}],
\]
where we identify the symbol of the rank-one free dg module $\osym_a\langle-\binom{a}{2}\rangle$ with $E^{(a)}$.

We now check that the relation \eqref{eqn-EaEb} is lifted to the categorical level, modulo the grading shifts.
\[
\bigind(\osym_a\boxtimes \osym_b) =Z_{a,b}^\vee \otimes^{\mathbf{L}}_{\osym_{a,b}}(\osym_a\boxtimes \osym_b)= Z^{\vee}_{a,b}.
\]
As a left dg module over $\osym_{a+b}$, we have seen that $Z_{a,b}^\vee$ is finite-cell (see Example \ref{eg-finite-cell}) and thus cofibrant. The subquotients of the finite-cell structure are ${a+b\brack a}$ shifted copies of the rank-one free module $\osym_{a+b}$, by Corollary \ref{cor-cofibrance-dual-generalized-Zoidberg}. Hence in the Grothendieck group $K_0(\osym_{a+b})$, we have the equality of symbols
\[
[\bigind(\osym_a\boxtimes \osym_b)]=[Z^\vee_{a,b}]={a+b \brack a}[\osym_{a+b}].
\]
Similarly, the comultiplication functor $\bigres$ sends the symbol $[\osym_{n}]$, which is identified with $E^{(n)}$, to the sum of objects
\[
[\bigres(\osym_{a+b})]=\sum_{a+b=n}[Z^{\natural}_{a,b}]=\sum_{a+b}[\osym_a]\otimes_{\Z[\sqrt{-1}]}[\osym_b].
\]
In the last equality, we have used Lemma \ref{lemma-Kunneth}.

It is an easy exercise to match our grading choices with equations \eqref{eqn-EaEb} and \eqref{eqn-comultiplication}. The result follows.
\end{proof}

Finally, we relate the categorified small (Theorem \ref{thm-small-sl2}) and big (Theorem \ref{thm-big-sl2}) quantized enveloping algebras for $\mf{sl}_2$ at a fourth root of unity. This will be a categorical lifting of the embedding of $u^+$ inside $U^+$.

Consider the bimodule $Z_n$ and its dual $Z_n^\vee$ in Definition \ref{def-zoidberg-bim}, which are respectively dg bimodules over $(\onh_n,\osym_n)$ and $(\osym_n,\onh_n)$ by Corollary \ref{cor-onh-end}. The differential actions on the module generators are diagrammatically depicted in Corollary \ref{cor-d-exploders}.

Recall that we have denoted by $\mc{D}(\onh)$ and $\mc{D}(\osym)$ the direct sums of dg derived categories,
\[
\mc{D}(\onh):=\bigoplus_{n\in \Z_{\geq0}}\mc{D}(\onh_n), \quad \mc{D}(\osym):=\bigoplus_{n\in \Z_{\geq0}}\mc{D}(\osym_n).
\]
We also collect together the abelian and homotopy categories:
\begin{equation}
\onh\dmod:=\bigoplus_{n\in \Z_{\geq 0}} \onh_n\dmod, \quad \mc{H}(\onh):= \bigoplus_{n\in \Z_{\geq 0}} \mc{H}(\onh_n),
\end{equation}
\begin{equation}
\osym\dmod:=\bigoplus_{n\in \Z_{\geq 0}} \osym_n\dmod, \quad \mc{H}(\osym):= \bigoplus_{n\in \Z_{\geq 0}} \mc{H}(\osym_n).
\end{equation}

\begin{defn}\label{def-embedding-functor}
Let $\mc{J}^\mc{A}$ be the functor
$
\mc{J}^\mc{A}: \onh\dmod\lra \osym\dmod,
$
which is component-wise given by
\begin{equation*}\begin{split}
\mc{J}^\mc{A}_n: \onh_n\dmod&\lra \osym_n\dmod,\\
M&\mapsto Z_n^\vee\otimes_{\onh_n}M.
\end{split}\end{equation*}
This functor induces functors $\mc{J}^\mc{H}$, $\mc{J}$ of homotopy and derived categories, respectively.
\end{defn}

\begin{cor}\label{cor-embedding}
The derived functor $\mc{J}:\mc{D}(\onh)\lra \mc{D}(\osym)$ is a fully faithful embedding of derived categories. On the level of Grothendieck groups, this categorifies the embedding of $u^+$ inside $U^+$ as a sub-twisted bialgebra.
\end{cor}
\begin{proof}Since $\mc{D}(\onh_n)\simeq 0$ if $n\geq 2$ (Proposition \ref{prop-Zoidberg-finite-cell}), one only has to prove the result for the components $\mc{J}_0$ and $\mc{J}_1$, in which cases the result is clear.
\end{proof}

However, this embedding on the homotopy level is not quite so trivial. In fact, it is easy to check that the underived tensor product with $Z_n^\vee$ gives rise to an equivalence of abelian categories of dg modules over $\onh$ and $\osym$:
\begin{equation}\label{eqn-abelian-Morita}
\begin{split}
\mc{J}_n^{\mc{A}}: \onh_n\dmod&\lra\osym_n\dmod,\\
M&\mapsto Z_n^\vee\otimes_{\onh_n}M.
\end{split}
\end{equation}
The quasi-inverse of $\mc{J}_n^{\mc{A}}$ is given by the tensor product with $Z_n$:
\begin{equation}
\begin{split}
(\mc{J}_n^{\mc{A}})^{-1}: \osym_n\dmod&\lra \onh_n\dmod,\\
M&\mapsto Z_n\otimes_{\osym_n}M.
\end{split}
\end{equation}
This is because, as bimodules,
\[
Z_n\otimes_{\osym_n}Z_n^\vee \cong \onh_n, \quad
Z_n^\vee \otimes_{\onh_n}Z_n \cong \osym_n.
\]
It follows that the functors descend to triangulated equivalences between the homotopy categories:
\[
\mc{J}_n^{\mc{H}}: \mc{H}(\onh_n)\lra\mc{H}(\osym_n),\quad
(\mc{J}_n^{\mc{H}})^{-1}: \mc{H}(\osym_n)\lra \mc{H}(\onh_n).
\]

Therefore, if one wants the categorical bialgebra embedding to be respected on the abelian and homotopy category levels, the restriction functor used in Theorem \ref{thm-small-sl2} will not do. Instead, one should consider a more ``natural'' induction functor on $\onh$, which is given by the (derived) tensor product with the bimodule
\[
\onh^{\natural}:=\bigoplus_{a,b\in\Z_{\geq0}}\onh_{a+b}^{\natural}.
\]
Each of the individual components $\onh_{a+b}^{\natural}$ is an $(\onh_a\otimes\onh_b,\onh_{a+b})$-bimodule of rank one over $\onh_{a+b}$, equipped with the differential
\begin{equation}\label{eqn-twisted-restriction}
\onh_{a+b}^{\natural}:=1^{\natural}\onh_{a+b},\quad d(1^{\natural})=\{a\}\et_1(x_{a+1},\dots,x_{a+b})1^\natural.
\end{equation}
Alternatively, this bimodule is isomorphic, up to a degree and parity shift, to the right ideal generated by $(x_{a+1}\dots x_{a+b})^a$ inside $\onh_n$. It follows from this isomorphism that $\onh_{a+b}^\natural$ is a dg $(\onh_a\otimes\onh_b,\onh_{a+b})$-bimodule.

\begin{defn}The functor $\smres^{\natural}$ is given by the derived tensor product with $\onh^{\natural}$,
\[
\smres^{\natural}:=\onh^{\natural}\otimes_{\onh}^{\mathbf{L}}(-):\mc{D}(\onh)\lra\mc{D}(\onh\otimes \onh).
\]
\end{defn}

This restriction functor behaves better on the homotopy category $\mc{H}(\onh)$ than our na\"{i}ve one in equation \eqref{eqn-restriction}. This phenomenon has appeared in the even setting in \cite{KQ, EliasQi2} and should not be surprising. This is because, under the abelian and homotopy Morita equivalences discussed in equation \eqref{eqn-abelian-Morita}, the bimodule $\onh^\natural$ naturally arises from the bimodule $Z_{a,b}^\natural$ in Definition \ref{def-comultiplication-functor} by applying this Morita equivalence functor, as we will see below. It happens that the two different restriction functors $\smres$ and $\smres^\natural$ coincide on the derived category $\mc{D}(\onh)$, since $\mc{D}(\onh)$ collapses all but the first two components. However, because it behaves better on homotopy categories, it is more natural to consider $\smres^\natural$ than $\smres$.

\begin{cor}\label{cor-correct-embedding}
The equivalence $\mc{J}^\mc{A}:\onh\dmod\to\osym\dmod$ intertwines $\bigind$ with $\smind$ and $\bigres$ with $\smres^\natural$.  The same is true of $\mc{J}^\mc{H}$ on homotopy categories and $\mc{J}$ on derived categories.
\end{cor}

\begin{proof}We prove the statement about the restriction functors.  The other half of the result concerning induction functors is similar and easier.

To do so, it suffices to show that the following diagram of functors commutes on the level of abelian categories:
\begin{equation*}
\begin{gathered}
\xymatrix{\onh_{a+b}\dmod\ar[r]^-{\Res^{\natural}}\ar[d]^-{\mc{J}^\mc{A}}
    &(\onh_a\otimes\onh_b)\dmod\ar@<.5ex>[d]^-{\mc{J}^\mc{A}}\\
\osym_{a+b}\dmod\ar[r]^-{\mathcal{R}}&(\osym_a\otimes\osym_b)\dmod\ar@<.5ex>[u]^-{(\mc{J}^\mc{A})^{-1}}}
\end{gathered}
\ .
\end{equation*}
Set $n=a+b$.  If $M$ is a dg module over $\onh_n$, we have, by construction,
\[
(\mc{J}^\mc{A})^{-1}\circ \bigres\circ\mc{J}^\mc{A}(M)\cong (Z_a\boxtimes Z_b)\otimes_{\osym_a\otimes \osym_b} \left(Z_{a,b}^\natural \otimes_{\osym_{n}} \left(Z_n^\vee\otimes_{\onh_n} M\right)\right).
\]
By tensor associativity and the fact that $Z_{a,b}^\natural\cong \osym_a\otimes \osym_b$, the composition functor is isomorphic to the tensor product with
\[
(Z_a\boxtimes Z_b)\otimes_{\osym_a\otimes \osym_b} \left(Z_{a,b}^\natural \otimes_{\osym_{n}} Z_n^\vee\right)\cong(\opol_a\cdot \onez_a\boxtimes \opol_b\cdot \onez_b)\otimes_{\osym_n}(\onez_n^\vee\cdot\opol_n),
\]
which is a dg bimodule over $(\onh_a\otimes\onh_b,\onh_n)$. Therefore, establishing the commutativity of the diagram is equivalent to finding an isomorphism of the above dg bimodule with $\onh_n^\natural$.

Both bimodules have the same graded rank over $\onh_n$. Thus it suffices to compare the differential actions on the lowest-degree $(\onh_a\otimes \onh_b,\onh_n)$-generators for both bimodules. On the one hand, we have
\begin{align*}
d((\onez_a\boxtimes\onez_b)\otimes \onez_n^\vee) & =
\sum_{i=1}^a\{i-1\}x_i (\onez_a\boxtimes\onez_b)\otimes \onez_n^\vee + \sum_{j=1}^b\{j-1\}x_{a+j}(\onez_a\boxtimes\onez_b)\otimes \onez_n^\vee\\
&\qquad -(-1)^{\binom{n}{2}}\sum_{k=1}^n\{n-k\}(\onez_a\boxtimes\onez_b)\otimes \onez_n^\vee x_k.
\end{align*}
On the other hand, a similar computation shows that the $d$-action on the lowest-degree element $1^\natural\cdot \pd_{w_0}\in \onh_n^\natural$ is exactly the same. Hence choosing an identification of the generators naturally extends to an isomorphism of dg bimodules, and the result follows.
\end{proof}

\appendix
\section{Cohomology and filtrations}\label{apx-cohomology-filtrations}

\subsection{Cohomology of odd symmetric functions}\label{subsec-cohomology-osym}

The goal of this subsection is to compute the cohomology of $\osym$ and of $\osym_n$.  We will identify cocycles of a dg algebra with their images in cohomology when no confusion is likely to result.  Throughout this section, let $\Bbbk$ be a field.

\subsubsection{Hypercube complexes}\label{subsubsec-hypercube}

Up to isomorphism and homological degree shift, there is only one indecomposable contractible chain complex over $\Bbbk$,
\begin{equation*}
\xymatrix{0\ar[r]&\Bbbk\ar[r]^-1&\Bbbk\ar[r]&0.}
\end{equation*}
The total complex of its $k$-th tensor power is the \emph{hypercube complex},
\begin{equation*}\begin{split}
&Y_k=\text{span}_\Bbbk\lb v_{\undvarepsilon}:\undvarepsilon\in\lb0,1\rb^k\rb\\
&d(v_{\undvarepsilon})=\sum_{|\undeta|=1+|\undvarepsilon|}\pm v_{\undeta}
\end{split}\end{equation*}
We call $v_{(0,\ldots,0)}$ the \emph{initial vector} of $Y_k$.

A useful combinatorial picture of $Y_k$ is the following: there are $k$ urns and each urn can hold one or zero balls.  A binary sequence $\undeta$ represents a situation in which the $j$-th urn is filled if $\undeta_j=1$, empty if $\undeta_j=0$.  The differential is a signed sum over arrangements obtained by adding a ball to one of the empty urns of $\undeta$.

It is easy to see that $Y_k$ is always contractible if $k\geq 1$. By convention, we assume that $Y_0$ is a single copy of $\Bbbk$ sitting in homological degree zero.

\subsubsection{Lima partitions}\label{subsubsec-lima}

Identify a partition $\lambda=(\lambda_1\geq\lambda_2\geq\ldots\geq0)$ with its corresponding Young diagram.  A box of $\lambda$ is called \emph{removable} if removing this box from $\lambda$ yields a Young diagram.  If adding a box at a position next to $\lambda$ results in a Young diagram, that position (or that box) is called \emph{addable}.

A \emph{Lima partition} is a partition $\lambda$ in which all parts $\lambda_i$ are even and any particular value $\lambda_i$ occurs an even number of times.  In other words, a Lima partition is a partition whose Young diagram can be subdivided into $2\times2$ squares of boxes.  The following Young diagrams correspond to Lima partitions:
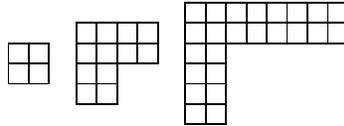
\begin{equation*}
\tiny\ytableaushort{\none,\none} * {2,2}
\qquad
\ytableaushort{\none,\none,\none,\none} * {4,4,2,2}
\qquad
\ytableaushort{\none,\none,\none,\none} * {8,8,2,2,2,2} ~
\end{equation*}
and the following do not:
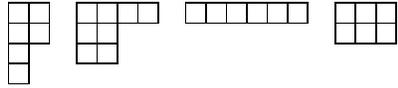
\begin{equation*}
\tiny\ytableaushort{\none,\none} * {2,2,1,1}
\qquad
\ytableaushort{\none,\none,\none,\none} * {4,2,2}
\qquad
\ytableaushort{\none,\none,\none,\none} * {6}
\qquad
\ytableaushort{\none,\none,\none,\none} * {3,3} ~.
\end{equation*}

It will be helpful to think of $\lambda$ as shaded like a checkerboard, for example:
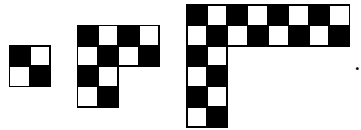
\begin{equation*}
\tiny\ytableaushort{\none,\none} * {2,2} *[*(black)] {0+1,1+1}
\qquad
\ytableaushort{\none,\none,\none,\none} * {4,4,2,2} *[*(black)] {0+1,1+1,0+1,1+1} * [*(black)] {2+1,3+1}
\qquad
\ytableaushort{\none,\none,\none,\none} * {8,8,2,2,2,2} *[*(black)] {0+1,1+1,0+1,1+1,0+1,1+1} * [*(black)] {2+1,3+1} * [*(black)] {4+1,5+1} * [*(black)] {6+1,7+1} ~.
\end{equation*}
White boxes have odd content and black boxes have even content (content equals column number less row number; see Subsection \ref{subsec-schur}).

\subsubsection{Cohomology}\label{subsubsec-lima-cohomology}

The following lemma is obvious.
\begin{lem} Adding or removing a white (respectively, black) box to a partition $\lambda$ results in no new addable white (respectively, black) boxes.\hfill$\square$\end{lem}
The formula of Proposition \ref{prop-d-schur} states that $d(s_\lambda)$ is a $\lb\pm1\rb$-linear combination of certain terms $s_\mu$.  The $\mu$ which appear are those obtained from $\lambda$ by adding one white box.  Therefore each partition can be thought of in terms of the balls and urns as above: an addable white box is an empty urn and a removable white box is an urn containing a ball.  Let $\osym$ be the limit of the dg algebras $\osym_n$ with respect to maps $\osym_{n+1}\to\osym_n$ which send $x_{n+1}\mapsto0$ and fix the other $x_i$'s.

\begin{prop}\label{prop-lima-cohomology}\begin{enumerate}
  \item The cohomology ring $\mH(\osym)$ is a free abelian group with basis
  \begin{equation*}
  B_\osym=\lb s_\lambda:\lambda\text{ is a Lima partition}\rb.
  \end{equation*}
  \item The cohomology ring $\mH(\osym_n)$ is a free abelian group with basis
  \begin{equation*}
  B_\osym=\lb s_\lambda:\lambda\text{ is a Lima partition and }\ell(\lambda)\leq n\rb.
  \end{equation*}
\end{enumerate}
\end{prop}
\begin{proof} By the urns-and-balls interpretation above, $\osym$ is a direct sum of hypercube complexes.  The odd Schur functions appearing in one indecomposable summand are exactly those obtainable from any one of them by adding and removing black boxes.  Therefore the cohomology $\mH(\osym)$ is spanned by the classes of those $s_\lambda$ for which $\lambda$ has no addable and no removable white boxes.  It is easy to see that this set contains the set of Lima partitions; we claim these two sets are equal.

Suppose $\lambda$ has no addable and no removable white boxes.  If $\lambda_1$ is odd, then a white box can be added in the first row, so $\lambda_1$ must be even.  The second row must be nonempty or else there would be an addable white box just below the top-left box.  Since the rightmost box of the first row is white and $\lambda$ has no removable white boxes, we must have $\lambda_1=\lambda_2$.  Thus the first two rows are of equal, even length.  The first claim now follows by induction, since $\lambda=(\lambda_1\geq\lambda_2\geq\ldots\geq\lambda_r)$ is Lima if and only if $\frac{2}{\lambda}=(\lambda_3\geq\lambda_4\geq\ldots\geq\lambda_r)$ is Lima.

The second claim follows from the first and an easy check of the boundary cases.
\end{proof}

The first few hypercube subcomplexes of $\osym$ are (drawing a partition $\lambda$ to stand for $s_\lambda$):
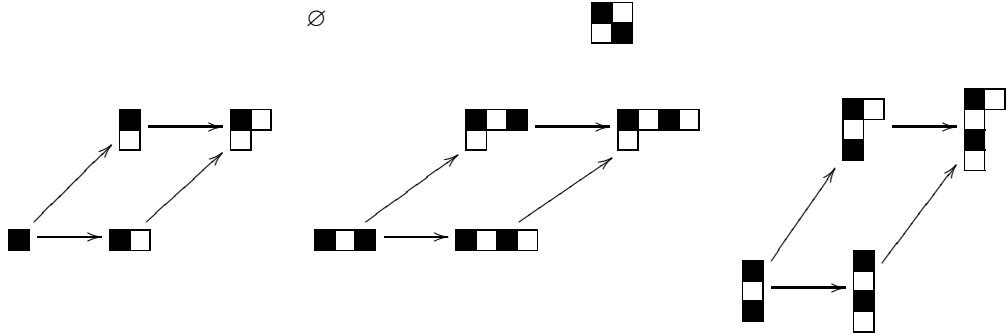
\begin{equation*}\begin{split}
&\varnothing\qquad\qquad\qquad\qquad\qquad\tiny\ytableaushort{\none,\none}*{2,2}*[*(black)]{1,1+1}\\
\\
\xymatrix{
        &\tiny\ytableaushort{\none,\none}*{1,1}*[*(black)]{1}\ar[r]
        &\tiny\ytableaushort{\none,\none}*{2,1}*[*(black)]{1}   \\
        \tiny\ytableaushort{\none}*[*(black)]{1}\ar[r]\ar[ur]
        &\tiny\ytableaushort{\none}*{2}*[*(black)]{1}\ar[ur]
    }\quad
&\xymatrix{
        &\tiny\ytableaushort{\none,\none}*{3,1}*[*(black)]{1}*[*(black)]{2+1}\ar[r]
        &\tiny\ytableaushort{\none,\none}*{4,1}*[*(black)]{1}*[*(black)]{2+1}   \\
        \tiny\ytableaushort{\none}*{3}*[*(black)]{1}*[*(black)]{2+1}\ar[r]\ar[ur]
        &\tiny\ytableaushort{\none}*{4}*[*(black)]{1}*[*(black)]{2+1}\ar[ur]
    }\quad
\xymatrix{
        &\tiny\ytableaushort{\none,\none,\none}*{2,1,1}*[*(black)]{1,0,1}\ar[r]
        &\tiny\ytableaushort{\none,\none,\none,\none}*{2,1,1,1}*[*(black)]{1,0,1}   \\
        \tiny\ytableaushort{\none,\none,\none}*{1,1,1}*[*(black)]{1,0,1}\ar[r]\ar[ur]
        &\tiny\ytableaushort{\none,\none,\none,\none}*{1,1,1,1}*[*(black)]{1,0,1}\ar[ur]
    }
\end{split}\end{equation*}

Over a field of characteristic $2$, the cohomology $\mH(\sym)$ is a polynomial algebra on the generators $1,e_2^2,e_4^2,e_6^2,\ldots$ \cite[Proposition 3.8]{EliasQi}.  The na\"{i}ve odd analogue of this statement is not true.  For instance
\begin{equation*}
d(e_2^2)=-2e_4e_1+2e_5\neq0,
\end{equation*}
so $e_2^2$ is not even a cocycle.  The correct odd analogue is the following.

\begin{prop}\label{prop-poly-alg} The cohomology ring $\mH(\osym)$ is a polynomial algebra on either of the sets $\lb s_{(2,\ldots,2)}:k\in\Z_{\geq0}\rb$, $\lb s_{(2k,2k)}:k\in\Z_{\geq0}\rb$.  In the former, there are $2k$ $2's$.\end{prop}
\begin{proof}
We will prove this for the second generating set; the other case is similar.  For Lima partitions, the odd Littlewood-Richardson coefficients of \cite[Theorem 4.8]{EOddLR} are equal to the usual (even) Littlewood-Richardson coefficients.  In particular, Lima Schur functions pairwise commute.

A routine application of the Littlewood-Richardson rule proves that all products of the form
$$s_{(2k_1,2k_1)}s_{(2k_2,2k_2)}\cdots s_{(2k_r,2k_r)}$$
are nonzero:  the key observation is that
\begin{equation}\label{eqn-Lima-product}
s_{(2k_1,2k_1)}s_{(2k_2,2k_2)}\cdots s_{(2k_r,2k_r)}=s_{(2k_1,2k_1,\ldots,2k_r,2k_r)}+\sum_{\mu>(2k_1,2k_1,\ldots,2k_r,2k_r)}a_\mu s_\mu
\end{equation}
for any $k_1\geq\cdots\geq k_r\geq0$ and certain non-negative integers $a_\mu$.

Thus the subalgebra generated by this generating set is a free abelian group of graded rank equal to the the graded rank of a polynomial on generators of degree $0,4,8,12,\ldots$.  This is also the graded rank of $\mH(\osym)$ itself.  The result now follows, noting that the coefficient of $s_{(2k_1,2k_1,\ldots,2k_r,2k_r)}$ in the left-hand side product is always $1$.
\end{proof}

\subsection{The filtration on $Z_n$}\label{subsec-gen-zoid-app}

The isomorphism $Z_n\cong U_n\otimes_\Bbbk\osym_n$ of Subsection \ref{subsec-cat-small} determines a filtration on the right dg $\osym_n$-module $Z_n$ with subquotients isomorphic to (shifts of) the right regular representation.  In other words, this isomorphism shows $Z_n$ to be finite-cell.  The goal of this subsection is to understand the chain complex $U_n$ and hence the filtration on $Z_n$.

Recall $U_n$ has $\Bbbk$-basis
\begin{equation*}
B'_n=\lbrace x^a\onez:0\leq i\leq i-1\rbrace
\end{equation*}
and differential given by \eqref{eqn-d-u},
\begin{equation*}
d(x^a\onez)=\sum_i\lb a_i+i-1\rb x_ix^a\onez.
\end{equation*}
For small $n$, we can draw $U_n$ (an arrow means a nonzero term of $d$, always with coefficient $\pm1$):
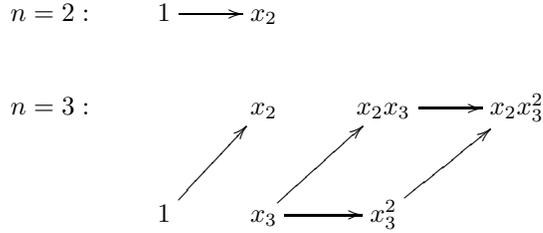
\begin{equation*}\begin{split}
&n=2:\qquad\xymatrix{1\ar[r]&x_2}\\
\\
&n=3:\qquad\xymatrix{&x_2&x_2x_3\ar[r]&x_2x_3^2\\1\ar[ur]&x_3\ar[r]\ar[ur]&x_3^2\ar[ur]}
\end{split}\end{equation*}
Already for $n=2,3$ the features of the general pattern can be seen.  $U_n$ is a direct sum of subcomplexes, each isomorphic to (a shift of) a hypercube complex.

In the monomial basis element $x^a\onez$, an empty urn is a factor $x_i^{a_i}$ with $i\equiv a_i$ mod $2$ and an urn with a ball in it is a factor $x_i^{a_i}$ with $i\equiv a_i+1$ mod $2$.  Adding a ball to this urn increases $a_i$ by $1$ and leaves the other $a_j$ unchanged.

For $n=5$, the set of initial vectors in $U_n$ is
\begin{equation*}
\lb x_3^ax_4^bx_5^c:a\in\lb0,1\rb,b\in\lb0,2\rb,c\in\lb0,1,3\rb\rb.
\end{equation*}
For $n=6$, take all products consisting of one element from the previous and one element from $\lb x_6^0,x_6^2,x_6^4\rb$.  And so forth for all $n$.

\subsection{The filtration on $\osym_a\otimes_\Bbbk\osym_b$}\label{apx-cohomology-thick}

Next we will analyze the structure of the chain complex $V_{a,b}$ defined in Subsection \ref{subsec-cat-big} in order to understand the filtration on $\osym_a\otimes_\Bbbk\osym_b$ by shifts of the right regular $\osym_{a+b}$ dg module.  As proved there, $V_{a,b}$ has basis
\begin{equation*}
B''_{a,b}=\lb s_\lambda\onez:\lambda\in\Par(b,a)\rb
\end{equation*}
and differential given by
\begin{equation*}
d(s_\lambda\onez)=\sum_{\mu=\lambda+\boxi}(-1)^{\lvert\frac{\lambda}{i}\rvert+i-1}\lb a+\ct(\boxi)\rb s_\mu\onez.
\end{equation*}
For $a\equiv0~(\mathrm{mod}~2)$, this is the same differential as on $\osym$.  So by the results of Subsection \ref{subsec-cohomology-osym}, it follows that $V_{a,b}$ is a direct sum of the same hypercube complexes occurring in the dg algebra $\oh_{a,b}$.  The cohomology of this chain complex is the $\Bbbk$-span of the classes of Lima Schur functions.

For $ a\equiv1~(\mathrm{mod}~2)$, the roles of black and white boxes are reversed.  It is not hard to prove directly that the resulting cohomology group is $0$: this is equivalent to the statement that no partition has neither addable nor removable white boxes.  For $a,b$ large and $a$ odd, the first few hypercube subcomplexes are (as above, $\lambda$ stands for $s_\lambda$):
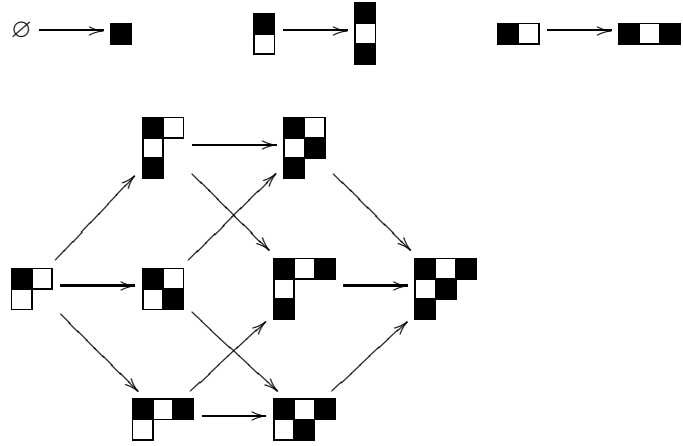
\begin{equation*}\begin{split}
&\xymatrix{
    \varnothing\ar[r]
    &\tiny\ytableaushort{\none}*[*(black)]{1}
}\qquad\qquad
\xymatrix{
    \tiny\ytableaushort{\none,\none}*{1,1}*[*(black)]{1}\ar[r]
    &\tiny\ytableaushort{\none,\none,\none}*{1,1,1}*[*(black)]{1,0,1}
}\qquad\qquad
\xymatrix{
    \tiny\ytableaushort{\none}*{2}*[*(black)]{1}\ar[r]
    &\tiny\ytableaushort{\none}*{3}*[*(black)]{1}*[*(black)]{2+1}
}\\
\\
&\xymatrix{
    &\tiny\ytableaushort{\none,\none,\none}*{2,1,1}*[*(black)]{1,0,1}\ar[r]\ar[dr]
    &\tiny\ytableaushort{\none,\none,\none}*{2,2,1}*[*(black)]{1,1+1,1}\ar[dr]  \\
    \tiny\ytableaushort{\none,\none}*{2,1}*[*(black)]{1}\ar[ur]\ar[r]\ar[dr]
    &\tiny\ytableaushort{\none,\none}*{2,2}*[*(black)]{1,1+1}\ar[ur]\ar[dr]
    &\tiny\ytableaushort{\none,\none,\none}*{3,1,1}*[*(black)]{1}*[*(black)]{2+1,0,1}\ar[r]
    &\tiny\ytableaushort{\none,\none,\none}*{3,2,1}*[*(black)]{1}*[*(black)]{2+1,1+1,1} \\
    &\tiny\ytableaushort{\none,\none}*{3,1}*[*(black)]{1}*[*(black)]{2+1}\ar[r]\ar[ur]
    &\tiny\ytableaushort{\none,\none}*{3,2}*[*(black)]{1}*[*(black)]{2+1,1+1}\ar[ur]
}
\end{split}\end{equation*}

\subsection{Slash cohomology of $p$-dg symmetric functions}

For this final subsection we work with $p$-dg algebras over a field $\Bbbk$ of characteristic $p>0$.  For an introduction to the theory of $p$-dg algebras see \cite[Section 2]{KQ}.  We will give a simple combinatorial proof of a result of Ben Elias and the second author \cite[Proposition 3.8]{EliasQi}.  \emph{En route}, we will give a basis for the slash cohomology of $\sym$ in terms of Schur functions.

\subsubsection{$p$-complexes and slash cohomology}

Define the graded algebra $\mathbf{H}=\Bbbk[d]/(d^p)$, $\deg(d)=2$.  It is a finite-dimensional graded Hopf algebra when given the coproduct
\begin{equation}
\Delta(d)=d\otimes1+1\otimes d.
\end{equation}
Graded modules over $\mathbf{H}$ are called \emph{$p$-complexes}.  In other words, a $p$-complex is a graded $\Bbbk$-module equipped with a degree-$2$ differential $d$ which satisfies $d^p=0$.

Up to isomorphism and grading shift, the indecomposable $H$-modules are
\begin{equation*}
\lb V_i=\mathbf{H}/(d^{i+1}):i=0,1,\ldots,p-1\rb.
\end{equation*}
For $p=2$, these are the chain complexes
\begin{equation*}
V_0=\left(\xymatrix{0\ar[r]&\Bbbk\ar[r]&0}\right)\qquad\text{and}\qquad
V_1=\left(\xymatrix{0\ar[r]&\Bbbk\ar[r]^-1&\Bbbk\ar[r]&0}\right).
\end{equation*}
In the classical setting of chain complexes, ``taking cohomology'' simply kills all summands isomorphic to (shifts of) $V_1$.  One analogue of this in the setting of $p$-complexes is \emph{slash cohomology}; for a $p$-complex $V$ and $k\in\lb0,1,\ldots,p-2\rb$, this is
\begin{equation}
\mH_{/k}(V)=\mathrm{Ker}(d^k)/\left(\mathrm{Im}(d^{p-k-1})+\mathrm{Ker}(d^{k+1})\right).
\end{equation}
When $p=2$, the only slash cohomology group is just usual cohomology, $\mH=\mH_{/0}$.  Each group $\mH_{/k}$ inherits a $\Z$-grading from the Hopf algebra grading on $\mathbf{H}$.

It follows from the results of \cite[Section 2]{KQ} that
\begin{equation}\label{eqn-slash-V}
\mH_{/k}(V_i)=
\begin{cases}
\text{span}_\Bbbk(v_{i-k})  &   0\leq k\leq i\text{ and }i<p-1, \\
0&  i+1\leq k\leq p-2\text{ or }i=p-1.
\end{cases}
\end{equation}
The K\"{u}nneth-type formula
\begin{equation*}
\mH_/(V\otimes_\Bbbk W)\cong \mH_/(V)\otimes_\Bbbk \mH_/(W),
\end{equation*}
does not necessarily hold on the nose: the tensor product among $V_i$'s may produce acyclic summands. For instance $V_{p-2}\otimes V_{p-2}$ is only isomorphic to $V_0\{2p-4\}$ after removing a $(p-2)$ graded copies of $V_{p-1}$. Instead, one has the following weaker version of the K\"{u}nneth property. If $V$ is  a $p$-complex that satisfies
$
\mH_{/0}(V)\cong \mH_{/}(V)
$
and $W$ is an arbitrary $p$-complex, then
\[
\mH_{/}(V\otimes W)\cong \mH_{/0}(V)\otimes_\Bbbk \mH_{/}(W).
\]

So just as the cohomology of a chain complex detects summands $V_0$, the $k$-th slash cohomology of a $p$-complex detects summands $V_i$ with $k\leq i\leq p-2$.  Acyclic summands $V_{p-1}$ contribute nothing to $\mH_{/k}$.

Tensor powers of $V_i$ admit a combinatorial description which generalizes the ``balls and urns'' description of the hypercube complexes $Y_k$ of Subsection \ref{subsec-cohomology-osym}.  The $p$-complex $V_i^{\otimes\ell}$ has a basis
\begin{equation*}
\lb v_{\undvarepsilon}:\undvarepsilon\in\lb 0,1,\ldots,i\rb^\ell\rb,
\end{equation*}
where $v_{\undvarepsilon}=d^{\undvarepsilon_1}\otimes\cdots\otimes d^{\undvarepsilon_\ell}$.  The differential acts on this basis as
\begin{equation*}
d(v_{\undvarepsilon})=\sum_{|\undeta|=1+|\undvarepsilon|}v_{\undeta}.
\end{equation*}
We can think of each tensor-factor of $V_i$ as an urn which can hold up to $i$ balls; the differential is a sum over ways to add a ball to an urn.  The coefficient of $v_{\undeta}$ in $d^j(v_{\undvarepsilon})$ is the number of sequences of ball-additions to $\undvarepsilon$ which yield $\undeta$.  For $p=3$, $\ell=2$, and $i=2$, the picture is:
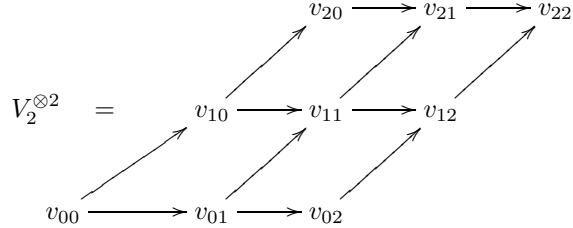
\begin{equation*}
\xymatrix{&&v_{20}\ar[r]&v_{21}\ar[r]&v_{22}\\
V_2^{\otimes2}\quad=&v_{10}\ar[r]\ar[ur]&v_{11}\ar[r]\ar[ur]&v_{12}\ar[ur]\\
v_{00}\ar[r]\ar[ur]&v_{01}\ar[r]\ar[ur]&v_{02}\ar[ur]}
\end{equation*}

\subsubsection{$p$-Lima Schur functions}

Define a $p$-differential $d$ on $f\in\Bbbk[x_1,\ldots,x_n]$ by
\begin{equation*}
d(x_i)=x_i^2,\qquad d(fg)=d(f)g+fd(g).
\end{equation*}
So $\Bbbk[x_1,\ldots,x_n]$ is a $p$-dg algebra (that is, an $\mathbf{H}$-module-algebra).  The subring $\sym_n$ of symmetric polynomials is closed under $d$.  The map $d$ is compatible with the inverse system
\begin{equation*}\begin{split}
\sym_{n+1}&\to\sym_n\\
x_i&\mapsto x_i\text{ if }1\leq i\leq n\\
x_{n+1}&\mapsto0,
\end{split}\end{equation*}
so there is an induced $p$-dg algebra structure on the limit $\sym$.  This differential acts on elementary, complete, and Schur functions as in \cite{EliasQi}:
\begin{equation}
d(e_k)=e_1e_k-e_{k+1},
\end{equation}
\begin{equation}
d(h_k)=h_{k+1}-h_1h_k,
\end{equation}
\begin{equation}
d(s_\lambda)=\sum_{\mu=\lambda+B}\ct(B)s_\mu.
\end{equation}
The summation is over all partitions $\mu$ obtained by adding one box $B$ to $\lambda$ and $\ct(B)$ is the content of the new box (column number less row number).

We saw in \ref{subsubsec-lima-cohomology} that the action of the differential on $\osym$ (in characteristic zero) can be interpreted in terms of urns and balls.  Considering content modulo $2$, an addable box of content $\onebar$ is an urn with $0$ balls and a removable box of content $\zerobar$ is an urn with $1$ ball.  The situation of $\sym$ in characteristic $p>0$ is similar.

Let $\ctbar(B)\in\Z/p\Z$ be the mod $p$ reduction of the content of a box $B$.  An addable box $B$ with $\ctbar(B)=\bar{i}$ corresponds to an urn with $i-1$ balls in it (taking the representative $i=p$ in place of $i=0$ for $\overline{i}=0$).  Any urn can hold up to $p-1$ balls.  It is straightforward to verify, then, that $\sym$ is a direct sum of $p$-complexes isomorphic to (shifts of) $V_{p-1}^{\otimes\ell}$ for various $\ell\geq0$.

Since $V_{p-1}$ is acyclic, the only summands which are homologically nontrivial are those for which $\ell=0$.  These correspond to partitions all of whose addable and removable boxes have content $\bar{p}$.  These are exactly those partitions built out of $p\times p$ squares of boxes; we call these \emph{$p$-Lima partitions}.  We have proven the first part of the following.

\begin{thm}
\begin{enumerate}
\item Let $B_{\sym,p}=\lb s_\lambda:\lambda\text{ is }p\text{-Lima}\rb$.  Then $\mH_{/k}(\sym)=0$ for $k>0$ and $\mH_{/0}(\sym)=\text{span}_{\Bbbk}(B_{\sym,p})$.
\item The inclusion $\Bbbk[e_p^p,e_{2p}^p,e_{3p}^p,\ldots]\hookrightarrow\sym$ is a quasi-isomorphism of $p$-dg algebras.
\end{enumerate}
\end{thm}
\begin{proof}
The first part was proven above.  The second follows from an argument similar to that in the proof of Proposition \ref{prop-poly-alg}.
\end{proof}

\addcontentsline{toc}{section}{References}


\bibliographystyle{alpha}
\bibliography{ellis-bib}

%

\vspace{0.1in}

\noindent A.P.E.: { \sl \small Department of Mathematics, University of Oregon, Eugene, OR 97403, USA} \newline \noindent {\tt \small email: ellis@uoregon.edu}

\vspace{0.1in}

\noindent Y.Q.: { \sl \small Department of Mathematics, Yale University, New Haven, CT 06511, USA} \newline \noindent {\tt \small email: you.qi@yale.edu}

%
\end{document}